\documentclass[10pt,reqno,final]{article}

\usepackage{amsmath,amsfonts,amssymb,amsthm,version}
\usepackage{mathrsfs,fancybox,pifont}
\usepackage{graphicx}
\usepackage{url}
\usepackage{hyperref}
\usepackage{bookmark}
\usepackage{color}
\usepackage{xcolor}
\usepackage{subfigure,multirow}
\usepackage{epstopdf}
\usepackage{cases}
\usepackage{mathtools}
\usepackage{algorithm,algorithmic}
\usepackage{authblk}
\usepackage{array} 
\usepackage{bm}
\usepackage{tikz-cd}
\usepackage{booktabs}
\usetikzlibrary{arrows.meta}
\usepackage{xspace}

\newcommand{\Aut}{\mathrm{Aut}}
\newcommand{\End}{\mathrm{End}}
\newcommand{\Hom}{\mathrm{Hom}}
\newcommand{\Res}{\mathrm{Res}}
\newcommand{\PSL}{\mathrm{PSL}}
\newcommand{\SL}{\mathrm{SL}}

\newcommand{\Q}{\mathbb{Q}}
\newcommand{\R}{\mathbb{R}}

\newcommand{\Rep}{\mathbf{Rep}}

\newcommand{\Loc}{\mathbf{Loc}}
\newcommand{\Higgs}{\mathbf{Higgs}}

\newcommand{\GL}{\mathrm{GL}}

\newcommand{\Lie}{\mathrm{Lie}}

\newcommand{\MIC}{\mathbf{MIC}}
\newcommand{\PU}{\mathrm{PU}}

\newcommand{\et}{\text{\'et}}
\newcommand{\Id}{\mathrm{Id}}

\allowdisplaybreaks

\numberwithin{equation}{section}
\numberwithin{figure}{section}
\numberwithin{table}{section}

\theoremstyle{plain}
\newtheorem{theorem}{Theorem}[section]

\newtheorem{corollary}{Corollary}[section]
\newtheorem{proposition}{Proposition}[section]

\theoremstyle{definition}
\newtheorem{definition}{Definition}[section]
\newtheorem{example}{Example}

\theoremstyle{remark}
\newtheorem{remark}{Remark}[section]

\title{Uniformization as Tannakian Reconstruction}
\author[1]{Xiaojin Lin\thanks{xj-lin@mail.tsinghua.edu.cn}}
\author[2]{Mao Sheng\thanks{msheng@mail.tsinghua.edu.cn}}
\affil[1,2]{Yau Mathematical Sciences Center, Tsinghua University}
\affil[2]{Beijing Institute of Mathematical Sciences and Applications}

\date{}

\begin{document}
\maketitle
 
\begin{abstract}
    Classical hyperbolic uniformization identifies every hyperbolic log--orbi
    curve with a compactified quotient of the upper half-plane by a cofinite
    Fuchsian lattice, unique up to conjugacy.  We reconstruct this lattice
    intrinsically.  For each hyperbolic log--orbi curve \(\mathcal C\), we
    construct a canonical maximal principal \(\PSL_2\)-Higgs object
    \((\mathcal U_{\mathcal C},\vartheta_{\mathcal C})\).  \'Etale-locally it
    comes from the standard square-root \(\SL_2\)-model; the central
    \(\mu_2\)-ambiguity disappears after passage to \(\PSL_2\).

    Using vector tame non-abelian Hodge theory and regular-singular
    Riemann--Hilbert as input, we assemble the required principal realizations
    Tannakianly, with parahoric structures encoding the orbifold and cusp data
    on the coarse curve.  After choosing \(c\in C^\circ\) and conjugating, the
    Betti realization of the canonical object is represented by a discrete,
    faithful, finite-covolume representation
    \[
    \rho_{\mathcal C}\colon
    \pi_1^{\mathrm{orb}}(\mathcal C,c)\longrightarrow \PSL_2(\mathbb R),
    \]
    whose image is the uniformizing lattice.  Compatibility with finite
    \'etale pullback makes the lattice construction a quasi-inverse to the
    compactified quotient functor.  Thus classical uniformization is recast as
    an intrinsic Tannakian reconstruction theorem.  We also identify finite
    \'etale covers with finite continuous sets for the profinite completion of
    the reconstructed lattice and, after fixing a separable closure and the
    resulting geometric generic point, recover the absolute Galois group of
    \(\mathbb C(C)\) as the inverse limit of the based \'etale fundamental
    groups of orbifold models over \(C\).

    \medskip
    \noindent\textbf{Keywords}: uniformization, non-abelian Hodge theory, orbifolds, Galois theory, Fuchsian groups, Tannakian reconstruction.

    \medskip
    \noindent\textbf{2020 Mathematics Subject Classification:} 14A20, 14A21, 14C30, 14D07, 14F30, 14H30 (Primary), 14H57.
\end{abstract}


\section{Introduction}

\subsection{Uniformization as Tannakian Reconstruction}
\label{subsec:intro-tannakian-reconstruction}

    Classical hyperbolic uniformization identifies a hyperbolic log--orbi curve
    with a compactified quotient
    \[
    \mathcal C\simeq(\Gamma\backslash\mathbb H)^c
    \]
    by a cofinite Fuchsian lattice
    \(\Gamma\subset\PSL_2(\mathbb R)\), unique up to conjugacy.  In its usual
    form this is an existence-and-uniqueness theorem.  The problem considered
    here is more specific: can the uniformizing representation and its image
    lattice be reconstructed from an intrinsic geometric object on
    \(\mathcal C\), compatibly with finite \'etale morphisms?

    The natural Hodge-theoretic candidate exposes an obstruction.  Locally,
    the standard maximal rank-two Higgs model is built on
    \(\Theta\oplus\Theta^{-1}\), where
    \(\Theta^{\otimes2}\simeq\omega_{\mathcal C}\).  As an \(\SL_2\)-object it
    depends on a square root of the log--orbi canonical bundle.  Such a square
    root exists only for certain orbifold structures; in particular, all
    inertia groups must have odd order.  Even when such a square root exists,
    it is not canonical: different choices differ by central
    \(\mu_2\)-twists.  Thus the evident rank-two construction introduces
    auxiliary spin data precisely where an intrinsic reconstruction should be
    choice-free.

    The obstruction also identifies the correct group.  After pushout along
    \(\SL_2\to\PSL_2\), the central twists become invisible, so the local
    maximal models descend to a canonical principal \(\PSL_2\)-Higgs object
    \((\mathcal U_{\mathcal C},\vartheta_{\mathcal C})\).  Its projective form
    is forced by descent rather than chosen for convenience.

    We apply vector tame non-abelian Hodge theory and regular-singular
    Riemann--Hilbert representation by representation, then reconstruct the
    principal realizations Tannakianly.  Parahoric structures provide only the
    auxiliary coarse-curve description of the full orbifold and logarithmic
    local data.  Applying the parahoric non-abelian Hodge package to
    \((\mathcal U_{\mathcal C},\vartheta_{\mathcal C})\) and then invoking
    maximality yields the desired uniformizing Fuchsian lattice.

    To formulate the categorical payoff, let \(\mathbf{FL}\) be the category
    of cofinite Fuchsian lattices with transporter-coset morphisms and let
    \(\mathbf{HypLO}\) be the category of hyperbolic log--orbi curves with
    finite \'etale morphisms.  Compactified quotient defines
    \[
    Q_c\colon\mathbf{FL}\longrightarrow\mathbf{HypLO},
    \qquad
    \Gamma\longmapsto(\Gamma\backslash\mathbb H)^c.
    \]

    \begin{theorem}[Poincaré--Koebe hyperbolic uniformization, categorical form]
    \label{thm:classical-uniformization-essential-surjectivity}
    The functor \(Q_c\) is essentially surjective.  Equivalently, every
    hyperbolic log--orbi curve is the compactified quotient by a cofinite
    Fuchsian lattice, unique up to conjugacy; orbifold points correspond to
    elliptic fixed points of the same order and logarithmic points to cusps
    \cite[\S\S1--2, pp.~223--228]{Faltings83RealprojstronRiemannsur}.
    \end{theorem}

    The equivalence of the underlying covering categories is classical once
    uniformization and covering theory are invoked.  The contribution
    established here is the
    intrinsic construction of the inverse: the uniformizing lattice is
    recovered as the Betti realization of a canonical principal
    \(\PSL_2\)-object, and this reconstruction is compatible with finite
    \'etale pullback.

    Uniformizing representations also have classical descriptions through
    projective structures, special coordinates, differential equations, and
    Higgs bundles
    \cite{DanielMichaelAlbert00ThemonogposSchwaequaonclosedRiesur,
    Gunning1967Specialcoordinatecoverings,Hitchin87TheselfdualequonRiemansurface,
    Simpson88constrctingVHSusingYMandapptouniformization}.  The analytic
    identification below uses the maximal-Higgs period-map approach.  Faltings's
    related projective construction treats marked compact Riemann surfaces
    with elliptic and cuspidal data under the additional hypothesis that the
    uniformizing \(\PSL_2(\mathbb R)\)-group lifts to \(\SL_2(\mathbb R)\)
    \cite[\S1, p.~224; Theorem~1, p.~228]
    {Faltings83RealprojstronRiemannsur}.  The construction here is intrinsically
    projective, requires no global theta characteristic or chosen
    \(\SL_2\)-lift, and is formulated on every object of \(\mathbf{HypLO}\).

    \subsection{Main Results}
    \label{subsec:intro-main-results}

    \begin{theorem}[Categorical uniformization]
    \label{thm:main-uniformization-intro}
    The compactified quotient functor
    \[
    Q_c:\mathbf{FL}\longrightarrow \mathbf{HypLO},
    \qquad
    \Gamma\longmapsto (\Gamma\backslash\mathbb H)^c,
    \]
    is an equivalence of categories.  
    After choosing a uniformizing identification and a base point for each
    object, a
    quasi-inverse is given by
    \[
    U:\mathbf{HypLO}\longrightarrow \mathbf{FL},
    \qquad
    \mathcal C\longmapsto
    \Gamma_{\mathcal C}:=
    \rho_{\mathcal C}\bigl(\pi_1^{\mathrm{orb}}(\mathcal C,c)\bigr)
    \subset \PSL_2(\mathbb R).
    \]
    Here \(\rho_{\mathcal C}\) is a representative of the Betti realization
    of the canonical maximal \(\PSL_2\)-object attached to \(\mathcal C\).
    Different choices give naturally isomorphic quasi-inverses.  If
    \[
    f\colon \mathcal C_1\longrightarrow \mathcal C_2
    \]
    is finite \'etale, then
    \[
    \deg(f)=
    [\Gamma_{\mathcal C_2}:g\Gamma_{\mathcal C_1}g^{-1}]
    \]
    for any representative \(g\) of the transporter-coset morphism
    \(U(f)\) defined in Definition~\ref{def:FL}.
    \end{theorem}

    The fundamental geometric input driving this reconstruction is the following distinguished Dolbeault object.

    \begin{theorem}[Canonical maximal \(\PSL_2\)-Higgs object]
    \label{thm:canonical-psl2-object-intro}
    Let \(\mathcal C\) be a hyperbolic log--orbi curve. 
    There exists a canonical principal \(\PSL_2\)-Higgs object \( (\mathcal U_{\mathcal C},\vartheta_{\mathcal C}) \) on \(\mathcal C\), characterized \'etale-locally as the pushout along \( \SL_2\longrightarrow \PSL_2 \) of the rank-two model
    \[
    \left(
    \Theta\oplus\Theta^{-1},
    \begin{pmatrix}
    0&0\\
    1&0
    \end{pmatrix}
    \right),
    \qquad
    \Theta^{\otimes2}\simeq \omega_{\mathcal C}.
    \]
    This global object is independent of all local choices of the square root \(\Theta\). Its local \(\PSL_2\)-parahoric type is canonically given by
    \[
    \theta_{\mathcal C,x}=\kappa_x\varpi^\vee,
    \qquad
    \kappa_x=
    \begin{cases}
    1-\frac{1}{m_x}, & x\in D_{\mathrm{orb}},\\[2mm]
    1, & x\in D_{\log}.
    \end{cases}.
    \]
    At an orbifold point of order \(m_x\), the denominator of the fractional
    local type is precisely \(m_x\); the object is regular on the stack chart,
    although its coarse strongly parabolic residue is nonzero.  At a
    logarithmic point the fractional \(\PSL_2\)-type is trivial, while its
    nilpotent residue in the strictly positive Moy--Prasad piece is nonzero.
    Moreover, every finite \'etale map
    \(f:\mathcal C_1\to\mathcal C_2\) carries a canonical coherent isomorphism
    \[
    f^*(\mathcal U_{\mathcal C_2},\vartheta_{\mathcal C_2})
    \simeq
    (\mathcal U_{\mathcal C_1},\vartheta_{\mathcal C_1}).
    \]
    \end{theorem}

    The central pushout \(\SL_2\to\PSL_2\) removes the square-root
    ambiguity; see Subsection~\ref{subsec:canonical-maximal-psl2}.
    Corollary~\ref{cor:log-orbi-parahoric-local-behavior} gives the auxiliary
    coarse-curve translation: orbifold algebraic type means regular
    equivariance and corresponds to strongly parabolic data whose coarse
    positive residue may be nonzero, while logarithmic points carry the full
    adjusted locus and may have nilpotent residue.  Together with the displayed
    cocharacter system, the computed residue determines the full canonical
    Dolbeault parameter \(\boldsymbol\theta_{\mathcal C}\).

    The realization theorem is obtained by combining this log--orbi/parahoric
    dictionary with Tannakian reconstruction.  By
    Proposition~\ref{prop:tannakian-parahoric-connections}, principal
    log--orbi \(G\)-\(\lambda\)-objects are recovered from their full
    tensor-compatible family of associated vector objects.  Therefore the
    vector-valued NAH/RH equivalences
    of Theorem~\ref{thm:vector-parahoric-nah-rh}, together with the principal
    reconstruction of
    Proposition~\ref{prop:tannakian-reconstruction-principal-realizations},
    give the tensor-functorial principal theorem below.

    \begin{theorem}[Tannakian log--orbi NAH/RH realization]
    \label{thm:nah-rh-log-orbi-intro}
    Let \(\mathcal C\) be a connected log--orbi curve, let \(G\) be a connected
    complex reductive group, and let
    \[
    \boldsymbol\eta
    =(\boldsymbol\theta,\boldsymbol\delta,\boldsymbol n)
    \]
    be a corresponding realization datum: \(\boldsymbol\theta\) and
    \(\boldsymbol\delta\) are admissible rational full local parameters,
    of orbifold algebraic type and logarithmic adjusted type, with zero
    transformed Betti weight, while \(\boldsymbol n\) records
    the corresponding local monodromy conjugacy data.  Then there are
    natural equivalences of groupoids
    \[
    \Higgs_G(\mathcal C,\boldsymbol\theta)_{\mathrm{poly},0}
    \xrightarrow{\ \mathrm{NAH}\ }
    \MIC_G(\mathcal C,\boldsymbol\delta)_{\mathrm{red}}
    \xrightarrow{\ \mathrm{RH}\ }
    \Rep_G^{\mathrm{red}}
    \bigl(\pi_1^{\mathrm{orb}}(\mathcal C),\boldsymbol n\bigr).
    \]
    On associated representations they are compatible with rigid tensor
    operations, and they are functorial under finite \'etale pullback.
    \end{theorem}

    The Betti term consists of ordinary reductive local systems with the
    prescribed full local monodromy datum and no additional filtration.
    Orbifold inertia is finite semisimple and may be nontrivial, while cusp
    monodromy may have a nontrivial unipotent factor.  For general \(G\), the
    full corresponding datum is input; only the distinguished \(\PSL_2\)-object
    carries the canonical datum constructed here.

    \begin{theorem}[Uniformizing representation]
    \label{thm:uniformizing-representation-intro}
    For a hyperbolic log--orbi curve \(\mathcal C\), the Betti realization of
    \[
    (\mathcal U_{\mathcal C},\vartheta_{\mathcal C})
    \in
    \Higgs_{\PSL_2}
    (\mathcal C,\boldsymbol\theta_{\mathcal C})_{\mathrm{poly},0}
    \]
    is an unbased \(\PSL_2\)-local system.  After choosing \(c\in C^\circ\), it
    is represented, up to conjugation, by a faithful homomorphism
    \[
    \rho_{\mathcal C}\colon
    \pi_1^{\mathrm{orb}}(\mathcal C,c)\longrightarrow \PSL_2(\mathbb C),
    \]
    whose image is a discrete finite-covolume subgroup of \(\PSL_2(\mathbb R)\).
    Hence
    \[
    \Gamma_{\mathcal C}:=
    \rho_{\mathcal C}\bigl(\pi_1^{\mathrm{orb}}(\mathcal C,c)\bigr)
    \subset \PSL_2(\mathbb R)
    \]
    is the uniformizing Fuchsian lattice of \(\mathcal C\).

    If
    \(
    f\colon \mathcal C_1\longrightarrow \mathcal C_2
    \)
    is a finite \'etale morphism of hyperbolic log--orbi curves, then, after
    choosing compatible base points and a connecting path,
    \[
    [\rho_{\mathcal C_1}]=[\rho_{\mathcal C_2}\circ f_*]
    \]
    in the \(\PSL_2(\mathbb R)\)-character variety.  Equivalently, \(f\)
    induces a finite-index inclusion
    \[
    g\Gamma_{\mathcal C_1}g^{-1}\hookrightarrow \Gamma_{\mathcal C_2},
    \]
    well-defined up to conjugation.
    \end{theorem}

    The mixed smooth, cone, and cusp period-map argument is given in
    Theorem~\ref{thm:uniformizing-representation}; Simpson supplies the
    compact and tame period-domain frameworks used there
    \cite[\S\S8--9]{Simpson88constrctingVHSusingYMandapptouniformization}
    \cite[Main Theorem, \S7, and Theorem~8]
    {Sim90harmobdonnoncomcurves}.

    \begin{theorem}[Galois enhancement]
    \label{thm:galois-category-compatibility-intro}
    For every hyperbolic log--orbi curve \(\mathcal C\) and a chosen
    uniformizing representative \(\Gamma_{\mathcal C}\), categorical
    uniformization induces an equivalence of Galois categories
    \[
    \mathbf{FEt}(\mathcal C)
    \simeq
    \mathbf{FSet}^{\mathrm{cont}}_{\widehat{\Gamma}_{\mathcal C}}.
    \]
    Equivalently, after choosing a geometric base point \(\bar x\to C^\circ\),
    a compatible analytic base point and lift, there is an isomorphism of
    profinite groups
    \[
    \pi_1^{\mathrm{\acute et}}(\mathcal C,\bar x)
    \simeq
    \widehat{\Gamma}_{\mathcal C}.
    \]
    Without these choices the group isomorphism is defined only up to inner
    automorphism, and \(\Gamma_{\mathcal C}\) only up to conjugacy.
    \end{theorem}

    The same finite-cover description yields the function-field application:
    ramification is resolved on suitable orbifold models over \(C\), and their
    based \'etale fundamental groups recover \(G_{\mathbb C(C)}\) by inverse
    limit.

    \begin{theorem}[Orbifold approximation of the absolute Galois group]
    \label{thm:GF-orbifold-limit-intro}
    Let \(C\) be a smooth projective connected curve over \(\mathbb C\), and
    let \(F=\mathbb C(C)\).  Fix a separable closure \(F^{\mathrm{sep}}/F\) and
    the resulting geometric generic point \(\bar\eta\).  Relative to these
    choices, there is a canonical isomorphism of profinite groups
    \[
    G_F:=\pi_1^{\mathrm{\acute et}}(\operatorname{Spec}F,\bar\eta)
    \xrightarrow{\sim}
    \varprojlim_{\mathcal X\in\operatorname{Orb}(C)}
    \pi_1^{\mathrm{\acute et}}(\mathcal X,\bar\eta),
    \]
    where the transition maps are induced by refinements.  Without a chosen
    geometric generic point, this identification is defined only up to inner
    automorphism.  Every finite extension of \(F\) is detected at a stage: it
    is represented by a finite \'etale cover of some
    \(\mathcal X\in\operatorname{Orb}(C)\).
    \end{theorem}

\subsection{Outline of the paper}
\label{subsec:intro-outline}

    Section~\ref{sec:log-orbi-fuchsian} fixes the geometric and Fuchsian
    categories.  Sections~\ref{sec:parahoric-canonical-psl2}
    and~\ref{sec:tannakian-nah-rh} construct the canonical log--orbi
    \(\PSL_2\)-object and its Tannakian realizations, using parahoric language
    only for coarse local data.  Section~\ref{sec:categorical-uniformization}
    proves uniformization and its Galois enhancement;
    Section~\ref{sec:examples-moduli} gives two illustrations; and
    Section~\ref{sec:function-field-galois} proves the orbifold approximation
    of the function-field Galois group.

\subsection{Conventions and notation}\label{subsec:intro-conventions}

    Throughout the paper we work over $\mathbb C$ and write $\mathbb H$ for the upper half-plane. The main conventions are:
    \[
    \renewcommand{\arraystretch}{1.25}
    \begin{array}{@{}>{\displaystyle}l@{\hspace{1.4em}}>{\raggedright\arraybackslash}p{0.58\textwidth}@{}}
    \mathcal C=(C,D_{\mathrm{orb}}^{\mathbf m},D_{\log})
    & a log--orbi curve: a smooth Deligne--Mumford curve with orbifold points and logarithmic marked points; all geometric objects are curves unless stated otherwise. \\

    \mathbf{HypLO},\ \mathbf{FL}
    & hyperbolic log--orbi curves and cofinite Fuchsian lattices in \(\PSL_2(\mathbb R)\). \\

    Q_c:\mathbf{FL}\to\mathbf{HypLO}
    & the compactified quotient functor. \\

    \mathbf{FEt}(\mathcal C)
    & finite \'etale covers of \(\mathcal C\). \\

    \pi_1^{\mathrm{orb}}(\mathcal C),\ \pi_1^{\et}(\mathcal C)
    & the discrete analytic orbifold group used on the Betti side, and the profinite \'etale group used for finite-cover and Galois-category arguments. \\

    \operatorname{Orb}(C)
    & the cofiltered refinement category over the fixed coarse curve \(C\). \\

    \theta_x\in X_*(T)_{\mathbb Q}/W
    & the cocharacter component of a rational Dolbeault or de Rham local parameter; Moy--Prasad filtrations use the unnormalized rational-cocharacter convention. \\

    \boldsymbol\eta=(\boldsymbol\theta,\boldsymbol\delta,\boldsymbol n)
    =((\theta_x,\delta_x,n_x))_{x\in D}
    & corresponding full Dolbeault/de Rham parameters and Betti local monodromy datum; their cocharacter components are marked by \({}^{\mathrm{coch}}\), and the Betti object has no additional filtration. \\

    \mathfrak g_{\ge a}(\theta_x),\ \mathfrak g_{>0}(\theta_x)
    & the associated Moy--Prasad pieces. \\

    \Higgs_G(\mathcal C,\boldsymbol\theta)
    & principal log--orbi \(G\)-Higgs objects with prescribed full Dolbeault parameter. \\

    \MIC_G(\mathcal C,\boldsymbol\delta)
    & principal log--orbi logarithmic \(G\)-connections with prescribed full de Rham parameter; their parahoric presentation is auxiliary. \\

    \Rep_G(\pi_1^{\mathrm{orb}}(\mathcal C),\boldsymbol n)
    & ordinary \(G\)-representations with local monodromy datum \(\boldsymbol n\). \\
    \end{array}
    \]

    \section{Log--orbi curves and Fuchsian lattices}\label{sec:log-orbi-fuchsian}

    In this section we fix the two categories entering categorical uniformization.  
    On the geometric side we define hyperbolic log--orbi curves and their finite \'etale covers; on the group-theoretic side we define cofinite Fuchsian lattices and the compactified quotient functor
    \[
    Q_c:\mathbf{FL}\longrightarrow \mathbf{HypLO}.
    \]
    Here and below, ``finite \'etale'' has the meaning fixed in
    Definition~\ref{def:etale-morphism-log-orbi}.

    \subsection{Galois categories and Grothendieck fundamental groups}

    We use Grothendieck's Galois-category formalism: a fiber functor
    \(F:\mathscr C\to\mathbf{FSet}\) determines
    \(\pi_1(\mathscr C,F)=\operatorname{Aut}(F)\) and an equivalence
    \(\mathscr C\simeq
    \mathbf{FSet}^{\mathrm{cont}}_{\pi_1(\mathscr C,F)}\)
    \cite[Exp.~V]{SGA11971}.  We apply it below to finite \'etale covers of a
    connected log--orbi curve.

\subsection{Log--orbi curves and finite \'etale covers}\label{subsec:log-orbi-curves}

    We now fix the geometric objects studied in this paper.

    \begin{definition}[Log--orbi curve]
    \label{def:log-orbi-curve}
    A \emph{log--orbi curve} over \(\mathbb C\) is a pair
    \(
    (\mathcal C,D_{\log})
    \), where \(\mathcal C\) is a smooth proper connected one-dimensional
    Deligne--Mumford stack over \(\mathbb C\) with trivial generic stabilizer,
    and \(D_{\log}\) is a finite reduced effective Cartier divisor contained in
    the nonstacky locus.  We equip \(\mathcal C\) with the divisorial fs log
    structure associated with \(D_{\log}\).  Let \(C\) be the coarse moduli
    curve of \(\mathcal C\); it is a smooth projective curve, and we identify
    \(D_{\log}\) with its image in \(C\).  Let
    \[
    D_{\mathrm{orb}}\subset C
    \]
    be the image of the stacky locus.  We use
    \[
    \mathcal C^\circ:=\mathcal C\setminus D_{\log}
    \qquad\text{and}\qquad
    C^\circ:=C\setminus(D_{\mathrm{orb}}\cup D_{\log}).
    \]
    Thus \(\mathcal C^\circ\) is the open orbifold retaining the stacky points,
    whereas \(C^\circ\) is its nonstacky coarse locus.
    \end{definition}

    Equivalently, the stacky part of \(\mathcal C\) is determined by the coarse
    curve \(C\), finitely many orbifold points
    \[
    D_{\mathrm{orb}}=\{x_1,\ldots,x_r\}\subset C,
    \]
    and their isotropy orders \(m_i\ge2\).  We therefore often write
    \[
    \mathcal C=(C,D_{\mathrm{orb}}^{\mathbf m},D_{\log}),
    \qquad
    D_{\mathrm{orb}}^{\mathbf m}:=\{(x_i,m_i)\}_{i=1}^r.
    \]
    We use here the standard structure of a smooth proper stacky curve with
    trivial generic stabilizer
    \cite[Definition~5.2.1 and Lemma~5.3.10]{JohnDavid22thecanonicalringofastackycurve}.

    Locally, the two special loci have different meanings.  Near an orbifold point
    of order \(m\), the curve is analytically modeled on
    \[
    [\Delta/\mu_m],
    \]
    where \(\mu_m\) acts on the disc by multiplication
    \cite[Propositions~3.5 and~4.5]{BehrendNoohi06UniformizationofDMcurves}.
    Near a logarithmic point, the underlying space is an ordinary marked disc
    equipped with the divisorial logarithmic structure.

    \begin{definition}[Finite \'etale morphism of log--orbi curves]
    \label{def:etale-morphism-log-orbi}
    A morphism
    \[
    p\colon \mathcal C_2\longrightarrow \mathcal C_1
    \]
    of log--orbi curves is called \emph{finite \'etale} if the underlying
    morphism of Deligne--Mumford stacks is globally representable and finite,
    its restriction over \(\mathcal C_1\setminus D_{1,\log}\) is \'etale, and
    the induced morphism of divisorial fs log stacks is Kummer \'etale along the
    logarithmic boundary.  We also require
    \[
    D_{2,\log}=\bigl(p^{-1}D_{1,\log}\bigr)_{\mathrm{red}}.
    \]
    \end{definition}

    In terms of the notation
    \(
    \mathcal C_1=(C_1,D_{1,\mathrm{orb}}^{\mathbf m_1},D_{1,\log}), \,
    \mathcal C_2=(C_2,D_{2,\mathrm{orb}}^{\mathbf m_2},D_{2,\log}),
    \)
    a finite \'etale morphism
    \[
    p: \mathcal C_2\longrightarrow \mathcal C_1
    \]
    is induced on coarse curves by a finite map
    \(
    f\colon C_2\longrightarrow C_1
    \)
    with the following local behavior.  Over \(C_1^\circ\), the coarse
    map \(f\) is unramified.  If \(x\in D_{1,\mathrm{orb}}\) has orbifold order
    \(m_x\), and \(y\in f^{-1}(x)\) has orbifold order \(m_y\), then analytic
    coordinates may be chosen so that \(p\) has the internal chart
    \[
    [\Delta_w/\mu_{m_y}]\longrightarrow[\Delta_z/\mu_{m_x}],
    \qquad z=w,
    \]
    with an injective inertia homomorphism
    \(\mu_{m_y}\hookrightarrow\mu_{m_x}\).  For coarse coordinates
    \(u=w^{m_y}\) and \(t=z^{m_x}\), one has
    \[
    t=u^{e_y},
    \qquad
    e_y m_y=m_x.
    \]
    Thus \(e_y\) is the coarse ramification index and divides \(m_x\).  At a
    point \(y\in D_{2,\log}\) over \(x\in D_{1,\log}\), the Kummer chart is
    \[
    \mathbb N\longrightarrow\mathbb N,
    \qquad 1\longmapsto e_y,
    \]
    for a positive Kummer index \(e_y\); equivalently, local boundary
    parameters satisfy \(t=u^{e_y}\) up to multiplication by a unit.  In
    particular,
    \[
    D_{2,\log}=\bigl(f^{-1}D_{1,\log}\bigr)_{\mathrm{red}}.
    \]
    These representable covering charts, whose inertia maps at orbifold points
    are injective, must not be confused with root-stack refinements that enlarge
    stabilizers and are generally nonrepresentable.  The same charts give the
    canonical log-canonical pullback isomorphism
    \[
    p^*\omega_{\mathcal C_1}\simeq\omega_{\mathcal C_2},
    \qquad
    \omega_{\mathcal C_i}:=\Omega^1_{\mathcal C_i}(\log D_{i,\log}).
    \]

    Log--orbi curves, and hence the objects of \(\mathbf{HypLO}\), remain
    connected by definition.  We write
    \(
    \mathbf{FEt}(\mathcal C)
    \)
    for the category of finite \'etale covers of \(\mathcal C\), allowing a
    source that is a finite disjoint union of connected log--orbi curves, and
    \[
    \mathbf{FEt}^{\mathrm{conn}}(\mathcal C)
    \subset
    \mathbf{FEt}(\mathcal C)
    \]
    for the full subcategory whose sources are connected.  After choosing a
    geometric base point
    \(
    \bar c\to C^\circ,
    \)
    the fiber functor
    \[
    F_{\bar c}\colon \mathbf{FEt}(\mathcal C)\longrightarrow \mathbf{FSet},
    \qquad
    (\mathcal D\to\mathcal C)\longmapsto \mathcal D_{\bar c},
    \]
    assigns to a cover its finite geometric fiber over \(\bar c\).

    \begin{proposition}
    \label{prop:FEt-galois}
    Let \(\mathcal C\) be a connected log--orbi curve over \(\mathbb C\), and let
    \(\bar c\to C^\circ\) be a geometric base point.  Then
    \(
    (\mathbf{FEt}(\mathcal C),F_{\bar c})
    \)
    is a Galois category.  Its Grothendieck fundamental group is
    \[
    \pi_1^{\et}(\mathcal C,\bar c)
    :=
    \Aut(F_{\bar c}),
    \]
    and there is an equivalence
    \[
    \mathbf{FEt}(\mathcal C)
    \simeq
    \mathbf{FSet}^{\mathrm{cont}}_{\pi_1^{\et}(\mathcal C,\bar c)}.
    \]
    \end{proposition}

    \begin{proof}
    Restriction to \(\mathcal C^\circ=\mathcal C\setminus D_{\log}\)
    identifies \(\mathbf{FEt}(\mathcal C)\) with the ordinary category of
    representable finite \'etale covers of the algebraic Deligne--Mumford curve
    \(\mathcal C^\circ\).  Indeed, the inverse sends such a cover to the
    normalization of \(\mathcal C\) in its function algebra.  Locally at a
    deleted point the resulting map has the form \(t=u^e\), hence is Kummer
    \'etale for the divisorial log structures, and morphisms extend uniquely
    across the boundary.  The assertion is therefore the usual Grothendieck
    Galois formalism for finite \'etale covers of
    \(\mathcal C^\circ\), with the same fiber functor over \(\bar c\).
    \end{proof}

\subsection{Orbifold fundamental groups}\label{subsec:orbifold-fundamental-groups}

    We use two fundamental groups attached to a connected log--orbi curve
    \(\mathcal C\): the profinite \'etale fundamental group, which controls finite
    covers, and the discrete analytic orbifold fundamental group, which underlies
    the Betti theory developed below.

    For a geometric base point \(\bar c\to C^\circ\),
    Proposition~\ref{prop:FEt-galois} defines the profinite group
    \(\pi_1^{\et}(\mathcal C,\bar c)=\Aut(F_{\bar c})\).  On the analytic side,
    the discrete orbifold fundamental group is the deck group of the universal
    orbifold cover.

    \begin{definition}[Discrete orbifold fundamental group]
    \label{def:discrete-orbifold-fundamental-group}
    Let \(\mathcal C\) be a connected complex log--orbi curve and
    \(c\in C^\circ\).  Choose a pointed universal analytic orbifold cover
    \[
    (\widetilde{\mathcal C^\circ},\widetilde c)
    \longrightarrow
    (\mathcal C^\circ,c).
    \]
    The \emph{discrete orbifold fundamental group} of \(\mathcal C\) is
    \[
    \pi_1^{\mathrm{orb}}(\mathcal C,c)
    :=
    \operatorname{Deck}
    (\widetilde{\mathcal C^\circ}/\mathcal C^\circ).
    \]
    The lift \(\widetilde c\) identifies this deck group with the usual
    path-based orbifold fundamental group; changing the lift changes that
    identification by an inner automorphism.
    \end{definition}

    The existence and basic functorial properties of this group may be formulated in the language of fundamental groups of Deligne--Mumford stacks \cite{Noohi02Fundamentalgpsofalgstacks,BehrendNoohi06UniformizationofDMcurves}.

    If
    \(
    \mathcal C=(C,D_{\mathrm{orb}}^{\mathbf m},D_{\log})
    \)
    has signature
    \[
    (g;m_1,\ldots,m_r;s),
    \]
    where \(g=g(C)\), the \(m_i\) are the orbifold orders, and
    \(s=|D_{\log}|\), then van Kampen's theorem gives the standard presentation
    \[
    \pi_1^{\mathrm{orb}}(\mathcal C,c)
    \cong
    \left\langle
    a_1,b_1,\ldots,a_g,b_g;c_1,\ldots,c_r;d_1,\ldots,d_s
    \ \middle|\
    \prod_{i=1}^g[a_i,b_i]\cdot c_1\cdots c_r\cdot d_1\cdots d_s=1,\ 
    c_j^{m_j}=1
    \right\rangle.
    \]
    Here the generators \(c_j\) correspond to loops around orbifold points, while the generators \(d_k\) correspond to loops around logarithmic points.

    \begin{definition}
    \label{def:hyperbolic-log-orbi}
    Let \(\mathcal C\) be a connected log--orbi curve of signature
    \[
    (g;m_1,\ldots,m_r;s).
    \]
    Its orbifold Euler characteristic is
    \[
    \chi_{\mathrm{orb}}(\mathcal C)
    :=
    2-2g-s-\sum_{j=1}^r\left(1-\frac1{m_j}\right).
    \]
    Equivalently, this is the Euler characteristic attached to the standard
    presentation of the discrete orbifold group
    \(\pi_1^{\mathrm{orb}}(\mathcal C,c)\).

    We say that \(\mathcal C\) is
    \emph{hyperbolic} if
    \(
    \chi_{\mathrm{orb}}(\mathcal C)<0
    \).
    We denote by \(\mathbf{HypLO}\) the category of hyperbolic log--orbi curves
    with finite \'etale morphisms.
    \end{definition}

    This definition is independent of the chosen presentation, since the signature is determined by the log--orbi structure.  
    Later, when the log--orbi canonical bundle is introduced, the condition will be equivalently written as
    \[
    \deg\omega_{\mathcal C}>0.
    \]
    For the present section, the group-theoretic formulation is the one needed: hyperbolicity means that the discrete group is of Fuchsian lattice type. 

    \begin{theorem}[Analytic--\'etale comparison]
    \label{thm:analytic-etale-comparison}
    Let \(\mathcal C\) be a connected complex log--orbi curve.  Choose a
    geometric base point \(\bar c\to C^\circ\), a corresponding analytic base
    point \(c\), and a lift of \(c\) to the universal orbifold cover.  Then
    \[
    \widehat{\pi_1^{\mathrm{orb}}(\mathcal C,c)}
    \simeq
    \pi_1^{\et}(\mathcal C,\bar c).
    \]
    Equivalently, finite \'etale covers of \(\mathcal C\) are classified by
    finite continuous sets for the profinite completion of the analytic
    orbifold fundamental group.
    \end{theorem}

    \begin{proof}
    By the restriction equivalence in the proof of
    Proposition~\ref{prop:FEt-galois}, the left-hand covering category is the
    category of representable finite \'etale covers of
    \(\mathcal C^\circ\).  Riemann existence for algebraic stacks identifies
    this category with finite topological covering stacks of the associated
    analytic orbifold, and topological covering stacks are classified by sets
    with an action of its discrete fundamental group
    \cite[Theorems~18.19 and~20.4, and Corollary~20.5]{Behr05Foundationoftopstack}.
    The comparison preserves the fiber over the chosen analytic/geometric
    point.  Taking automorphisms of this fiber functor gives the displayed
    isomorphism.  A different lift, or a different connecting path between
    base points, changes it by an inner automorphism.
    \end{proof}

    Thus \(\pi_1^{\mathrm{orb}}(\mathcal C,c)\) is the group used for Betti representations, while \(\pi_1^{\et}(\mathcal C,\bar c)\) is the group used for finite-cover arguments, Galois descent, and inverse-limit constructions.

    \begin{proposition}
    \label{prop:pi1-limit-galois}
    Let \(\mathcal C\) be a connected log--orbi curve and let \(\bar c\to C^\circ\) be a geometric base point.  
    Then connected finite \'etale Galois covers of \(\mathcal C\) form a cofinal system among connected finite \'etale covers, and
    \[
    \pi_1^{\et}(\mathcal C,\bar c)
    \simeq
    \varprojlim_{\mathcal C'\to \mathcal C}
    \Aut_{\mathcal C}(\mathcal C'),
    \]
    where the inverse limit ranges over connected finite \'etale Galois covers \(\mathcal C'\to\mathcal C\).
    \end{proposition}

    \begin{proof}
    This is the standard inverse-limit description of the fundamental group of
    the Galois category in Proposition~\ref{prop:FEt-galois}: connected Galois
    covers correspond to quotients by open normal subgroups.
    \end{proof}

\subsection{Fuchsian lattices and the compactified quotient functor}\label{subsec:fuchsian-lattices}

    We conclude this section by fixing the analytic category which will be compared
    with hyperbolic log--orbi curves.  We use standard facts about Fuchsian groups
    from \cite{Beardon1995TheGO,Katok1992FuchsianGroups}.

    A \emph{Fuchsian group} is a discrete subgroup of \(\PSL_2(\mathbb R)\).
    A cofinite Fuchsian group of signature \((g;m_1,\ldots,m_r;s)\) has the
    presentation and orbifold Euler characteristic recorded for the
    corresponding log--orbi group in
    Subsection~\ref{subsec:orbifold-fundamental-groups}; see also
    \cite[\S2, p.~228]{Faltings83RealprojstronRiemannsur}.
    Writing this number as \(\chi(\Gamma)\), one has
    \(\chi(\Gamma)<0\), and the hyperbolic
    area is proportional to \(-\chi(\Gamma)\); see
    \cite{Katok1992FuchsianGroups}.  
    In this paper, a \emph{Fuchsian lattice} means a cofinite Fuchsian subgroup
    \[
    \Gamma\subset \PSL_2(\mathbb R)
    \]
    of negative orbifold Euler characteristic.

    \begin{definition}[Transporter category of Fuchsian lattices]
    \label{def:FL}
    Let \(\mathbf{FL}\) be the category whose objects are Fuchsian lattices
    \(
    \Gamma\subset \PSL_2(\mathbb R)
    \).
    For two objects, set
    \[
    \operatorname{Trans}(\Gamma_1,\Gamma_2)
    :=
    \left\{
    g\in\PSL_2(\mathbb R)
    \ \middle|\
    g\Gamma_1g^{-1}\subset\Gamma_2
    \text{ with finite index}
    \right\}
    \]
    and
    \[
    \operatorname{Hom}_{\mathbf{FL}}(\Gamma_1,\Gamma_2)
    :=
    \Gamma_2\backslash\operatorname{Trans}(\Gamma_1,\Gamma_2).
    \]
    We call \([g]=\Gamma_2g\) a \emph{transporter-coset morphism} and
    \(g\) a representative of it.
    If \([g]:\Gamma_1\to\Gamma_2\) and
    \([h]:\Gamma_2\to\Gamma_3\), define
    \[
    [h]\circ[g]:=[hg],
    \qquad
    \operatorname{id}_\Gamma=[1].
    \]
    \end{definition}

    Thus morphisms record finite covering maps while retaining the correct
    automorphisms of fixed quotient orbifolds; in particular
    \(\operatorname{Aut}_{\mathbf{FL}}(\Gamma)
    =\Gamma\backslash N_{\PSL_2(\mathbb R)}(\Gamma)\).  We do not include
    orientation-reversing transformations.
    Given a Fuchsian lattice \(\Gamma\subset \PSL_2(\mathbb R)\), the quotient
    \(
    \Gamma\backslash\mathbb H
    \)
    is a finite-area hyperbolic orbifold.  Its elliptic points give the orbifold
    locus, and its cusps give logarithmic points after compactification.  Hence the
    cusp compactification
    \[
    (\Gamma\backslash\mathbb H)^c
    \]
    is a hyperbolic log--orbi curve.
    A morphism \([g]:\Gamma_1\to\Gamma_2\) induces
    \[
    \Gamma_1\backslash\mathbb H\longrightarrow
    \Gamma_2\backslash\mathbb H,
    \qquad
    [z]\longmapsto[gz].
    \]
    The left-coset ambiguity is a target deck transformation, so the map is
    well defined by \([g]\).  It is a finite orbifold covering on the open part
    and extends uniquely across each cusp by the local Kummer chart.  Hence the
    compactified quotient construction defines a functor
    \[
    Q_c\colon \mathbf{FL}\longrightarrow \mathbf{HypLO},
    \qquad
    \Gamma\longmapsto(\Gamma\backslash\mathbb H)^c.
    \]
    The categorical uniformization theorem will construct a quasi-inverse to
    \(Q_c\) by recovering the lattice \(\Gamma\) from a hyperbolic log--orbi curve
    as the Betti realization of its canonical maximal \(\PSL_2\)-object.

\section{Log--orbi objects, parahoric bookkeeping, and the canonical
\texorpdfstring{\(\PSL_2\)}{PSL\_2}-object}
\label{sec:parahoric-canonical-psl2}

    The differential objects in this paper live intrinsically on log--orbi
    curves.  Parahoric structures serve as a tensor-compatible coarse-curve
    description of their local data, using rational cocharacters and
    Moy--Prasad filtrations
    \cite{BalaSesha15ModuliofparhorictorsoronRiemannsurface,
    Wilson2010TannakianParahoricBT}.

\subsection{Local types, parahoric bundles, and degrees}
\label{subsec:parahoric-bundles}

    Let \(G\) be a connected complex reductive group, let \(T\subset G\) be a
    maximal torus, and let \(W\) be the Weyl group.  We write
    \[
    X^*(T)=\Hom(T,\mathbb G_m),
    \qquad
    X_*(T)=\Hom(\mathbb G_m,T),
    \]
    with pairing
    \[
    \langle\ ,\ \rangle:X^*(T)\times X_*(T)\to\mathbb Z.
    \]
    Let \(\mathfrak g=\Lie(G)\).

    \subsubsection{Local types and model filtrations}

    A local type at a marked point \(x\) is an element
    \[
    \theta_x \in X_*(T)_\mathbb Q/W, \qquad
    X_*(T)_\mathbb Q := X_*(T)\otimes_\mathbb Z \mathbb Q.
    \]
    Although formulas are written after choosing a dominant representative
    \(\theta_x\in X_*(T)_\mathbb Q\), the associated filtered tensor functor
    depends only on the class of \(\theta_x\) in \(X_*(T)_\mathbb Q/W\).

    The adjoint action of \(\theta_x\) gives a \(\mathbb Q\)-grading
    \[
    \mathfrak g=
    \bigoplus_{r\in\mathbb Q}\mathfrak g_r(\theta_x),
    \qquad
    \mathfrak g_r(\theta_x)
    =
    \{Y\in \mathfrak g\mid
    \mathrm{Ad}(\theta_x(t))(Y)=t^rY\}.
    \]
    The associated Moy--Prasad filtration is
    \[
    \mathfrak g_{\ge a}(\theta_x)
    :=
    \bigoplus_{r\ge a}\mathfrak g_r(\theta_x),
    \qquad
    \mathfrak g_{>a}(\theta_x)
    :=
    \bigoplus_{r>a}\mathfrak g_r(\theta_x).
    \]
    In particular,
    \[
    \mathfrak p_{\theta_x}:=\mathfrak g_{\ge0}(\theta_x)
    \]
    is a parabolic Lie algebra, and
    \[
    \mathfrak g_{>0}(\theta_x)
    \]
    is its positive Moy--Prasad part.

    For an algebraic representation
    \(
    \rho:G\to\GL(V),
    \)
    write
    \[
    V=\bigoplus_{\chi\in X^*(T)}V_\chi
    \]
    for the \(T\)-weight decomposition.  The local type \(\theta_x\) induces the
    model filtration
    \[
    F^{\theta_x}_{\ge a}V
    :=
    \bigoplus_{\langle\chi,\theta_x\rangle\ge a}V_\chi.
    \]
    All filtrations in this paper are written in decreasing Moy--Prasad
    notation \(F_{\ge a}\); with this convention positive Moy--Prasad degree
    strictly raises the filtration.
    The Moy--Prasad filtration and the model filtration are compatible in the
    sense that
    \[
    d\rho\bigl(\mathfrak g_{\ge b}(\theta_x)\bigr)
    \subset
    \left\{
    A\in\End(V)
    \ \middle|\
    A\bigl(F^{\theta_x}_{\ge a}V\bigr)
    \subset
    F^{\theta_x}_{\ge a+b}V
    \text{ for all }a
    \right\}.
    \]
    Thus \(\mathfrak g_{\ge0}(\theta_x)\) acts by filtration-preserving
    endomorphisms, while \(\mathfrak g_{>0}(\theta_x)\) acts by strictly
    filtration-raising endomorphisms.

    We write \(\mathbf{ParVect}(X,D)\) for the rigid exact tensor category of
    rational periodic parabolic vector bundles on \((X,D)\), with the
    parabolic tensor product.  The periodic indexing retains the integral
    lattice shifts that disappear after normalizing the weights to a
    fundamental interval.

    Fix a local parameter \(z\) at \(x\) and write
    \(\widehat{\mathcal O}_x\simeq\mathbb C[[z]]\).  For a chosen
    representative of \(\theta_x\), the full periodic lattice model is
    \[
    \Lambda^{\theta_x}_{\ge a}(V)
    :=
    \bigoplus_{\chi}
    z^{\lceil a-\langle\chi,\theta_x\rangle\rceil}
    \widehat{\mathcal O}_x\otimes V_\chi,
    \qquad a\in\mathbb Q.
    \]
    It satisfies
    \(\Lambda^{\theta_x}_{\ge a+1}(V)
    =z\Lambda^{\theta_x}_{\ge a}(V)\).  Choosing a fundamental interval
    and a reference lattice extracts the normalized filtered fiber; the
    reference lattice records the integral shifts.

    \begin{definition}[Model parahoric datum at a point]
    \label{def:parahoric-structure-point}
    Let \(x\in D\).  A periodic parabolic tensor functor has local type
    \(\theta_x\) if, after choosing a dominant representative, its completed
    periodic lattice datum is equipped with a facetwise enhancement,
    fpqc-locally tensor-isomorphic to the standard Wilson system over the facet
    defining \(\mathcal G_{\theta_x}\), whose \(\theta_x\)-component is
    \(V\mapsto\Lambda^{\theta_x}_{\ge\bullet}(V)\).  The enhancement retains
    all inclusion, periodicity, functoriality, unit, and tensor
    compatibilities.
    After choosing a parabolic weight window, this local isomorphism induces
    the corresponding normalized filtered fiber functor.  The periodic datum,
    rather than only the normalized filtered fiber, retains the integral part
    of \(\theta_x\).  Morphisms preserve the entire facetwise enhancement;
    no trivialization of it is part of the datum.
    \end{definition}

    \begin{example}[The case \(G=\GL_n\)]
    \label{ex:parahoric-GLn}
    For \(G=\GL_n\), a local type is represented by
    \[
    \theta_x=(\alpha_1,\ldots,\alpha_n)\in\mathbb Q^n/S_n.
    \]
    After choosing a dominant representative, assume
    \[
    \alpha_1\ge\alpha_2\ge\cdots\ge\alpha_n.
    \]
    For the standard representation \(V=\mathbb C^n\), the model filtration is
    \[
    F^{\theta_x}_{\ge a}V
    =
    \bigoplus_{\alpha_i\ge a}\mathbb C e_i.
    \]
    Hence \(\theta_x\) recovers the usual weighted flag of a parabolic vector
    bundle.

    We do not impose the normalization \(0\le \alpha_i<1\). 
    From the parabolic-sheaf viewpoint of Iyer--Simpson \cite{IyerSimpson07Arelationbetweentheparabolic}, the intrinsic object is a periodic system of local lattices indexed by rational weights; choosing a normalized window changes the local lattice by elementary modification.
    Thus the normalized parabolic weights are obtained only after choosing a fundamental window for the periodic lattice system, while the parahoric local type itself is the unnormalized rational cocharacter class.
    \end{example}

    \begin{remark}
    If the local type
    \(
    \theta_x\in X_*(T)_{\mathbb Q}/W
    \)
    is represented by an \textbf{integral} cocharacter, then its fractional parahoric type is trivial.  
    Equivalently, after normalizing weights modulo \(\mathbb Z\), all associated parabolic weights are zero, hence the parabolic structure is trivial.  
    In Yokogawa's language, the parabolic structure is trivial if all parabolic weights are zero; see \cite[Def.~1.1]{Yokogawa1995Infinitesimal}.    
    
    \end{remark}

    \begin{definition}
    Let \(G\) and \(H\) be connected complex reductive groups, with maximal
    tori \(T_G\subset G\) and \(T_H\subset H\).  Let
    \[
    \theta_G=(\theta_{G,x})_{x\in D},
    \qquad
    \theta_H=(\theta_{H,x})_{x\in D}
    \]
    be systems of rational parahoric local types, where
    \[
    \theta_{G,x}\in X_*(T_G)_{\mathbb Q}/W_G,
    \qquad
    \theta_{H,x}\in X_*(T_H)_{\mathbb Q}/W_H.
    \]
    Their product local type is
    \[
    \theta_G\boxtimes\theta_H
    :=
    (\theta_{G,x},\theta_{H,x})_{x\in D}
    \in
    X_*(T_G\times T_H)_{\mathbb Q}/(W_G\times W_H).
    \]

    If \(P\) is a parahoric \(G\)-bundle of type \(\theta_G\) and \(Q\) is a
    parahoric \(H\)-bundle of type \(\theta_H\), their external parahoric
    product is the torsor
    \[
    P\boxtimes_{\mathrm{par}} Q:=P\times_X Q
    \]
    under the product parahoric group scheme, equipped with local type
    \(\theta_G\boxtimes\theta_H\).  Its periodic parabolic tensor functor is
    \[
    \omega^{\mathrm{par}}_{P\boxtimes Q}
    =
    \omega^{\mathrm{par}}_{P}
    \boxtimes
    \omega^{\mathrm{par}}_{Q}.
    \]
    \end{definition}

    For \(V\in \Rep_{\mathbb C}(G)\) and \(W\in \Rep_{\mathbb C}(H)\), the linear operations on the associated parahoric vector bundles are obtained by applying the standard rigid tensor operations in \(\mathrm{Fil}_{\mathbb Q}(\mathrm{Vect}_{\mathbb C})\). 
    Explicitly, the induced decreasing filtrations at \(x\in D\) are given by \[F^{(x)}_{\ge c}(E\oplus_{\mathrm{par}} F)_x = F^{(x)}_{\ge c} E_x \oplus F^{(x)}_{\ge c} F_x\], the convolution filtration \[F^{(x)}_{\ge c}(E\otimes_{\mathrm{par}} F)_x = \sum_{a+b \ge c} F^{(x)}_{\ge a} E_x \otimes F^{(x)}_{\ge b} F_x,\] the dual filtration \(F^{(x)}_{\ge c}(E^\vee_{\mathrm{par}})_x = \{ \varphi \in E_x^\vee \mid \varphi(F^{(x)}_{>-c} E_x) = 0 \}\), and analogously for \(\mathcal Hom_{\mathrm{par}}\).

    \begin{proposition}
    Parahoric vector bundles with decreasing Moy--Prasad filtrations form a
    rigid exact tensor category under
    \[
    \oplus,\qquad
    \otimes_{\mathrm{par}},\qquad
    (\cdot)^\vee_{\mathrm{par}},\qquad
    \mathcal Hom_{\mathrm{par}}.
    \]
    For a principal parahoric \(G\)-bundle \(P\), the associated bundle functor
    \[
    \omega_P:\Rep_{\mathbb C}(G)\longrightarrow \mathbf{ParVect}(X,D),
    \qquad
    V\longmapsto P(V),
    \]
    is exact, faithful, and symmetric monoidal.  In particular,
    \[
    P(V\oplus W)\simeq P(V)\oplus_{\mathrm{par}}P(W),
    \]
    \[
    P(V\otimes W)\simeq P(V)\otimes_{\mathrm{par}}P(W),
    \]
    \[
    P(V^\vee)\simeq P(V)^\vee_{\mathrm{par}},
    \]
    and
    \[
    P(\mathcal Hom(V,W))
    \simeq
    \mathcal Hom_{\mathrm{par}}(P(V),P(W)).
    \]
    \end{proposition}

    The fundamental operation for principal parahoric objects is the product
    \[
    (G,\theta_G)\boxtimes(H,\theta_H)
    =
    (G\times H,\theta_G\boxtimes\theta_H).
    \]
    The usual linear operations on associated parahoric vector bundles are then
    obtained by applying the direct-sum, tensor, dual, and Hom representations.
    Thus the tensor formalism is defined at the level of local types and
    principal objects, while the familiar convolution formulas are its
    associated-vector realization.

\subsubsection{Parahoric \texorpdfstring{\(G\)}{G}-bundles}

    Let \(X\) be a smooth projective curve and let \(D\subset X\) be a reduced
    divisor.  Fix local types
    \[
    \mathbf{\theta}=(\theta_x)_{x\in D},
    \qquad
    \theta_x\in X_*(T)_{\mathbb Q}/W.
    \]
    Choose positive integers \(r_x\) clearing the denominators of
    \(\theta_x\), and write \(\mathbf r=(r_x)_{x\in D}\).  We denote by
    \[
    \mathbf{ParVect}_{\mathbf r}(X,D)
    \]
    the full rigid exact tensor category of periodic parabolic vector bundles
    whose jumps at \(x\) lie in \(r_x^{-1}\mathbb Z\), equipped with the
    parabolic tensor product.

    \begin{definition}
    \label{def:parahoric-G-bundle}
    Let \(\mathcal G_\theta\) be the Bruhat--Tits group scheme on \(X\)
    determined by the local types \(\theta_x\).  A \emph{parahoric
    \(G\)-bundle of type \(\theta\)} on \((X,D)\) is a
    \(\mathcal G_\theta\)-torsor on \(X\).
    \end{definition}

    For rational types, choosing a normalized parabolic window identifies
    \(\mathbf{ParVect}_{\mathbf r}(X,D)\) with vector bundles on the root stack
    \[
    \mathfrak X_{\mathbf r}
    =X\bigl[\sqrt[r_x]{x}\mid x\in D\bigr].
    \]
    This tensor equivalence is the root-stack/parabolic correspondence
    \cite[Theorem~2.12]{BorneAmine23paraconneandstackofroots}; its
    principal tensor-functor form is
    \cite[Proposition~6.1]{BiswasMajumderWong12RootStacks}.
    After choosing a representative of \(\theta_x\), the residual gerbe
    \(i_x:B\mu_{r_x}\hookrightarrow\mathfrak X_{\mathbf r}\) carries the
    inertia homomorphism
    \[
    \tau_{\theta_x}:\mu_{r_x}\longrightarrow T\subset G,
    \qquad
    \tau_{\theta_x}(\zeta)=(r_x\theta_x)(\zeta),
    \]
    whose \(G\)-conjugacy class is independent of the representative and
    records the fractional part of the local type.  The full unnormalized type
    is the periodic lattice model of Definition~\ref{def:parahoric-structure-point};
    this additional datum retains integral lattice shifts.  We use this
    periodic parabolic formulation below.  Wilson characterizes a pointwise
    lattice model in Theorem~3.7.2, defines the automorphisms of the full
    facetwise system in Definition~4.1.1, proves representability and
    constructs the parahoric comparison morphism in Theorem~4.2.1, and proves
    that it is an isomorphism in the present residue-characteristic-zero
    setting in Theorem~5.2.2
    \cite{Wilson2010TannakianParahoricBT}.

    \begin{proposition}
    \label{prop:tannakian-parahoric-bundles}
    The groupoid of parahoric \(G\)-bundles of type \(\theta\) on \((X,D)\)
    is equivalent to the groupoid of
    exact faithful \(\mathbb C\)-linear strong symmetric monoidal functors
    \[
    \omega^{\mathrm{par}}:
    \Rep_{\mathbb C}(G)
    \longrightarrow
    \mathbf{ParVect}_{\mathbf r}(X,D)
    \]
    equipped at every \(x\in D\) with the facetwise local enhancement of
    Definition~\ref{def:parahoric-structure-point}.  The monoidal structure on
    the target is the parabolic tensor product, and arrows on the right are
    tensor natural isomorphisms preserving these enhancements.
    \end{proposition}

    \begin{proof}
    A \(\mathcal G_\theta\)-torsor gives the displayed associated-bundle
    functor and its standard facetwise lattice systems.  Conversely, on
    \(U=X\setminus D\), relative Tannakian reconstruction gives a
    \(G\)-torsor.  On
    \(\widehat X_x=\operatorname{Spec}\widehat{\mathcal O}_{X,x}\), the
    tensor-isomorphism sheaf from the standard Wilson facetwise system to the
    prescribed enhancement is a torsor under its tensor automorphism group,
    which Wilson identifies with \(\mathcal G_{\theta_x}\).  Over the
    punctured formal disc it reconstructs the same generic tensor functor as
    the torsor on \(U\).  For the fpqc cover
    \(U\amalg\coprod_{x\in D}\widehat X_x\to X\), these generic
    identifications supply the descent cocycle; effective descent for affine
    torsors then glues the local objects to a
    \(\mathcal G_\theta\)-torsor on \(X\)
    \cite[Tags 04US and 0245]{StacksProject}.  The associated-lattice
    construction and this reconstruction are inverse on objects and arrows.
    \end{proof}

    Let \(E\) be a parahoric \(G\)-bundle of type \(\theta\). For any representation \(V\in\Rep_{\mathbb C}(G)\), we write
    \[
    F^{(x)}_{\ge a}E(V)_x
    \]
    for the induced filtration at \(x\), and
    \[
    \mathrm{gr}_a^{(x)}E(V)_x
    :=
    F^{(x)}_{\ge a}E(V)_x/F^{(x)}_{>a}E(V)_x.
    \]
    Following the standard Mehta--Seshadri convention, the associated
    parahoric vector degree is defined as
    \[
    \operatorname{pdeg}_\theta E(V)
    =
    \deg E(V)
    +
    \sum_{x\in D}\sum_a
    a\cdot \dim_{\mathbb C}\mathrm{gr}_a^{(x)}E(V)_x.
    \]
    This is the degree of the associated parahoric vector bundle \(E(V)\).
    It is functorial in \(V\), but it is not a representation-independent
    scalar degree of the principal \(G\)-bundle itself.  The intrinsic
    principal degree used in the stability condition below is attached, as in
    Ramanathan's formulation, to a parahoric reduction and a character.

    \begin{remark}
    The preceding triviality refers only to the fractional parahoric, or stacky,
    structure.  In the unnormalized Moy--Prasad convention used here, an
    integral cocharacter may still contribute an integral lattice shift to the
    preceding degree formula.  A normalized parabolic window absorbs this shift
    by elementary modification: the normalized parabolic structure is trivial,
    while the unnormalized parahoric description remembers the lattice shift.
    \end{remark}

    \begin{definition}[Principal parahoric degree]
    \label{def:principal-parahoric-degree}
    Let \(Q\subset G\) be a parabolic subgroup, let \(E_Q\) be a parahoric
    reduction of \(E\) to \(Q\), and let \(\chi\in X^*(Q)\).  The associated
    parabolic line bundle
    \[
    (E_Q)_\chi:=E_Q\times^Q\mathbb C_\chi
    \]
    carries the filtered, equivalently parabolic, line structure induced from
    the \(Q\)-reduction.  Its parahoric degree is
    \[
    \operatorname{pdeg}_\theta(E_Q,\chi)
    =
    \deg((E_Q)_\chi)
    +
    \sum_{x\in D}a_x(\chi),
    \]
    where \(\deg((E_Q)_\chi)\) is the ordinary degree of the underlying line
    in the chosen parabolic window and \(a_x(\chi)\) is the unique jump of its
    filtered fiber at \(x\).  After choosing a representative of \(\theta_x\)
    and a local splitting compatible with the reduction and this window,
    \[
    a_x(\chi)=\langle\chi,\theta_x\rangle.
    \]
    This extends Ramanathan's degree of a reduction to the parahoric setting
    \cite{RamananRamanathan84InstabilityFlag,
    BalaSesha15ModuliofparhorictorsoronRiemannsurface,
    BalaBiswasPandy17Conneonparahorictorovercurves}.
    \end{definition}

    \begin{proposition}[Intrinsic nature of parahoric degree]
    \label{prop:principal-parahoric-degree-independent}
    The number \(\operatorname{pdeg}_\theta(E_Q,\chi)\) is independent of the
    parabolic window, of the choice of local splitting compatible with the
    reduction, and of any auxiliary faithful representation used to realize
    the parahoric bundle.
    If \(\rho:G\to GL(V)\) is a representation and
    \(\chi_\rho=\det\circ\rho\), then, for \(Q=G\),
    \[
    \operatorname{pdeg}_\theta E(V)
    =
    \operatorname{pdeg}_\theta(E,\chi_\rho).
    \]
    Thus the numerical associated vector degree is determined by the
    determinant character of \(\rho\).
    \end{proposition}

    \begin{proof}
    The reduction \(E_Q\) and the character \(\chi\) canonically determine
    the parabolic line \((E_Q)_\chi\).  Applying the parabolic
    associated-bundle construction to the \(Q\)-reduction, the
    one-dimensional \(Q\)-representation \(\mathbb C_\chi\) receives a
    filtered line at each
    \(x\in D\).  In a filtered line there is only one nonzero graded piece,
    so its jump \(a_x(\chi)\) is
    intrinsic.  Changing a compatible local splitting changes a local basis of
    this line by a nonzero scalar and does not change the filtration jump.
    Tensoriality gives the same filtered line for any auxiliary faithful
    realization.  For a fixed window the ordinary degree of
    \((E_Q)_\chi\) is intrinsic; changing the window performs an elementary
    modification, and the change in ordinary degree is the negative of the
    corresponding change in the local jump.  Hence their sum is independent
    of all the stated choices.

    For the determinant comparison, the parabolic tensor determinant satisfies
    \(\det_{\mathrm{par}}E(V)\simeq E_{\chi_\rho}\), and its filtration has
    local weight
    \[
    \sum_a a\,\dim\operatorname{gr}^{(x)}_aE(V)_x.
    \]
    The two degree formulas are therefore identical.
    \end{proof}

    Later, when stability is formulated using strictly antidominant characters \(\chi\), the semistability inequality will require
    \[
    \operatorname{pdeg}_{\theta}(E_Q,\chi)\ge 0.
    \]
    This is the inverse-character form of the convention in
    \cite[\S6.3, Definition~6.3.4, pp.~31--32]
    {BalaSesha15ModuliofparhorictorsoronRiemannsurface}: if
    \(\chi_{\mathrm{BS}}=\chi^{-1}\) is their dominant character, then
    \((E_Q)_\chi\simeq((E_Q)_{\chi_{\mathrm{BS}}})^\vee\) and
    \(\operatorname{pdeg}_\theta(E_Q,\chi)
    =-\operatorname{pdeg}_\theta(E_Q,\chi_{\mathrm{BS}})\).
    Thus their \(\le0\) inequality is precisely the above \(\ge0\)
    inequality.  This fixes the sign convention for parahoric degree and
    stability used throughout the paper.

\subsection{Log--orbi \texorpdfstring{\(\lambda\)}{lambda}-connections and the auxiliary parahoric dictionary}
\label{subsec:parahoric-lambda-log-orbi}

    Let \(\mathcal C\) be a log--orbi curve.  A principal
    \(G\)-\(\lambda\)-connection is defined intrinsically on its
    lisse--\'etale site, with logarithmic cotangent sheaf
    \[
    \Omega^1_{\mathcal C,\log}:=\Omega^1_{\mathcal C}(\log D_{\log}),
    \]
    \cite{laumonBailly00champsalg,Olsson2016AlgebraicSpacesandStacks}.  On an
    orbifold chart \([\Delta/\mu_m]\) it is a regular
    \(\mu_m\)-equivariant object on \(\Delta\); at a logarithmic point it may
    have a pole generated by \(dz/z\).  Tensorially, for every
    algebraic representation \(\rho:G\to \GL(V)\), the associated bundle
    \(P_\rho=P\times^G V\) is equipped functorially with a logarithmic
    \(\lambda\)-connection
    \[
    \nabla^{\lambda,\rho}:P_\rho
    \longrightarrow
    P_\rho\otimes\Omega^1_{\mathcal C,\log},
    \]
    satisfying the \(\lambda\)-Leibniz rule and compatible with tensor products,
    duals, and morphisms in \(\Rep_{\mathbb C}(G)\).  Equivalently, this is a
    \(\lambda\)-splitting of the logarithmic Atiyah sequence
    \[
    0\longrightarrow \operatorname{ad}(P)
    \longrightarrow \operatorname{At}_{\log}(P)
    \longrightarrow T_{\mathcal C}(-\log D_{\log})
    \longrightarrow 0.
    \]
    For \(\lambda=1\) this is a logarithmic principal connection, and for
    \(\lambda=0\) it is a principal Higgs field.  Since all curves considered
    here are one-dimensional, integrability is automatic.

\subsubsection{Principal parahoric \texorpdfstring{\(\lambda\)}{lambda}-connections}

    Let \((X,D)\) be a smooth pointed curve and let
    \[
    \theta=(\theta_x)_{x\in D},
    \qquad
    \theta_x\in X_*(T)_{\mathbb Q}/W,
    \]
    be a system of rational local types.
    Let \(\mathbf{\Lambda}^{\lambda}_{\mathrm{par}}(X,D)\) denote the
    ambient exact tensor category of periodic parabolic vector bundles of
    arbitrary rational local type equipped with filtration-preserving
    logarithmic \(\lambda\)-connections.

    \begin{definition}
    \label{def:parahoric-lambda-connection}
    Let \(E\) be a parahoric \(G\)-bundle of type \(\theta\) on \((X,D)\).
    A \emph{parahoric \(G\)-\(\lambda\)-connection} on \(E\) is a tensor-compatible family of logarithmic \(\lambda\)-connections
    \[
    \nabla^{\lambda,V}\colon
    E(V)\longrightarrow E(V)\otimes\Omega_X^1(\log D),
    \qquad
    V\in\Rep_{\mathbb C}(G),
    \]
    such that, for every \(x\in D\), the residue preserves the parahoric filtration:
    \[
    \Res_x(\nabla^{\lambda,V})
    \bigl(F^{(x)}_{\ge a}E(V)_x\bigr)
    \subset
    F^{(x)}_{\ge a}E(V)_x
    \qquad
    \text{for all }a\in\mathbb Q.
    \]
    Here \(E(V)\) is the associated periodic parabolic vector bundle and
    \(F^{(x)}_{\ge a}E(V)_x\) is its decreasing filtration in a local
    parabolic window.
    Equivalently, the assignment
    \[
    V\longmapsto (E(V),\nabla^{\lambda,V})
    \]
    is an exact faithful strong symmetric monoidal functor
    \[
    \Rep_{\mathbb C}(G)
    \longrightarrow
    \mathbf{\Lambda}^{\lambda}_{\mathrm{par}}(X,D)
    \]
    whose underlying periodic parabolic tensor functor has local type
    \(\theta\).
    \end{definition}
    We next define the auxiliary coarse-curve residue conditions.
    \begin{definition}[Adjustedness and logarithmic strata]
    \label{def:adjusted-algebraic-log}
    Let \((E,\nabla^\lambda)\) be a parahoric
    \(G\)-\(\lambda\)-connection of local type \(\theta_x\) at \(x\).
    \begin{enumerate}
    \item The connection is \emph{adjusted at \(x\)} if, after choosing a
    representative \(\widetilde\theta_x\) of \(\theta_x\) and a compatible
    local parahoric trivialization, for every algebraic representation
    \(\rho:G\to\GL(V)\) there is an endomorphism \(N_{V,x}\) such that
    \[
    \operatorname{Res}_x(\nabla^{\lambda,V})
    =
    \lambda\,d\rho_V(\widetilde\theta_x)+N_{V,x},
    \qquad
    N_{V,x}\bigl(F^{(x)}_{\ge a}E(V)_x\bigr)
    \subset
    F^{(x)}_{>a}E(V)_x
    \quad(a\in\mathbb Q).
    \]
    \item Suppose that \(x\) is an ordinary logarithmic point and that the
    connection is adjusted there.
    \begin{enumerate}
    \item If \(\lambda=0\), it is of \emph{algebraic type} at \(x\) if
    \(N_{V,x}=0\) for every \(V\), and of \emph{log type} at \(x\) if
    \(N_{V,x}\ne0\) for some \(V\).
    \item If \(\lambda\ne0\), let \(T_{x,u}\) be the unipotent factor of any
    representative of the local monodromy conjugacy class of the normalized
    flat germ \(\lambda^{-1}\nabla^\lambda\).  The connection is of
    \emph{algebraic type} at \(x\) if \(T_{x,u}=1\), and of
    \emph{log type} at \(x\) if \(T_{x,u}\ne1\).
    \end{enumerate}
    \end{enumerate}
    \end{definition}

    The adjusted residue condition is equivalently
    \[
    \operatorname{Res}_x(\nabla^{\lambda,V})
    \big|_{\operatorname{gr}^{(x)}_aE(V)_x}
    =
    \lambda a\cdot\operatorname{id}
    \]
    on every associated graded piece.  After choosing a representative
    \(\widetilde\theta_x\) of \(\theta_x\) and a compatible local parahoric
    trivialization, its principal shorthand is
    \[
    \operatorname{Res}_x(\nabla^\lambda)
    =
    \lambda\widetilde\theta_x+N_x,
    \qquad
    N_x\in\mathfrak g_{>0}(\widetilde\theta_x),
    \]
    with \(N_{V,x}=d\rho_V(N_x)\).  Thus adjustedness includes both the zero
    and nonzero positive Moy--Prasad terms; for \(\lambda=0\), these are
    precisely the algebraic-type and log-type strata.

    For \(G=\GL_n\), adjustedness is precisely the strongly parabolic
    \(\lambda\)-condition of
    \cite[Definition~4.3 and Section~4.6]
    {BorneAmine23paraconneandstackofroots}, after translating the filtration
    and residue-sign conventions.

    The usual direct-sum, tensor-product, dual, and internal-Hom connection
    formulas preserve the strictly filtration-raising term.  Hence adjusted
    objects form a rigid tensor category; the algebraic stratum is preserved,
    whereas the complementary log stratum is not a tensor subcategory.

    \begin{proposition}\label{prop:tannakian-parahoric-connections}
    Let \(\mathbf{\Lambda}^\lambda_{\mathrm{adj}}(X,D)\) be the rigid
    tensor category of adjusted periodic parabolic vector
    \(\lambda\)-connections of arbitrary rational local type.  Adjusted
    parahoric principal \(G\)-\(\lambda\)-connections of type \(\theta\) on
    \((X,D)\) are canonically equivalent to exact faithful
    \(\mathbb C\)-linear strong symmetric monoidal functors
    \[
    \Rep_{\mathbb C}(G)
    \longrightarrow
    \mathbf{\Lambda}^\lambda_{\mathrm{adj}}(X,D)
    \]
    whose underlying periodic parabolic tensor functor has local type
    \(\theta\).
    Under this equivalence, the subgroupoid cut out by algebraic type at the
    ordinary logarithmic points corresponds to tensor functors whose
    associated object in one, equivalently every, faithful representation is
    of algebraic type there.  Log type is the complementary adjusted stratum
    at those points and is not imposed by requiring a tensor functor to land
    in a vector-valued log-type subcategory.
    \end{proposition}

    \begin{proof}
    A parahoric principal \(G\)-\(\lambda\)-connection gives, by associated
    parabolic bundles, an exact tensor functor
    \[
    V\longmapsto (E(V),\nabla^{\lambda,V})
    \]
    to the ambient varying-type category.  Its local type is \(\theta\) in the
    sense of the underlying periodic parabolic functor.  Conversely, forgetting
    the \(\lambda\)-connections and applying
    Proposition~\ref{prop:tannakian-parahoric-bundles} reconstructs the
    underlying parahoric principal \(G\)-bundle.
    The compatible family of logarithmic \(\lambda\)-connections on all associated bundles then reconstructs the principal \(\lambda\)-connection, equivalently the \(\lambda\)-splitting of the logarithmic Atiyah sequence.

    The residue condition is tensorial: after choosing a representative
    \(\widetilde\theta_x\) of \(\theta_x\), adjustedness is exactly the
    condition that every algebraic representation of \(G\) satisfy
    \[
    \operatorname{Res}_x(\nabla^{\lambda,V})
    =
    \lambda\,d\rho_V(\widetilde\theta_x)+N_{V,x},
    \qquad
    N_{V,x}(F_{\ge a})\subset F_{>a}.
    \]
    Thus adjustedness is checked representation by representation.  The
    algebraic condition at an ordinary logarithmic point is likewise
    tensorial.  At \(\lambda=0\), a faithful differential detects whether
    \(N_x\) vanishes; at \(\lambda\ne0\), a faithful representation detects
    whether the unipotent local monodromy is trivial.  Log type is the
    complementary intrinsic condition there; it does not define a
    vector-valued tensor subcategory in which the tensor functor must land.
    \end{proof}

\subsubsection{Functoriality}

    We establish the standard functorialities in the form needed later. 
    The filtered-functor and vector root-stack background is developed in
    \cite{Ziegler2015FilteredFiberFunctors,BorneAmine23paraconneandstackofroots}.
    The principal compatibilities below are checked after applying algebraic
    representations.
    Because the underlying parahoric bundles and their logarithmic \(\lambda\)-connections transform compatibly, we state these properties unified for the full differential objects.

    \begin{proposition}
    \label{prop:parahoric-lambda-tensor-operations}
    The category
    \(\mathbf{\Lambda}^\lambda_{\mathrm{adj}}(X,D)\) is rigid tensor, with
    filtrations given by the standard parabolic tensor rules.  At ordinary
    logarithmic points, its algebraic-type sector is a tensor subcategory.
    Fixed-type pieces and the complementary log-type locus are not tensor
    subcategories, so membership in the log-type locus must be checked after
    each operation.  At the principal level, the corresponding operation is
    the external product \(\boxtimes\), not an internal tensor product for a
    fixed structure group.
    \end{proposition}

    \begin{proof}
    For tensor products, the underlying filtrations follow the convolution rule, and the residues add:
    \[
    \Res_x(\nabla^{\lambda,V\otimes W}) = \Res_x(\nabla^{\lambda,V})\otimes\Id + \Id\otimes\Res_x(\nabla^{\lambda,W}).
    \]
    The positive Moy--Prasad parts strictly raise the convolution filtration.
    Thus adjustedness is preserved.  At ordinary logarithmic points,
    algebraic type is preserved because zero Higgs residue terms and trivial
    unipotent monodromy factors are preserved by the rigid tensor operations.
    Since the tensor unit is algebraic, the complementary log-type stratum is
    not a tensor subcategory.
    \end{proof}

    \begin{proposition}
    \label{prop:functoriality-parahoric-lambda}
    Let \(\varphi\colon G\to H\) be a homomorphism of connected complex
    reductive groups, and let \((E,\nabla^\lambda)\) be an adjusted parahoric
    \(G\)-\(\lambda\)-connection of type \(\theta\).  The Tannakian pushout
    \(E_H:=\varphi_*E\) has type \(\varphi_*\theta\), where
    \(\varphi_*\theta_x\) is the \(H\)-conjugacy class of the composite
    cocharacter.  For every \(W\in\Rep_{\mathbb C}(H)\),
    \[
    \operatorname{pdeg}_{\varphi_*\theta}E_H(W)
    =
    \operatorname{pdeg}_\theta E(\varphi^*W).
    \]
    After choosing a representative \(\widetilde\theta_x\) of \(\theta_x\)
    and a compatible local parahoric trivialization, the residue in the
    associated representation is
    \[
    \Res_x(\nabla^{\lambda,W}_{E_H})
    =
    \lambda\,d(\rho_W\circ\varphi)(\widetilde\theta_x)
    +d(\rho_W\circ\varphi)(N_x).
    \]
    In principal shorthand after the same choice,
    \[
    \lambda\widetilde\theta_x+N_x
    \longmapsto
    \lambda\varphi_*(\widetilde\theta_x)+d\varphi(N_x).
    \]
    Adjustedness and algebraic type at ordinary logarithmic points are
    preserved.  For \(\lambda=0\), log type is preserved at such a point
    precisely when \(d\varphi(N_x)\ne0\).  For \(\lambda\ne0\), it is
    preserved precisely when the unipotent local monodromy remains nontrivial
    after applying \(\varphi\).  Hence log type is preserved by a faithful
    representation or a closed immersion, but not by an arbitrary
    homomorphism.
    \end{proposition}

    \begin{proof}
    For \(W\in\Rep_{\mathbb C}(H)\), the periodic parabolic bundles
    \(E_H(W)\) and \(E(\varphi^*W)\) are canonically isomorphic, which gives
    the degree formula.  Moreover,
    \[
    d\varphi\bigl(\mathfrak g_{>0}(\widetilde\theta_x)\bigr)
    \subset
    \mathfrak h_{>0}(\varphi_*(\widetilde\theta_x)).
    \]
    This inclusion gives the residue formula and adjustedness.  Functoriality
    of local monodromy gives the stated algebraic/log qualifications when
    \(\lambda\ne0\); the case \(\lambda=0\) follows from the displayed
    positive term.
    \end{proof}

    \begin{proposition}
    \label{prop:parahoric-lambda-finite-pullback}
    Let \(f\colon(Y,D_Y)\to(X,D)\) be a finite ramified morphism of pointed
    smooth curves with \(D_Y=(f^{-1}D)_{\mathrm{red}}\), equipped with the
    induced parahoric pullback.  Let \((E,\nabla^\lambda)\) be an adjusted
    parahoric \(G\)-\(\lambda\)-connection of type \(\theta\).  If \(y\) lies
    over \(x\) with ramification index \(e_y\), then, for every
    \(V\in\Rep_{\mathbb C}(G)\),
    \[
    (f^*\theta)_y=e_y\theta_x,
    \qquad
    \Res_y(f^*\nabla^{\lambda,V})
    =e_y\Res_x(\nabla^{\lambda,V}).
    \]
    After choosing a representative \(\widetilde\theta_x\) and a compatible
    local parahoric trivialization, the principal shorthand is
    \[
    \lambda\widetilde\theta_x+N_x
    \longmapsto
    \lambda(e_y\widetilde\theta_x)+e_yN_x.
    \]
    Moreover,
    \[
    \operatorname{pdeg}_{f^*\theta}\bigl(f^*E(V)\bigr)
    =\deg(f)\,\operatorname{pdeg}_{\theta}E(V).
    \]
    Pullback preserves adjustedness and the algebraic/log stratification at
    ordinary logarithmic points.
    \end{proposition}

    \begin{proof}
    The logarithmic differential \(dz/z\) pulls back to \(e_y\,dw/w\), giving
    the residue formula
    \cite[Section~5.2]{AlfayaBiswas23pullbackanddirectimofparconandparHiggs}.
    The parahoric degree formula is
    \cite[Lemma~3.2(1)]{AlfayaBiswas23pullbackanddirectimofparconandparHiggs}.
    Since \(e_y>0\) and
    \(e_yN_x\in\mathfrak g_{>0}(e_y\widetilde\theta_x)\), adjustedness is
    preserved.  For \(\lambda=0\), multiplication by \(e_y\) preserves
    vanishing and nonvanishing of the positive term.  For \(\lambda\ne0\),
    normalized local monodromy satisfies \(T_y=T_x^{e_y}\); if its unipotent
    factor is \(\exp(M_x)\), then its pullback is \(\exp(e_yM_x)\), which is
    trivial exactly when \(M_x=0\).
    \end{proof}

    \begin{proposition}
    \label{prop:parahoric-lambda-galois-descent}
    Let \(p\colon\widetilde{\mathcal X}\to\mathcal X\) be a finite Galois
    cover of the corresponding log--orbi or root-stack curves, with Galois
    group \(\Delta\).  An adjusted parahoric \(G\)-\(\lambda\)-connection
    \(A\) of type \(p^*\theta\) on the cover descends effectively if it is
    equipped with connection-preserving isomorphisms
    \[
    \epsilon_\delta\colon\delta^*A\xrightarrow{\sim}A,
    \qquad \delta\in\Delta,
    \]
    satisfying
    \[
    \epsilon_{\delta\delta'}
    =
    \epsilon_{\delta'}\circ\delta'^*\epsilon_\delta
    \]
    and preserving the full facetwise periodic enhancements.  Conversely,
    pullback equips every adjusted object of type \(\theta\) on the base with
    such descent data.  This equivalence preserves the algebraic/log
    stratification at ordinary logarithmic points.
    \end{proposition}

    \begin{proof}
    Fpqc descent for torsors and their tensor-compatible connection morphisms
    is effective on the indicated covers.  The full facetwise periodic
    enhancements are tensor-compatible quasi-coherent lattice data and hence
    descend with their cocycle; the Leibniz, \(\lambda\)-, residue, and local
    type conditions are fpqc-local
    \cite[Tags 04US and 03O6]{StacksProject}.  At \(\lambda=0\), faithful base
    change detects vanishing and nonvanishing of the positive term.  At
    \(\lambda\ne0\), if \(y\) lies over \(x\) and \(e_y\) denotes the coarse
    ramification index, a local unipotent factor \(\exp(M_x)\) pulls back to
    \(\exp(e_yM_x)\), which is trivial exactly when \(M_x=0\).  Thus the
    algebraic/log stratification descends as well.
    \end{proof}

\subsubsection{The log--orbi/parahoric dictionary}

    We first specify the two groupoids that occur in the coarse-curve
    translation.  Let \(G\) be a connected complex reductive group, let
    \(T\subset G\) be a maximal torus with Weyl group \(W\), and fix
    \(\lambda\in\mathbb C\).  Write
    \[
    \mathcal C=(C,D_{\mathrm{orb}}^{\mathbf m},D_{\log}),
    \qquad D:=D_{\mathrm{orb}}\cup D_{\log},
    \]
    and fix an admissible system of rational local types
    \[
    \theta=(\theta_x)_{x\in D},
    \qquad \theta_x\in X_*(T)_{\mathbb Q}/W.
    \]
    Here admissibility means compatibility with the orbifold orders: for each
    \(x\in D_{\mathrm{orb}}\), there exists (equivalently, every)
    representative \(\widetilde\theta_x\in X_*(T)_{\mathbb Q}\) such that
    \(m_x\widetilde\theta_x\in X_*(T)\).  Define
    \[
    \tau_{\widetilde\theta_x}\colon\mu_{m_x}\longrightarrow G,
    \qquad
    \tau_{\widetilde\theta_x}(\zeta)
    :=(m_x\widetilde\theta_x)(\zeta),
    \]
    and let \([\tau_{\theta_x}]\) be its \(G\)-conjugacy class, which is
    independent of \(\widetilde\theta_x\).

    Let \(\mathscr L_G^\lambda(\mathcal C,D;\theta)\) be the groupoid of
    principal logarithmic \(G\)-\(\lambda\)-connections on \(\mathcal C\)
    equipped, in every associated representation and compatibly with tensor
    operations, with the facetwise periodic local enhancements of full type
    \(\theta\) from Definition~\ref{def:parahoric-structure-point}, subject to
    the following conditions:
    \begin{enumerate}
    \item at \(x\in D_{\mathrm{orb}}\), the connection is regular on an
    \'etale-local root-stack chart and its inertia action has type
    \([\tau_{\theta_x}]\); we call this the \emph{orbifold algebraic-type}
    condition.  In every representation, the normalized filtered object
    underlying the chosen periodic enhancement is required to be the
    parabolic object induced from this equivariant germ by the root-stack
    correspondence, with its induced \(\lambda\)-connection;
    \item at \(x\in D_{\log}\), after choosing a representative
    \(\widetilde\theta_x\) and a compatible local parahoric trivialization,
    for every algebraic representation \(\rho_V\colon G\to\GL(V)\) there is
    an endomorphism \(N_{V,x}\in\End(E(V)_x)\) such that
    \[
    \operatorname{Res}_x(\nabla^{\lambda,V})
    =\lambda\,d\rho_V(\widetilde\theta_x)+N_{V,x},
    \qquad
    N_{V,x}\bigl(F^{(x)}_{\ge a}E(V)_x\bigr)
    \subset F^{(x)}_{>a}E(V)_x
    \quad\text{for every }a\in\mathbb Q.
    \]
    \end{enumerate}
    Let \(\mathscr P_G^\lambda(C,D;\theta)\) be the groupoid of adjusted
    parahoric \(G\)-\(\lambda\)-connections of type \(\theta\) on
    \((C,D)\).  Arrows in both groupoids are connection-preserving
    isomorphisms of principal bundles that preserve the full periodic local
    data.

    More generally, for an auxiliary reduced divisor
    \(D^{\mathrm{aux}}\supset D\), we use the same definitions for
    \(\mathscr L_G^\lambda(\mathcal C,D^{\mathrm{aux}};\theta)\) and
    \(\mathscr P_G^\lambda(C,D^{\mathrm{aux}};\theta)\).  At a point
    \(x\in D^{\mathrm{aux}}\setminus D\), the type \(\theta_x\) is required to
    be integral, the principal connection is regular, and the periodic
    enhancement is required to correspond to its coarse filtered lattice
    under the order-one root-stack/parabolic construction.  Thus such points
    carry trivial inertia but retain their integral lattice shifts.  This
    convention is used below for pullback along finite \'etale morphisms.

    \begin{proposition}[Log--orbi/parahoric dictionary]
    \label{prop:log-orbi-parahoric-dictionary}
    The representationwise coarse-model construction induces an equivalence
    of groupoids
    \[
    \operatorname{Coarse}_{\theta}\colon
    \mathscr L_G^\lambda(\mathcal C,D;\theta)
    \xrightarrow{\ \sim\ }
    \mathscr P_G^\lambda(C,D;\theta).
    \]
    \end{proposition}

    \begin{proof}
    Fix \(x\in D_{\mathrm{orb}}\), put \(m=m_x\), and choose an \'etale
    neighborhood \(U=\operatorname{Spec}A\to C\), a point \(u\in U\) above
    \(x\), and a parameter \(z\) cutting out \(u\).  The corresponding
    algebraic root-stack chart is
    \[
    \left[
    \operatorname{Spec}\bigl(A[w]/(w^m-z)\bigr)/\mu_m
    \right],
    \qquad \zeta\cdot w=\zeta w.
    \]
    For each \(V\in\Rep_{\mathbb C}(G)\), the root-stack/parabolic
    correspondence gives a natural tensor equivalence between regular
    equivariant vector \(\lambda\)-connections on this chart and strongly
    parabolic vector \(\lambda\)-connections on \((U,u)\)
    \cite[Theorem~4.31 and Section~4.6]
    {BorneAmine23paraconneandstackofroots}.  In the decreasing-filtration
    convention used here, the coarse residue acts on
    \(\operatorname{gr}^{(u)}_aE_U(V)_u\) as
    \(\lambda a\operatorname{id}\).
    Consequently, after choosing \(\widetilde\theta_x\), it has the form
    \[
    \operatorname{Res}_u(\nabla_U^{\lambda,V})
    =\lambda\,d\rho_V(\widetilde\theta_x)+N_{V,u},
    \qquad
    N_{V,u}\bigl(F^{(u)}_{\ge a}E_U(V)_u\bigr)
    \subset F^{(u)}_{>a}E_U(V)_u
    \quad\text{for every }a\in\mathbb Q.
    \]
    The inertia action \([\tau_{\theta_x}]\) records the fractional part of
    \(\theta_x\), while the prescribed periodic enhancement retains its
    integral lattice shifts.  Hence the coarse object has full type
    \(\theta_x\), and the displayed condition places no vanishing requirement
    on \(N_{V,u}\), equivalently on the residue term at \(x\).

    Now let \(x\in D_{\log}\) and choose a pointed \'etale neighborhood
    \((U,u)\to(C,x)\).  The stack chart is the ordinary pointed neighborhood
    \((U,u)\), and the coarse-model functor retains the specified
    \(\theta_x\)-periodic filtered extension.  The residue formula in the
    definition of \(\mathscr L_G^\lambda(\mathcal C,D;\theta)\) is therefore
    exactly the adjustedness condition of
    Definition~\ref{def:adjusted-algebraic-log}; it is not a consequence of
    the existence of the logarithmic generator alone.  The assertions about
    the algebraic and log strata now follow directly from that definition.

    The local vector equivalences are natural in \(V\), exact strong
    symmetric monoidal, and compatible with restriction.  They therefore
    carry the exact faithful tensor functor associated with a principal
    object on \(\mathcal C\) to a functor with values in
    \(\mathbf{\Lambda}^\lambda_{\mathrm{adj}}(C,D)\).
    Proposition~\ref{prop:tannakian-parahoric-connections} reconstructs from
    it a parahoric principal object and reconstructs tensor-natural arrows.
    Applying the inverse vector equivalence in every representation gives the
    inverse principal functor; Tannakian full faithfulness supplies the unit
    and counit isomorphisms.

    On \(C\setminus D\), both local functors are the identity on ordinary
    principal \(G\)-\(\lambda\)-connections.  Their restrictions, inverses,
    units, and counits therefore agree on the punctured overlaps.  Principal
    torsors descend effectively, and a logarithmic \(\lambda\)-connection is
    an \(\mathcal O\)-linear \(\lambda\)-splitting of the logarithmic Atiyah
    sequence; the splitting morphism descends as quasi-coherent data, while
    the Leibniz and residue identities are fpqc-local equalities
    \cite[Tags 04US and 03O6]{StacksProject}.  Thus the local functors and
    their natural transformations glue to the asserted global equivalence.

    \end{proof}

    \begin{corollary}
    \label{cor:log-orbi-parahoric-local-behavior}
    Under the equivalence of
    Proposition~\ref{prop:log-orbi-parahoric-dictionary}, the following hold.
    \begin{enumerate}
    \item If \(x\in D_{\mathrm{orb}}\), the associated coarse connection is
    strongly parabolic in every representation.  The term \(N_{V,x}\) in its
    adjusted residue decomposition is strictly filtration-raising and is not
    required to vanish.  When \(\lambda=0\), regularity on the orbifold chart
    does not force the coarse Higgs residue to vanish.
    \item If \(x\in D_{\log}\), the equivalence identifies the full adjusted
    loci and the algebraic/log stratifications of
    Definition~\ref{def:adjusted-algebraic-log}.  For \(\lambda=0\), the two
    strata are defined by \(N_{V,x}=0\) for every \(V\) and
    \(N_{V,x}\ne0\) for some \(V\), respectively; for \(\lambda\ne0\), they
    are defined by trivial and nontrivial unipotent local monodromy of the
    normalized flat germ \(\lambda^{-1}\nabla^\lambda\).  In particular, when
    \(\lambda=0\), a logarithmic point may have a nonzero nilpotent Higgs
    residue.
    \end{enumerate}
    \end{corollary}

    \begin{proof}
    The first assertion is the root-stack residue calculation in the proof of
    Proposition~\ref{prop:log-orbi-parahoric-dictionary}.  The second follows
    from the ordinary pointed chart there and
    Definition~\ref{def:adjusted-algebraic-log}.
    \end{proof}

    \begin{corollary}
    \label{cor:log-orbi-parahoric-functoriality}
    Coarse formation commutes with associated representations and their rigid
    tensor operations, and it is compatible with external products.  For a
    homomorphism \(\varphi\colon G\to H\) of connected complex reductive
    groups, there is a canonical natural isomorphism
    \[
    \operatorname{Coarse}_{\varphi_*\theta}\bigl(\varphi_*(-)\bigr)
    \simeq
    \varphi_*\bigl(\operatorname{Coarse}_{\theta}(-)\bigr).
    \]
    If \(f\colon\mathcal C'\to\mathcal C\) is finite \'etale with coarse map
    \(\bar f\colon C'\to C\), set
    \[
    D_f:=\bigl(\bar f^{-1}D\bigr)_{\mathrm{red}},
    \qquad
    (f^*\theta)_y:=e_y\theta_{\bar f(y)}
    \quad(y\in D_f),
    \]
    where \(e_y\) is the coarse ramification index.  Pullback, indexed by
    \(D_f\) also at order-one preimages of orbifold points, satisfies the
    canonical natural isomorphism
    \[
    \operatorname{Coarse}_{f^*\theta}\bigl(f^*(-)\bigr)
    \simeq
    \bar f^*\bigl(\operatorname{Coarse}_{\theta}(-)\bigr).
    \]
    If \(p\colon\widetilde{\mathcal C}\to\mathcal C\) is finite Galois, the
    equivalence identifies descent data of type \(p^*\theta\), indexed by
    \(D_p=(\bar p^{-1}D)_{\mathrm{red}}\), on the two sides.  Preservation of
    the log-type stratum under change of structure group is subject to the
    nonannihilation conditions in
    Proposition~\ref{prop:functoriality-parahoric-lambda}.
    \end{corollary}

    \begin{proof}
    The vector root-stack construction is defined by equivariant pullback and
    descent, so it commutes with associated representations, external
    products, structure-group change, and base change.  For finite \'etale
    pullback, retaining all points of \(D_f\) records the integral type at
    order-one preimages; the order-one local construction above is the
    \(m=1\) case of the same vector correspondence.  The four compatibilities
    follow, respectively, from
    Propositions~\ref{prop:parahoric-lambda-tensor-operations},~\ref{prop:functoriality-parahoric-lambda},~\ref{prop:parahoric-lambda-finite-pullback}, and~\ref{prop:parahoric-lambda-galois-descent}.
    The identity and composition coherence of the resulting natural
    isomorphisms is inherited from equivariant pullback.
    \end{proof}

\subsection{The canonical maximal \texorpdfstring{\(\PSL_2\)}{PSL2}-Higgs object}\label{subsec:canonical-maximal-psl2}

    We now construct the distinguished Dolbeault object which will later be transported, by non-abelian Hodge theory and Riemann--Hilbert, to the uniformizing representation. Working with \(\PSL_2\), rather than \(\SL_2\), is essential: the usual square-root construction produces half-integral local weights on the \(\SL_2\)-side, while after pushout to \(\PSL_2\) the orbifold denominators match the given log--orbi structure and the isomorphism class is independent of the square-root choices.

    Let \(\mathcal C=(C,D_{\mathrm{orb}}^{\mathbf m},D_{\log})\) be a hyperbolic log--orbi curve. Its log--orbi canonical bundle is
    \[
    \omega_{\mathcal C}:=\Omega^1_{\mathcal C}(\log D_{\log}).
    \]
    For \(x\in D_{\mathrm{orb}}\), put
    \[
    m_x:=\lvert\operatorname{Aut}_{\mathcal C}(x)\rvert.
    \]
    On the coarse curve \(C\), this line is represented by the rational divisor
    \[
    K_C^{\mathrm{log\text{-}orb}}
    =
    K_C
    +
    \sum_{x\in D_{\mathrm{orb}}}
    \left(1-\frac1{m_x}\right)x
    +
    \sum_{y\in D_{\log}}y.
    \]
    This is the hybrid canonical divisor which records both stacky orbifold contributions and logarithmic boundary contributions. This combines the usual canonical divisor formula for stacky curves with the ordinary logarithmic contribution along \(D_{\log}\) \cite{ThurstonGTMTheGeometryandTopologyofThreeManifolds, JohnDavid22thecanonicalringofastackycurve}. We write
    \[
    \kappa_x
    =
    \begin{cases}
    1-\dfrac1{m_x}, & x\in D_{\mathrm{orb}},\\[6pt]
    1, & x\in D_{\log},
    \end{cases}
    \]
    for the local coefficient of \(\omega_{\mathcal C}\) at a marked point \(x\). Thus \(\kappa_x\) is the local coefficient of the log--orbi canonical bundle at \(x\), before any square root is chosen. We denote the associated canonical \(\PSL_2\)-local type by
    \[
    \theta_{\mathcal C,x} := \kappa_x\varpi^\vee, \qquad \theta_{\mathcal C} := (\theta_{\mathcal C,x})_{x\in D_{\mathrm{orb}}\cup D_{\log}}.
    \]

    \begin{theorem}
    \label{thm:canonical-maximal-psl2-higgs}
    There exists a canonical principal \(\PSL_2\)-Higgs object
    \[
    (\mathcal U_{\mathcal C},\vartheta_{\mathcal C})
    \]
    on \(\mathcal C\), of canonical local type
    \[
    \theta_{\mathcal C,x}=\kappa_x\varpi^\vee.
    \]
    It is characterized \'etale-locally as the pushout along \(\SL_2\longrightarrow \PSL_2\) of
    \[
    \left(
    \Theta\oplus\Theta^{-1},
    \begin{pmatrix}
    0&0\\
    1&0
    \end{pmatrix}
    \right),
    \qquad
    \Theta^{\otimes2}\simeq\omega_{\mathcal C}.
    \]
    It is independent of all local choices of \(\Theta\). At an orbifold
    point of order \(m_x\), its fractional local type has denominator
    \(m_x\); the intrinsic Higgs field is regular on the stack chart, although
    its coarse strongly parabolic residue is nonzero.  At a logarithmic point,
    its fractional \(\PSL_2\)-type is trivial and its logarithmic nilpotent
    residue is nonzero.
    \end{theorem}

    \begin{proof}
    Let \(T_{\SL}\subset \SL_2\) be the diagonal torus with standard coroot
    \[
    \alpha^\vee(t) = \begin{pmatrix} t & 0 \\ 0 & t^{-1} \end{pmatrix},
    \]
    and let \(T_{\mathrm{ad}}\subset \PSL_2\) be its image. Inside the common rational cocharacter space, we have
    \[
    X_*(T_{\SL})=\mathbb Z\alpha^\vee, \qquad X_*(T_{\mathrm{ad}})=\mathbb Z\varpi^\vee, \qquad \alpha^\vee=2\varpi^\vee.
    \]
    Choose an \'etale cover \(\{U_i\to\mathcal C\}\) on which
    \(\omega_{\mathcal C}\) admits square roots
    \(\Theta_i^{\otimes2}\simeq\omega_{\mathcal C}|_{U_i}\).  On \(U_i\),
    form the local maximal \(\SL_2\)-Higgs object
    \[
    (E_i,\vartheta_i)
    =
    \left(
    \Theta_i\oplus\Theta_i^{-1},
    \begin{pmatrix}0&0\\1&0\end{pmatrix}
    \right),
    \]
    where the lower-left entry is the tautological isomorphism
    \(\Theta_i\simeq\Theta_i^{-1}\otimes\omega_{\mathcal C}\).  Since the
    local coefficient of \(\omega_{\mathcal C}\) at \(x\) is \(\kappa_x\),
    the local types before and after pushout are
    \[
    \theta_{\SL,x}=\frac{\kappa_x}{2}\alpha^\vee,
    \qquad
    \alpha^\vee=2\varpi^\vee,
    \qquad
    \theta_{\PSL,x}=\kappa_x\varpi^\vee.
    \]

    If \(x\in D_{\mathrm{orb}}\), then
    \[
    \theta_{\PSL,x}
    =\left(1-\frac{1}{m_x}\right)\varpi^\vee,
    \]
    whose fractional class has exact denominator \(m_x\), since
    \(\varpi^\vee\) is primitive and \(\gcd(m_x-1,m_x)=1\).  If
    \(x\in D_{\log}\), then \(\theta_{\PSL,x}=\varpi^\vee\) is integral, so
    its fractional class is trivial, whereas
    \[
    \mathfrak g_{>0}(\varpi^\vee)=\mathfrak g_1(\varpi^\vee)\neq0.
    \]

    The full local parameter is equally explicit.  On an orbifold chart
    \([\Delta_w/\mu_m]\), with \(z=w^m\), choose a projective root vector
    \(F\) on which inertia acts by \(\zeta^{-1}\), and put
    \(\widetilde F=wF\).  Then \(\widetilde F\) is invariant and
    \[
    F\,dw=\frac1m\,\widetilde F\,\frac{dz}{z}.
    \]
    Thus the stack Higgs field is regular, while the class of
    \(m^{-1}\widetilde F\) in the coarse parahoric fiber is a nonzero,
    strictly filtration-raising residue and hence has zero graded Levi image.
    At a logarithmic
    point the local field is \(F\,dz/z\), with nonzero nilpotent residue
    \(F\).  After a Weyl-compatible choice of representative these are the
    positive pieces associated with the displayed local types.

    On an overlap \(U_{ij}\), put
    \[
    L_{ij}:=\Theta_j\otimes\Theta_i^{-1},
    \qquad
    L_{ij}^{\otimes2}\simeq\mathcal O_{U_{ij}}.
    \]
    Thus the two local \(\SL_2\)-Higgs objects differ by the central
    \(\mu_2\)-twist defined by \(L_{ij}\).  Pushout along
    \(\SL_2\to\PSL_2\) kills this twist; the same applies to the discrepancy
    on triple overlaps.  The pushed-out local objects therefore descend to a
    global principal \(\PSL_2\)-Higgs object.  Changing the cover, square
    roots, or comparison maps changes only the intermediate central twists,
    so the descended object is independent of these choices up to canonical
    isomorphism.
    \end{proof}

    \begin{proposition}
    \label{prop:maximal-psl2-functoriality}
    The Higgs object \((\mathcal U_{\mathcal C},\vartheta_{\mathcal C})\) is maximal: \'etale-locally, for any choice of \(\Theta^{\otimes2}\simeq\omega_{\mathcal C}\), the Higgs field induces the isomorphism
    \[
    \Theta \xrightarrow{\sim} \Theta^{-1}\otimes\omega_{\mathcal C}.
    \]
    For any finite \'etale morphism \(f:\mathcal C_1\longrightarrow\mathcal C_2\) of hyperbolic log--orbi curves, there is a canonical isomorphism
    \[
    \alpha_f:
    f^*(\mathcal U_{\mathcal C_2},\vartheta_{\mathcal C_2})
    \xrightarrow{\sim}
    (\mathcal U_{\mathcal C_1},\vartheta_{\mathcal C_1}).
    \]
    Under the standard pullback identifications, these isomorphisms satisfy
    \[
    \alpha_{\mathrm{id}_{\mathcal C}}=\mathrm{id},
    \qquad
    \alpha_{g\circ f}=\alpha_f\circ f^*\alpha_g
    \]
    for every composable pair
    \(\mathcal C_1\xrightarrow f\mathcal C_2\xrightarrow g\mathcal C_3\).
    If \(f\) is Galois with group \(\Gamma\), the isomorphisms
    \[
    \alpha_\gamma:
    \gamma^*(\mathcal U_{\mathcal C_1},\vartheta_{\mathcal C_1})
    \xrightarrow{\sim}
    (\mathcal U_{\mathcal C_1},\vartheta_{\mathcal C_1}),
    \qquad \gamma\in\Gamma,
    \]
    form a Galois descent datum whose effective descent is
    \((\mathcal U_{\mathcal C_2},\vartheta_{\mathcal C_2})\).
    \end{proposition}

    \begin{proof}
    Maximality is the tautological lower-left isomorphism in the local
    \(\SL_2\)-model.  By Definition~\ref{def:etale-morphism-log-orbi}, finite
    \'etale pullback gives a canonical, composition-compatible isomorphism
    \(f^*\omega_{\mathcal C_2}\simeq\omega_{\mathcal C_1}\).  Hence local
    square roots and their maximal \(\SL_2\)-models pull back compatibly.
    The overlap comparisons and their independence after
    \(\SL_2\to\PSL_2\) are those of
    Theorem~\ref{thm:canonical-maximal-psl2-higgs}, and yield \(\alpha_f\).
    Their construction gives the identity and composition formulas.

    For a Galois cover, composition coherence is the cocycle condition for
    \((\alpha_\gamma)_{\gamma\in\Gamma}\).  Effective descent for principal
    torsors and their Higgs sections on the lisse--\'etale site then identifies
    the descended object with the one constructed on \(\mathcal C_2\)
    \cite[Tags 04US and 03O6]{StacksProject}.
    \end{proof}

\section{Tannakian realization for log--orbi curves}
\label{sec:tannakian-nah-rh}

    Throughout this section \(G\) is a connected complex reductive group.  We
    write a corresponding realization datum uniformly as
    \[
    \boldsymbol\eta
    =
    (\boldsymbol\theta,\boldsymbol\delta,\boldsymbol n)
    =
    \bigl((\theta_x,\delta_x,n_x)\bigr)_{x\in D}.
    \]
    Here \(\boldsymbol\theta\) and \(\boldsymbol\delta\) are the admissible
    rational full Dolbeault and de Rham local-parameter systems in the
    zero-Betti-weight slice, and \(\boldsymbol n\) is the corresponding system
    of full local monodromy conjugacy classes.  The Betti filtration is
    trivial and is not part of the object; in particular, trivial filtered
    structure does not mean trivial local monodromy.  In a residue formula,
    \(\theta_x^{\mathrm{coch}}\) or \(\delta_x^{\mathrm{coch}}\) denotes a
    representative of the cocharacter component of the corresponding full
    local datum.  For arbitrary \(G\), \(\boldsymbol\eta\) is part of the
    input: a log--orbi curve does not determine a canonical \(G\)-type.  Under
    a representation \(\rho\), \(\boldsymbol\theta\) and
    \(\boldsymbol\delta\) push forward in full, while
    \([T_x]\in n_x\) maps to \([\rho(T_x)]\).  The canonical datum enters only
    for the distinguished \(\PSL_2\)-object constructed in
    Theorem~\ref{thm:canonical-maximal-psl2-higgs}.

    The rigid exact tensor categories used below are the ambient
    vector-valued Dolbeault and de Rham categories in which the rank and
    rational local parameter are allowed to vary, together with the ordinary
    Betti category of local systems in which the local monodromy data
    vary.  A fixed-datum piece is
    a full subcategory of the corresponding ambient category, but is not
    itself a tensor category: tensor operations generally change the type.
    We first record the vector-valued NAH/RH input, then pass to principal
    objects by Tannakian reconstruction, and finally state the functorial
    principal realization theorem.

    \subsection{Vector-valued NAH/RH input for log--orbi curves}
    \label{subsec:vector-valued-input}

    We use Simpson's vector-valued tame correspondence on a punctured curve,
    together with its local residue transformation
    \cite[Main Theorem; Synopsis; \S5; Theorem~7; Remark]
    {Sim90harmobdonnoncomcurves}.  Its compatibility with the standard tensor
    operations is established in
    \cite[Proposition~3.1 and Theorem~2; Lemma~3.2; Proposition~3.3]
    {Sim90harmobdonnoncomcurves}.  For regular-singular
    Riemann--Hilbert we use
    \cite[Part~I, Theorem~2.17; Part~II, Theorem~5.9]{deligne1970}.

    Let \(\mathcal C\) be a log--orbi curve with coarse pointed curve
    \(
    (C,D)
    \),
    where \(D:=D_{\mathrm{orb}}\cup D_{\log}\).
    For rank \(r\), fix a compatible system
    \[
    \boldsymbol\eta
    =
    (\boldsymbol\theta,\boldsymbol\delta,\boldsymbol n)
    =
    \bigl((\theta_x,\delta_x,n_x)\bigr)_{x\in D}
    \]
    of full rational vector local data.  Thus \(\theta_x\) is the Dolbeault
    datum, \(\delta_x\) is its de Rham transform, and \(n_x\) is the resulting
    full local monodromy conjugacy class.  The first two include the weighted
    filtration and residue data, while \(\delta_x\) also retains the chosen
    regular-singular extension class.  Compatibility includes zero transformed
    Betti weight and the adjusted condition at ordinary logarithmic points.

    On a fine residue-diagram summand, write the Higgs filtration jump as
    \(\alpha\in\Q\) and its graded-residue eigenvalue as \(b+ci\), where
    \(b,c\in\R\).  Simpson's local table is
    \[
    \begin{array}{@{}lccc@{}}
    \toprule
    & \text{Dolbeault }(E,\vartheta)
    & \text{de Rham }(V,\nabla)
    & \text{Betti }L\\
    \midrule
    \text{filtration jump}
    & \alpha & \alpha-2b & \beta=-2b\\
    \text{residue/monodromy eigenvalue}
    & b+ci & \alpha+2ci & \exp(-2\pi i\alpha+4\pi c)\\
    \bottomrule
    \end{array}
    \]
    \cite[Synopsis; \S5]{Sim90harmobdonnoncomcurves}.  Adjustedness makes the
    induced Higgs residue on each associated-graded summand zero, so
    \(b+ci=0\), hence \(b=c=0\).  Thus \(\alpha\) is the remaining semisimple
    local label of \((\theta_x,\delta_x,n_x)\): the Dolbeault jump is
    \(\alpha\) with residue eigenvalue \(0\), the de Rham jump is \(\alpha\)
    with residue eigenvalue \(\alpha\), and the Betti monodromy eigenvalue is
    \(\exp(-2\pi i\alpha)\).  The nilpotent Jordan data are carried separately
    by Simpson's standard rank-two \(\SL_2\)-model, its symmetric powers,
    rank-one tensor products, and direct sums
    \cite[\S5]{Sim90harmobdonnoncomcurves}.
    For the compatible full data fixed here, choose
    \(\widetilde\theta_x\) so that, on a length-\(\ell\) block, the model
    nilpotent \(e_x\) has a Jordan chain \(e_xv_j=v_{j+1}\) in which \(v_j\)
    has periodic weight \(\alpha+j\).  Then \(e_x\) has
    \(\widetilde\theta_x\)-degree \(+1\), so
    \(e_x(F_{\ge a})\subset F_{\ge a+1}\subset F_{>a}\); symmetric powers,
    rank-one tensor products, and direct sums preserve this condition.
    Compatibility requires \(\delta_x\) to be the transform of this full
    periodic datum, including its resonant extension data, and \(n_x\) to be
    its full monodromy class; reduction modulo \(\mathbb Z\) identifies only
    the semisimple eigenvalue.  Thus the theorem concerns compatible adjusted
    triples, not independently prescribed extensions or lattices.

    We denote by
    \[
    \Higgs(\mathcal C,\boldsymbol\theta)_{\mathrm{poly},0},\qquad
    \MIC(\mathcal C,\boldsymbol\delta)_{\mathrm{ss}},\qquad
    \Loc(\mathcal C,\boldsymbol n)_{\mathrm{ss}}
    \]
    the full fixed-datum pieces of the corresponding ambient categories.  The
    last is the category of ordinary semisimple local systems with prescribed
    full local monodromy; no filtration is included.  After choosing
    \(c\in C^\circ\), let
    \[
    \Rep_r^{\mathrm{ss}}
    \bigl(\pi_1^{\mathrm{orb}}(\mathcal C,c),\boldsymbol n\bigr)
    \]
    denote the category of globally semisimple rank-\(r\) representations with
    local monodromy classes \(\boldsymbol n\), with intertwiners as morphisms.
    These fixed-datum pieces are not individually tensor categories.

    \begin{theorem}[Vector-valued parahoric NAH/RH input]
    \label{thm:vector-parahoric-nah-rh}
    Let \(\mathcal C\) be a log--orbi curve, choose \(c\in C^\circ\), and fix
    \(r\ge1\) and a compatible realization datum
    \[
    \boldsymbol\eta
    =
    (\boldsymbol\theta,\boldsymbol\delta,\boldsymbol n)
    =
    \bigl((\theta_x,\delta_x,n_x)\bigr)_{x\in D}
    \]
    as above.  Then there are natural equivalences of categories
    \[
    \Higgs(\mathcal C,\boldsymbol\theta)_{\mathrm{poly},0}
    \xrightarrow[\sim]{\ \mathrm{NAH}\ }
    \MIC(\mathcal C,\boldsymbol\delta)_{\mathrm{ss}}
    \xrightarrow[\sim]{\ \mathrm{RH}\ }
    \Loc(\mathcal C,\boldsymbol n)_{\mathrm{ss}}
    \xrightarrow[\sim]{\ \mathrm{Mon}_c\ }
    \Rep_r^{\mathrm{ss}}
    \bigl(\pi_1^{\mathrm{orb}}(\mathcal C,c),\boldsymbol n\bigr).
    \]
    At an orbifold point \(x\) of order \(m_x\), regular equivariance is
    automatic for objects on \(\mathcal C\), and \(T_x^{m_x}=1\) for
    \(T_x\in n_x\).  At an ordinary logarithmic point, the full adjusted datum
    and its corresponding monodromy class \(n_x\) are retained.
    \end{theorem}

    \begin{proof}
    The preceding local calculation lets Simpson's Main Theorem restrict to
    the fixed \((\boldsymbol\theta,\boldsymbol\delta)\)-pieces.  His
    Theorem~7 identifies the associated graded pieces and residues; in the
    present zero-Betti-weight case, the following remark recovers the full
    local monodromy conjugacy class
    \cite[Theorem~7; Remark]{Sim90harmobdonnoncomcurves}.  Forgetting the
    filtered-local-system presentation then leaves the ordinary local system
    with full monodromy datum \(\boldsymbol n\).  On \(C^\circ\), ordinary
    regular-singular Riemann--Hilbert gives the underlying de Rham/local-system
    equivalence
    \cite[Part~I, Theorem~2.17; Part~II, Theorem~5.9]{deligne1970};
    the prescribed compatible data restrict it to the stated fixed-datum
    pieces.

    The dictionary of
    Proposition~\ref{prop:log-orbi-parahoric-dictionary} transports these fixed
    local data to \(\mathcal C\) and, at an orbifold point, gives
    \(T_x^{m_x}=1\) from regular equivariance.  Hence the monodromy
    representation of \(\pi_1(C^\circ,c)\) kills the orbifold relations and
    factors through \(\pi_1^{\mathrm{orb}}(\mathcal C,c)\).  Conversely, a
    representation of the orbifold group has the required finite inertia, and
    the same dictionary recovers the regular equivariant germ.  No finite-order
    relation is imposed at an ordinary logarithmic point.  Taking monodromy at
    \(c\) and its inverse local-system construction gives the last equivalence.
    \end{proof}

    \begin{proposition}
    \label{prop:vector-tensor-compatibility}
    As the compatible full realization data vary, the fixed-datum equivalences
    in
    Theorem~\ref{thm:vector-parahoric-nah-rh} assemble into rigid symmetric
    monoidal equivalences of the corresponding ambient varying-data
    categories.  Write
    \[
    \xi=(\boldsymbol\theta,\boldsymbol\delta,\boldsymbol n),
    \qquad
    \xi'=(\boldsymbol\theta',\boldsymbol\delta',\boldsymbol n')
    \]
    for compatible full realization data.  Direct sum, tensor product, dual,
    and internal Hom carry the induced compatible data
    \[
    \xi\oplus\xi',\qquad
    \xi\otimes\xi',\qquad
    \xi^\vee,\qquad
    \mathcal Hom(\xi,\xi'),
    \]
    and a semisimple direct summand carries its induced compatible summand
    datum.  The Dolbeault/de Rham entries are full periodic data, with the de
    Rham entry retaining the induced resonant extension data; the Betti entry
    is the induced full monodromy conjugacy datum, with no filtration.  These
    operations generally change \(\xi\), so no individual fixed-datum piece is
    asserted to be monoidal.  They preserve regular equivariance at orbifold
    points and adjustedness at ordinary logarithmic points.  Within each
    realization separately, the algebraic-type stratum is tensor-stable; its
    complementary log-type stratum is not a tensor subcategory and is not
    stable under arbitrary semisimple direct summands.
    \end{proposition}

    \begin{proof}
    Simpson's realization functors are compatible with direct sums, tensor
    products, and duals; compatibility with internal Homs follows from
    \(\mathcal Hom(V,W)=V^\vee\otimes W\)
    \cite[Proposition~3.1 and Theorem~2; Lemma~3.2; Proposition~3.3]
    {Sim90harmobdonnoncomcurves}.  On \(C^\circ\), Riemann--Hilbert carries the
    same operations through its standard quasi-inverse functors
    \cite[Part~I, Theorem~2.17]{deligne1970}.  Their canonical comparison maps
    give rigid symmetric monoidal equivalences on the ambient varying-data
    categories.

    The full-data, adjustedness, and summand assertions use the filtered
    vector-space calculation specific to the present formulation.  The
    tensor-product filtration is
    \[
    F^{V\otimes W}_{\ge c}
    =
    \sum_{a+b\ge c}F^V_{\ge a}\otimes F^W_{\ge b}.
    \]
    At a logarithmic point, write \(A_V\) and \(A_W\) for the strictly
    filtration-raising residue remainders.  Then
    \(A_{V\otimes W}=A_V\otimes\operatorname{id}+
    \operatorname{id}\otimes A_W\) sends
    \(F^{V\otimes W}_{\ge c}\) into \(F^{V\otimes W}_{>c}\).  For duals and
    internal Homs,
    \[
    A_{V^\vee}(\phi)=-\phi\circ A_V,
    \qquad
    A_{\mathcal Hom(V,W)}(f)=A_W\circ f-f\circ A_V;
    \]
    the standard dual and internal-Hom filtrations show that both remainders
    strictly raise the decreasing filtration.  Direct sums and invariant
    filtered direct summands are immediate.  Thus adjustedness is preserved,
    while regular equivariant objects at orbifold points are closed under all
    these operations.

    The compatible de Rham entry is the induced full transform, including its
    resonant extension data, and the Betti entry is the induced full monodromy
    class.  This uses monoidal functoriality; it does not identify a strictly
    filtration-raising residue remainder with the logarithm of the Betti
    unipotent factor.
    By contrast, within each realization the tensor unit is algebraic, and a
    log-type direct sum may have an algebraic-type summand, which proves the
    stated qualifications for the log-type strata.
    \end{proof}
\subsection{Principal objects by Tannakian reconstruction}
\label{subsec:principal-tannakian-reconstruction}

    We reconstruct principal \(G\)-objects intrinsically from exact faithful
    tensor functors out of \(\Rep_{\mathbb C}(G)\).

    We write
    \[
    \Higgs(\mathcal C),\qquad \MIC(\mathcal C),\qquad \Loc(\mathcal C)
    \]
    for the ambient Dolbeault/de Rham vector categories and the ordinary Betti
    category of local systems.  We denote by
    \(\Higgs_G(\mathcal C,\boldsymbol\theta)\) the groupoid of principal
    log--orbi \(G\)-Higgs objects of Dolbeault type
    \(\boldsymbol\theta\), and by
    \(\MIC_G(\mathcal C,\boldsymbol\delta)\) the corresponding groupoid of
    principal log--orbi logarithmic \(G\)-connections of de Rham type
    \(\boldsymbol\delta\).  Their auxiliary coarse local types are
    understood through
    Proposition~\ref{prop:log-orbi-parahoric-dictionary}.
    If \(V\in\Rep_{\mathbb C}(G)\), with representation
    \(\rho_V:G\to\GL(V)\), set
    \[
    \boldsymbol\eta_V
    :=
    (\boldsymbol\theta_V,\boldsymbol\delta_V,\boldsymbol n_V)
    :=
    (\rho_{V,*}\boldsymbol\theta,
    \rho_{V,*}\boldsymbol\delta,
    \rho_{V,*}\boldsymbol n).
    \]
    Here the first two pushforwards retain the full weighted and graded local
    parameters.  On the Betti side,
    \(\boldsymbol n_V\) sends \([T_x]\) to \([\rho_V(T_x)]\); its
    filtration remains trivial.

    On the Betti side, we denote by
    \(
    \Rep_G\bigl(\pi_1^{\mathrm{orb}}(\mathcal C),\boldsymbol n\bigr)
    \)
    the groupoid of unbased ordinary \(G\)-local systems with local monodromy
    datum \(\boldsymbol n\), tested after every
    \(V\in\Rep_{\mathbb C}(G)\).
    Choosing a regular basepoint presents such an object by a representation
    \(\varrho:\pi_1^{\mathrm{orb}}(\mathcal C,c)\to G(\mathbb C)\), with a
    change of presentation acting by conjugation.
    No filtration is included in a Betti object.  The local monodromy
    conditions are imposed representation by representation.  At orbifold
    points, all associated local systems have
    finite semisimple monodromy of the prescribed orbifold type.  At
    logarithmic points, the local monodromy has prescribed semisimple
    part and an allowed unipotent factor.  Writing the Jordan decomposition as
    \(T_x=T_{x,s}T_{x,u}\), we set
    \[
    N_x^{\mathrm B}:=\log T_{x,u}.
    \]
    The nontrivial-unipotent stratum is \(N_x^{\mathrm B}\ne0\); it is visible
    in every faithful associated representation as in
    Proposition~\ref{prop:associated-representations}.  We do not identify
    \(N_x^{\mathrm B}\) with the strictly filtration-raising Dolbeault or de
    Rham residue term.

    The Dolbeault groupoid used in the realization theorem is the full
    polystable degree-zero subgroupoid
    \[
    \Higgs_G(\mathcal C,\boldsymbol\theta)_{\mathrm{poly},0}.
    \]
    Stability is measured by the parahoric Ramanathan condition using
    Definition~\ref{def:principal-parahoric-degree}.
    With the inverse-character sign convention fixed there, if \(P_Q\) is a parahoric reduction of \(P\) to a parabolic subgroup \(Q\subset G\), and if \(\chi\in X^*(Q)\) is a strictly antidominant character trivial on the center of \(G\), then semistability requires
    \[
    \operatorname{pdeg}_{\boldsymbol\theta}(P_Q,\chi)\ge 0
    \]
    for every Higgs-invariant reduction \(P_Q\). Stability means strict inequality for every nontrivial such reduction, and polystability means that equality cases are induced from compatible Levi reductions. 
    Degree zero means
    \(
    \operatorname{pdeg}_{\boldsymbol\theta}(P,\chi)=0
    \)
    for every \(\chi\in X^*(G)\).

    We use the standard principal-bundle stability theory of Ramanathan type, together with the corresponding analytic Einstein--Hermitian results for principal and parabolic principal Higgs bundles \cite{RamananRamanathan84InstabilityFlag, AnchoucheBiswas01EinsteinHermitian, BiswasStemmler11HermitianEinsteinParabolic}.
    On the de Rham side, we write
    \[
    \MIC_G(\mathcal C,\boldsymbol\delta)_{\mathrm{red}}
    \]
    for the full reductive subgroupoid. A parahoric logarithmic flat principal \(G\)-bundle is called reductive if its associated logarithmic flat vector bundle is semisimple for every algebraic representation of \(G\).
    On the Betti side, we write
    \[
    \Rep_G^{\mathrm{red}}
    \bigl(\pi_1^{\mathrm{orb}}(\mathcal C),\boldsymbol n\bigr)
    \]
    for the full subgroupoid of reductive representations \(\varrho\colon\pi_1^{\mathrm{orb}}(\mathcal C,c)\longrightarrow G(\mathbb C)\).
    The superscript refers only to global reductivity, not to local monodromy or
    to a filtered structure.  Here, reductivity means that for every algebraic
    representation \(\rho\colon G\longrightarrow \GL(V)\), the composition
    \(\rho\circ\varrho\) is semisimple as a representation of the global
    orbifold fundamental group.
    This global condition does not require the local monodromy at logarithmic
    points to be semisimple: the adjusted logarithmic sector allows unipotent
    cusp monodromy, and its nontrivial-unipotent stratum is detected by
    \(N_x^{\mathrm B}\ne0\).

    In practice, reductivity on both the de Rham and Betti sides can be tested using a single faithful representation:

    \begin{proposition}[Faithful representation criterion for reductivity]
    \label{prop:faithful-representation-reductivity}
    Let \(\iota\colon G\hookrightarrow \GL(W)\) be a faithful finite-dimensional algebraic representation. 
    \begin{enumerate}
        \item A Betti representation \(\varrho\colon\pi_1^{\mathrm{orb}}(\mathcal C,c)\longrightarrow G(\mathbb C)\) is reductive if and only if the composition \(\iota\circ\varrho\) is semisimple.
        \item A parahoric logarithmic flat principal \(G\)-bundle \((P,\nabla)\) is reductive if and only if its \(\iota\)-associated logarithmic flat vector bundle \((P_\iota,\nabla_\iota)\) is semisimple.
    \end{enumerate}
    \end{proposition}

    \begin{proof}
    In both cases, the forward direction is immediate by definition. Conversely, since \(G\) is a reductive group, every finite-dimensional algebraic representation of \(G\) occurs as a direct summand of a tensor construction in \(W\) and \(W^\vee\). Because tensor products, duals, and direct summands of semisimple local systems (respectively, semisimple flat vector bundles) remain semisimple, the semisimplicity of the \(\iota\)-associated object forces semisimplicity for all algebraic representations.
    \end{proof}

    We now present the Tannakian reconstruction statement in the three
    realizations simultaneously.

    \begin{proposition}
    \label{prop:tannakian-reconstruction-principal-realizations}
    Let \(\rho_V:G\to\GL(V)\) denote the representation associated with
    \(V\in\Rep_{\mathbb C}(G)\).  There are natural equivalences of groupoids
    \[
    \Higgs_G(\mathcal C,\boldsymbol\theta)
    \simeq
    \mathrm{Fun}_{\otimes,\boldsymbol\theta}^{\mathrm{ex,faith}}
    \bigl(
    \Rep_{\mathbb C}(G),
    \Higgs(\mathcal C)
    \bigr),
    \]
    \[
    \MIC_G(\mathcal C,\boldsymbol\delta)
    \simeq
    \mathrm{Fun}_{\otimes,\boldsymbol\delta}^{\mathrm{ex,faith}}
    \bigl(
    \Rep_{\mathbb C}(G),
    \MIC(\mathcal C)
    \bigr),
    \]
    and
    \[
    \Rep_G(\pi_1^{\mathrm{orb}}(\mathcal C),\boldsymbol n)
    \simeq
    \mathrm{Fun}_{\otimes,\boldsymbol n}^{\mathrm{ex,faith}}
    \bigl(
    \Rep_{\mathbb C}(G),
    \Loc(\mathcal C)
    \bigr).
    \]
    Here the first two targets are the ambient varying-type categories and the
    third is the ordinary local-system category.  The subscript means that the
    value at \(V\) has local parameter \(\boldsymbol\theta_V\) or
    \(\boldsymbol\delta_V\), or local monodromy datum
    \(\boldsymbol n_V\), respectively.  The right-hand sides have tensor
    natural isomorphisms as morphisms.  Evaluation at \(V\) gives the associated
    vector-valued Higgs bundle, logarithmic flat bundle, or local system.
    \end{proposition}

    \begin{proof}
    For the Dolbeault realization, an exact faithful tensor functor
    \[
    F:\Rep_{\mathbb C}(G)\longrightarrow \Higgs(\mathcal C)
    \]
    has an underlying exact faithful tensor functor to parahoric vector bundles.
    By the Tannakian description of parahoric bundles, this reconstructs the
    underlying principal parahoric \(G\)-bundle.  The Higgs fields on all
    associated bundles are compatible with tensor products, duals, and
    morphisms in \(\Rep_{\mathbb C}(G)\), and therefore reconstruct a principal
    Higgs field.  The conditions
    \(\operatorname{type}(F(V))=\boldsymbol\theta_V\) recover the principal
    local type \(\boldsymbol\theta\).

    The de Rham case is identical, replacing Higgs fields by logarithmic
    connections.  Tensor compatibility of the associated logarithmic
    connections reconstructs the principal logarithmic connection, and the
    adjusted residue condition is checked representation by representation.
    Algebraic type is detected by any faithful member of the compatible tensor
    family.

    For the Betti realization, an exact faithful tensor functor
    \[
    \Rep_{\mathbb C}(G)\longrightarrow \Loc(\mathcal C)
    \]
    is the same as a \(G\)-local system on \(\mathcal C\), equivalently a
    representation
    \(
    \pi_1^{\mathrm{orb}}(\mathcal C,c)\longrightarrow G(\mathbb C),
    \)
    with local monodromy datum \(\boldsymbol n\) tested through
    \(\boldsymbol n_V\) in every algebraic representation.  No filtration
    is reconstructed on this side.  These
    constructions are inverse on objects and morphisms because isomorphisms
    of principal objects correspond to tensor natural isomorphisms of their
    associated functors.
    \end{proof}

    We finally record the compatibility with algebraic representations.  This will be used in the next subsection to state the principal realization theorem and its functoriality.

    \begin{proposition}
    \label{prop:associated-representations}
    Let
    \(
    \rho:G\longrightarrow \GL(V)
    \)
    be an algebraic representation.

    \begin{enumerate}
    \item If
    \(
    (P,\vartheta)\in \Higgs_G(\mathcal C,\boldsymbol\theta),
    \)
    then the associated Higgs bundle
    \[
    (P_\rho,\vartheta_\rho)
    :=
    (P\times^G V,\vartheta_\rho)
    \]
    is a vector-valued parahoric Higgs object of type
    \(\rho_*\boldsymbol\theta\).
    Adjustedness and algebraic type are preserved pointwise.  If the positive
    term at \(x\) is \(N_x\), the associated object is of log type precisely
    when
    \[
    d\rho(N_x)\ne0.
    \]

    \item If
    \(
    (P,\nabla)\in \MIC_G(\mathcal C,\boldsymbol\delta),
    \)
    then the associated logarithmic flat bundle
    \[
    (P_\rho,\nabla_\rho)
    :=
    (P\times^G V,\nabla_\rho)
    \]
    is a vector-valued parahoric de Rham object of type
    \(\rho_*\boldsymbol\delta\).
    Again, adjustedness and algebraic type are preserved pointwise.  If
    \(T_{x,u}\) is the unipotent factor of the local monodromy, log type is
    preserved precisely when
    \[
    \rho(T_{x,u})\ne1,
    \]
    equivalently when \(d\rho(\log T_{x,u})\ne0\).

    \item If
    \(
    \varrho\in
    \Rep_G(\pi_1^{\mathrm{orb}}(\mathcal C),\boldsymbol n),
    \)
    then
    \(
    \rho\circ\varrho
    \)
    is an ordinary vector-valued Betti object with induced local
    monodromy datum
    \(\rho_*\boldsymbol n\).
    If \(N_x^{\mathrm B}\) denotes the logarithm of the unipotent factor of the
    local monodromy, its image remains nonzero precisely when
    \(d\rho(N_x^{\mathrm B})\ne0\).
    \end{enumerate}

    In particular, in both the Higgs and de Rham realizations log type is
    preserved by every faithful representation, but need not be preserved by
    a nonfaithful one.  As \(\rho\) varies, these
    assignments are exact and symmetric monoidal into the corresponding
    ambient vector categories.
    \end{proposition}

    \begin{proof}
    The local assertions follow from functoriality of Moy--Prasad filtrations
    and residues under \(d\rho\).  If
    \[
    \operatorname{Res}_x(\nabla^\lambda)
    =\lambda\theta_x^{\mathrm{coch}}+N_x,
    \qquad
    N_x\in\mathfrak g_{>0}(\theta_x^{\mathrm{coch}}),
    \]
    then after applying \(\rho\) one obtains
    \[
    \operatorname{Res}_x(\nabla^{\lambda,V})
    =
    \lambda\,d\rho(\theta_x^{\mathrm{coch}})+d\rho(N_x),
    \]
    and
    \[
    d\rho(N_x)(F_{\ge a})\subset F_{>a}.
    \]
    Thus adjustedness is preserved.  For Higgs fields, functoriality under
    \(d\rho\) preserves algebraic type and gives the stated
    \(d\rho(N_x)\ne0\) criterion for log type.  For flat objects, local
    parahoric Riemann--Hilbert identifies the full gauge class with its
    monodromy pair.  A faithful representation preserves a nontrivial
    unipotent factor, whereas a nonfaithful one may annihilate it.  The Betti assertion
    follows directly from functoriality of Jordan decomposition and of the
    logarithm of the unipotent factor.
    Exactness and tensor compatibility are the standard identities
    \[
    P\times^G(V_1\oplus V_2)
    \simeq
    (P\times^G V_1)\oplus(P\times^G V_2),
    \]
    \[
    P\times^G(V_1\otimes V_2)
    \simeq
    (P\times^G V_1)\otimes(P\times^G V_2),
    \]
    together with the analogous identities for duals, internal Homs, and the
    tensor unit.  The Betti statement is the same assertion after composing
    representations.
    \end{proof}

    \begin{proposition}
    \label{prop:associated-representations-semisimplicity}
    Let
    \(
    \rho:G\longrightarrow \GL(V)
    \)
    be an algebraic representation.

    \begin{enumerate}
    \item If
    \(
    (P,\vartheta)\in
    \Higgs_G(\mathcal C,\boldsymbol\theta)_{\mathrm{poly},0},
    \)
    then the associated vector-valued parahoric Higgs bundle
    \(
    (P_\rho,\vartheta_\rho)
    \)
    is polystable of parahoric degree zero.

    \item If
    \(
    (P,\nabla)\in
    \MIC_G(\mathcal C,\boldsymbol\delta)_{\mathrm{red}},
    \)
    then
    \(
    (P_\rho,\nabla_\rho)
    \)
    is a semisimple vector-valued parahoric logarithmic flat bundle.

    \item If
    \(
    \varrho\in
    \Rep_G^{\mathrm{red}}
    (\pi_1^{\mathrm{orb}}(\mathcal C),\boldsymbol n),
    \)
    then
    \(
    \rho\circ\varrho
    \)
    is a semisimple vector-valued Betti object.
    \end{enumerate}
    \end{proposition}

    \begin{proof}
    For the first statement, use the Hermitian--Einstein characterization of
    polystability on the Higgs side of the tame parahoric principal
    non-abelian Hodge correspondence
    \cite[Theorem~6.2, pp.~34--36]{HuangGeorgiosSunZhao22TameparahoricnonabelianHodge}.
    A polystable degree-zero principal object admits a Hermitian--Einstein
    reduction with zero central
    term: the central term \(\tau\in\mathfrak z(\mathfrak g)\) is paired with
    the degrees of the character lines, and the differentials of characters
    span \(\mathfrak z(\mathfrak g)^*\), so degree zero for every character
    forces \(\tau=0\).  After choosing a maximal compact subgroup and an invariant
    Hermitian form on \(V\), the induced reduction satisfies the vector
    Hitchin--Simpson equation with zero central term.  Hence
    \((P_\rho,\vartheta_\rho)\) is polystable, and
    \[
    \operatorname{pdeg}_{\rho_*\boldsymbol\theta}(P_\rho)
    =
    \operatorname{pdeg}_{\boldsymbol\theta}(P,\det\rho)
    =0.
    \]

    For the remaining statements, let \(H\subset G\) be the Zariski closure
    of the global monodromy image.  Reductivity means that \(H\) is reductive,
    so every algebraic \(G\)-module restricts to a completely reducible
    \(H\)-module.  Thus \(\rho\circ\varrho\) is semisimple; the same argument,
    through monodromy, proves the de Rham assertion.
    \end{proof}

    The preceding reconstruction proposition is the formal bridge from the
    vector-valued correspondence to the principal correspondence.  The
    principal passage below is the paper's Tannakian assembly of the cited
    vector inputs, rather than a theorem imported verbatim from one source.  We
    apply the vector NAH/RH equivalences to all associated representations and
    reconstruct the resulting principal objects by
    Proposition~\ref{prop:tannakian-reconstruction-principal-realizations}.

\subsection{The realization theorem and its functoriality}
\label{subsec:realization-theorem-functoriality}

    \begin{theorem}[Tannakian log--orbi NAH/RH realization]
    \label{thm:tannakian-parahoric-realization}
    Let \(\mathcal C\) be a log--orbi curve, let \(G\) be a connected complex
    reductive group, and let
    \(\boldsymbol\eta=(\boldsymbol\theta,
    \boldsymbol\delta,\boldsymbol n)\)
    be a corresponding realization datum: the first two components are
    admissible rational full local parameters with zero transformed Betti
    weight, of orbifold algebraic type at \(D_{\mathrm{orb}}\) and adjusted
    type at \(D_{\log}\), and \(\boldsymbol n\) is the induced local
    monodromy datum.  Then there are
    natural equivalences of groupoids
    \[
    \Higgs_G(\mathcal C,\boldsymbol\theta)_{\mathrm{poly},0}
    \xrightarrow{\;\sim\;}
    \MIC_G(\mathcal C,\boldsymbol\delta)_{\mathrm{red}}
    \xrightarrow{\;\sim\;}
    \Rep_G^{\mathrm{red}}
    \bigl(\pi_1^{\mathrm{orb}}(\mathcal C),\boldsymbol n\bigr).
    \]
    For \(V\in\Rep_{\mathbb C}(G)\), let
    \(\boldsymbol\eta_V\) be the induced corresponding realization datum defined
    above.  The equivalences are compatible with
    associated objects through functorial isomorphisms
    \[
    \bigl(\operatorname{NAH}_G(P,\vartheta)\bigr)_V
    \simeq
    \operatorname{NAH}_{\boldsymbol\eta_V}(P_V,\vartheta_V),
    \qquad
    \bigl(\operatorname{RH}_G(P,\nabla)\bigr)_V
    \simeq
    \operatorname{RH}_{\boldsymbol\eta_V}(P_V,\nabla_V),
    \]
    compatibly with the rigid tensor operations and morphisms in
    \(\Rep_{\mathbb C}(G)\).  The principal Dolbeault/de Rham parameters and
    Betti monodromy datum are recovered from the family
    \(\boldsymbol\eta_V\).  The Betti groupoid consists of ordinary
    representations, with no additional filtration.  At an orbifold point,
    orbifold algebraic-type (regular equivariant) data correspond to finite
    semisimple inertia; the
    coarse strongly parabolic residue may be nonzero.  At a logarithmic point,
    the full adjusted regular-singular locus is allowed.
    \end{theorem}

    \begin{proof}
    We prove the Dolbeault--de Rham equivalence first.  Let
    \[
    (P,\vartheta)\in
    \Higgs_G(\mathcal C,\boldsymbol\theta)_{\mathrm{poly},0}.
    \]
    For each \(V\in\Rep_{\mathbb C}(G)\),
    Proposition~\ref{prop:associated-representations-semisimplicity} gives
    \[
    (P_V,\vartheta_V)
    \in
    \Higgs(\mathcal C,\boldsymbol\theta_V)_{\mathrm{poly},0}.
    \]
    Define
    \[
    \mathscr F_{P,\vartheta}(V)
    :=
    \operatorname{NAH}_{\boldsymbol\eta_V}(P_V,\vartheta_V).
    \]
    By Proposition~\ref{prop:vector-tensor-compatibility}, this is an exact
    faithful tensor functor to the ambient varying-type de Rham category, with
    semisimple \(V\)-component of type \(\boldsymbol\delta_V\).
    Proposition~\ref{prop:tannakian-reconstruction-principal-realizations}
    reconstructs from it a reductive principal logarithmic \(G\)-connection
    of type \(\boldsymbol\delta\).

    A morphism \(f:(P,\vartheta)\to(P',\vartheta')\) induces the tensor natural
    transformation
    \[
    \bigl(\operatorname{NAH}_{\boldsymbol\eta_V}(f_V)\bigr)_V:
    \mathscr F_{P,\vartheta}\Longrightarrow
    \mathscr F_{P',\vartheta'}.
    \]
    Full faithfulness of Tannakian reconstruction therefore defines the
    principal NAH functor on morphisms.

    Conversely, applying the inverse vector NAH functor to every associated
    object of a reductive principal connection gives an exact faithful tensor
    functor to the ambient varying-type Dolbeault category.  Reconstruction
    produces a principal parahoric Higgs object.  Its Betti realization has
    zero weight and local monodromy datum \(\boldsymbol n\); flatness and
    zero weight make every character parabolic degree vanish.  In this case
    \cite[Proposition~7.7]{BiquardGarciaPradaMundet2020ParabolicHiggs}
    identifies the polystable parabolic-local-system locus with the reductive
    representation locus, and
    \cite[Theorem~6.6]{BiquardGarciaPradaMundet2020ParabolicHiggs} supplies one
    adapted principal harmonic reduction and its induced polystable principal
    parabolic Higgs object.  The transformed parameters in the Main Table and
    the analytic/parahoric comparison
    \cite[Main Table, pp.~3--4; Proposition~4.21 and pp.~25--29]
    {HuangGeorgiosSunZhao22TameparahoricnonabelianHodge}, together with
    Proposition~\ref{prop:log-orbi-parahoric-dictionary}, places it in
    Dolbeault type \(\boldsymbol\theta\).  The full
    \(\boldsymbol n\)-datum is retained.

    We use this principal reduction only to certify the reconstructed object's
    polystable degree-zero membership.  Its associated reductions give the
    vector NAH objects tensor-compatibly; Tannakian full faithfulness therefore
    identifies it with the reconstructed object.  The vector unit and counit
    assemble into principal unit and counit isomorphisms, whose triangle
    identities may be checked after applying
    every \(V\), where they are those of the vector equivalence.  Hence the
    two principal NAH functors are quasi-inverse.

    The de Rham--Betti equivalence is constructed from ordinary
    regular-singular Riemann--Hilbert and the tensor functor
    \[
    V\longmapsto\operatorname{RH}_{\boldsymbol\eta_V}(P_V,\nabla_V)
    \]
    and its inverse.  Tensor natural transformations reconstruct morphisms,
    while the vector unit and counit reconstruct the principal quasi-inverse
    isomorphisms.  Reductivity agrees on the two sides because it is detected
    by semisimplicity of the associated global local systems.

    Finally, every \(V\)-component has the relevant Dolbeault/de Rham parameter
    or Betti monodromy datum from \(\boldsymbol\eta_V\); local Tannakian
    reconstruction therefore recovers \(\boldsymbol\eta\).
    Adjustedness is preserved representation by representation, while the
    algebraic stratum is detected by a faithful member of the compatible
    tensor family.  The complementary log stratum is preserved by faithful
    representations but not unconditionally by nonfaithful ones, as in
    Proposition~\ref{prop:associated-representations}.
    \end{proof}


    We next investigate the functorialities needed later for the canonical
    \(\PSL_2\)-object.  On the vector side, pullback for parabolic Higgs
    bundles and parabolic connections is compatible with the usual operations
    \cite[Lemma~3.2(1), Section~5.2, Theorems~5.6 and~7.1, and
    Remark~7.2]{AlfayaBiswas23pullbackanddirectimofparconandparHiggs}; effective
    descent follows from descent for the corresponding torsors and
    quasi-coherent morphisms \cite[Tags 04US and 03O6]{StacksProject}.
    Pullback local data are interpreted in the unnormalized periodic convention
    of Subsection~\ref{subsec:parahoric-bundles}; the proposition below records
    the resulting cocharacter, residue, and monodromy formulas.

    \begin{proposition}
    \label{prop:realization-finite-etale-pullback}
    Let \(f:\mathcal C_1\to\mathcal C_2\) be a finite \'etale morphism of
    log--orbi curves, and let
    \[
    \boldsymbol\eta
    =(\boldsymbol\theta,\boldsymbol\delta,\boldsymbol n)
    \]
    be a corresponding realization datum on \(\mathcal C_2\).  Its pullback is
    \[
    f^*\boldsymbol\eta
    =(f^*\boldsymbol\theta,f^*\boldsymbol\delta,
      f^*\boldsymbol n).
    \]
    At a point \(y\mapsto x\), the Dolbeault/de Rham cocharacter components
    scale by \(e_y\) in the periodic convention, residues have their natural
    scaled pullback, and the Betti datum pulls back by
    \[
    T_y=T_x^{e_y}.
    \]
    There are natural pullback functors
    \[
    f^*:\Higgs_G(\mathcal C_2,\boldsymbol\theta)
    \longrightarrow
    \Higgs_G(\mathcal C_1,f^*\boldsymbol\theta),
    \]
    \[
    f^*:\MIC_G(\mathcal C_2,\boldsymbol\delta)
    \longrightarrow
    \MIC_G(\mathcal C_1,f^*\boldsymbol\delta),
    \]
    and
    \[
    f^*:\Rep_G(\pi_1^{\mathrm{orb}}(\mathcal C_2),\boldsymbol n)
    \longrightarrow
    \Rep_G(\pi_1^{\mathrm{orb}}(\mathcal C_1),f^*\boldsymbol n),
    \]
    where, after choosing compatible base points, the Betti pullback is
    precomposition with
    \[
    f_*:\pi_1^{\mathrm{orb}}(\mathcal C_1)
    \longrightarrow
    \pi_1^{\mathrm{orb}}(\mathcal C_2).
    \]
    The realization equivalences commute with these pullback functors:
    \[
    f^*\circ \operatorname{NAH}_{\boldsymbol\eta}
    \simeq
    \operatorname{NAH}_{f^*\boldsymbol\eta}\circ f^*,
    \qquad
    f^*\circ \operatorname{RH}_{\boldsymbol\eta}
    \simeq
    \operatorname{RH}_{f^*\boldsymbol\eta}\circ f^*.
    \]
    They preserve degree zero, polystability, reductivity, the Dolbeault/de
    Rham local parameters, and the Betti local monodromy datum.
    \end{proposition}

    \begin{proof}
    The type formula is the pullback of the periodic lattice filtration.  In
    particular, for the canonical \(\PSL_2\)-type one has
    \[
    e_y\kappa_x-\kappa_y=e_y-1\in\mathbb Z;
    \]
    the associated integral modification is the one occurring under
    \(f^*\omega_{\mathcal C_2}\simeq\omega_{\mathcal C_1}\), in agreement
    with Proposition~\ref{prop:maximal-psl2-functoriality}.

    For an algebraic representation \(\rho:G\to\GL(V)\), after choosing a
    representative \(\delta_x^{\mathrm{coch}}\) of the
    cocharacter component, write
    \[
    \operatorname{Res}_x(\nabla^{\lambda,V})
    =
    \lambda\,d\rho(\delta_x^{\mathrm{coch}})+N_{V,x},
    \qquad
    N_{V,x}(F_{\ge a})\subset F_{>a}.
    \]
    The residue pullback formula gives
    \[
    \operatorname{Res}_y(f^*\nabla^{\lambda,V})
    =
    \lambda\,d\rho(e_y\delta_x^{\mathrm{coch}})+e_yN_{V,x},
    \]
    and \(e_yN_{V,x}\) still strictly raises the pulled-back filtration.
    For Higgs residues, this scaling preserves vanishing and nonvanishing
    because \(e_y\ne0\).  For flat objects, the
    local parahoric Riemann--Hilbert description sends the unipotent factor to
    its \(e_y\)-th power, which is nontrivial whenever the original factor is
    nontrivial in characteristic zero.  Thus adjusted, algebraic, and log type
    are preserved at the principal level.

    If \(E_Q\) is a parahoric reduction and \(\chi\in X^*(Q)\), then
    \[
    \operatorname{pdeg}_{f^*\boldsymbol\theta}(f^*E_Q,\chi)
    =
    \deg(f)\operatorname{pdeg}_{\boldsymbol\theta}(E_Q,\chi).
    \]
    Here \(\boldsymbol\theta\) is the Dolbeault component of
    \(\boldsymbol\eta\).  Moreover, the
    pullback of a zero-central-term Hermitian--Einstein reduction is again such
    a reduction; the same Hermitian--Einstein characterization on the Higgs
    side of the tame parahoric principal non-abelian Hodge correspondence
    therefore gives preservation of polystability and degree zero.
    On the Betti side, \(f_*\pi_1^{\mathrm{orb}}(\mathcal C_1)\) has finite
    index in \(\pi_1^{\mathrm{orb}}(\mathcal C_2)\), so the two monodromy
    Zariski closures have the same identity component; reductivity is
    preserved.

    Finally, pullback commutes with associated representations:
    \[
    f^*(P_\rho,\vartheta_\rho)
    \simeq
    \bigl((f^*P)_\rho,(f^*\vartheta)_\rho\bigr),
    \]
    and similarly for logarithmic connections.  
    Applying vector functoriality representation by representation and then
    reconstructing gives the NAH comparison.  Naturality of horizontal
    sections gives the RH comparison.  Pullback naturality makes both
    comparisons compatible with identity maps and composition.
    \end{proof}
    \begin{proposition}
    \label{prop:realization-galois-descent}
    Let \(\pi:\widetilde{\mathcal C}\to\mathcal C\) be a finite Galois
    \'etale cover of log--orbi curves with deck group \(\Delta\), and let
    \(\boldsymbol\eta\) be a corresponding realization datum on
    \(\mathcal C\).  For any of the groupoids below, write
    \(\operatorname{Desc}_\Delta(-)\) for the groupoid of objects equipped
    with isomorphisms \(\epsilon_\delta:\delta^*A\to A\) satisfying
    \(\epsilon_{\delta\delta'}=
    \epsilon_{\delta'}\circ\delta'^*\epsilon_\delta\).  Then
    \[
    \Higgs_G(\mathcal C,\boldsymbol\theta)
    \simeq
    \operatorname{Desc}_\Delta\!
    \bigl(\Higgs_G(\widetilde{\mathcal C},
    \pi^*\boldsymbol\theta)\bigr),
    \]
    \[
    \MIC_G(\mathcal C,\boldsymbol\delta)
    \simeq
    \operatorname{Desc}_\Delta\!
    \bigl(\MIC_G(\widetilde{\mathcal C},
    \pi^*\boldsymbol\delta)\bigr),
    \]
    and
    \[
    \Rep_G(\pi_1^{\mathrm{orb}}(\mathcal C),\boldsymbol n)
    \simeq
    \operatorname{Desc}_\Delta\!\bigl(
    \Rep_G(\pi_1^{\mathrm{orb}}(\widetilde{\mathcal C}),
    \pi^*\boldsymbol n)\bigr).
    \]
    Here \(\pi^*\boldsymbol n\) is the restricted local monodromy datum,
    equivalently obtained by taking the appropriate local powers.
    In the last display, descent is intrinsically descent of \(G\)-local
    systems; a presentation by based representations requires choices of a
    base point, its lifts, and connecting paths.  The NAH and RH equivalences
    identify these effective descent data, and the equivalences restrict to
    the polystable and reductive subgroupoids.
    \end{proposition}

    \begin{proof}
    Principal torsors and their Higgs fields or connections descend
    effectively on the relevant lisse--\'etale/Kummer charts; the field and
    connection conditions descend because they are equalities of morphisms.
    Finite topological covering descent gives the corresponding statement for
    \(G\)-local systems.  Choices used to express it by representations alter
    the deck action by inner automorphisms only.

    Proposition~\ref{prop:realization-finite-etale-pullback} shows that NAH
    and RH carry the \(\epsilon_\delta\) to descent isomorphisms and preserve
    their cocycle.  Equivalently, this may be checked in every algebraic
    representation and reconstructed by
    Proposition~\ref{prop:tannakian-reconstruction-principal-realizations}.
    Finally, vector polystability is equivalent under finite pullback by
    \cite[Theorem~5.6]{AlfayaBiswas23pullbackanddirectimofparconandparHiggs};
    applying this to associated objects and using the Hermitian--Einstein
    characterization on the Higgs side of the tame parahoric principal
    non-abelian Hodge correspondence gives the principal assertion.  Reductivity is
    detected by the identity component of the monodromy Zariski closure.
    Hence the stated restrictions also descend.
    \end{proof}


    We next record the compatibility with extension of structure group.

    \begin{proposition}
    \label{prop:realization-change-structure-group}
    Let
    \(
    \varphi:G\longrightarrow H
    \)
    be a homomorphism of connected complex reductive groups.  The corresponding
    realization datum \(\boldsymbol\eta\) for \(G\) induces
    \(\varphi_*\boldsymbol\eta\) for \(H\): the Dolbeault/de Rham parameters
    push forward in full and a Betti class \([T_x]\) maps to
    \([\varphi(T_x)]\).  Extension of structure group
    gives functors
    \[
    \varphi_*:\Higgs_G(\mathcal C,\boldsymbol\theta)
    \longrightarrow
    \Higgs_H(\mathcal C,\varphi_*\boldsymbol\theta),
    \]
    \[
    \varphi_*:\MIC_G(\mathcal C,\boldsymbol\delta)
    \longrightarrow
    \MIC_H(\mathcal C,\varphi_*\boldsymbol\delta),
    \]
    and
    \[
    \varphi_*:
    \Rep_G(\pi_1^{\mathrm{orb}}(\mathcal C),\boldsymbol n)
    \longrightarrow
    \Rep_H(\pi_1^{\mathrm{orb}}(\mathcal C),
    \varphi_*\boldsymbol n).
    \]
    In the Higgs or de Rham realization, write \(\boldsymbol\xi\) for the
    relevant full component.  If an adjusted residue at \(x\) is written
    \[
    \lambda\xi_x^{\mathrm{coch}}+N_x,
    \qquad N_x\in\mathfrak g_{>0}(\xi_x^{\mathrm{coch}}),
    \]
    then its pushout residue is
    \[
    \lambda\varphi_*(\xi_x^{\mathrm{coch}})+d\varphi(N_x).
    \]
    Thus adjustedness and algebraic type are preserved.  For a Higgs field,
    log type at \(x\) is preserved precisely when \(d\varphi(N_x)\ne0\); for a
    flat object, it is preserved precisely when the unipotent local monodromy
    remains nontrivial after applying \(\varphi\).  This holds in particular
    for a closed immersion, but need not hold for an arbitrary \(\varphi\).
    For arbitrary \(\varphi\), extension of structure group preserves the
    polystable degree-zero Dolbeault groupoid and the reductive de Rham and
    Betti groupoids.  On these realization groupoids, NAH and RH commute with
    \(\varphi_*\).  For a nonzero Hitchin--Simpson central term \(\tau\), the
    same argument only gives polystability when
    \(d\varphi(\tau)\in\mathfrak z(\mathfrak h)\).
    \end{proposition}

    \begin{proof}
    Functoriality of the Moy--Prasad grading gives
    \[
    d\varphi\bigl(\mathfrak g_{>0}(\theta_x^{\mathrm{coch}})\bigr)
    \subset
    \mathfrak h_{>0}(\varphi_*\theta_x^{\mathrm{coch}}).
    \]
    This proves the residue formula and adjustedness.  Functoriality under
    \(d\varphi\) preserves the zero Higgs residue term, while functoriality of
    Jordan decomposition preserves trivial unipotent local monodromy.  This
    gives the algebraic and log-type assertions above.

    If \((P,\vartheta)\) is polystable of degree zero, then for every character
    \(\chi\in X^*(H)\),
    \[
    \operatorname{pdeg}_{\varphi_*\boldsymbol\theta}
    \bigl((\varphi_*P)_\chi\bigr)
    =
    \operatorname{pdeg}_{\boldsymbol\theta}
    \bigl(P_{\chi\circ\varphi}\bigr)
    =0.
    \]
    By the same character-pairing argument as in
    Proposition~\ref{prop:associated-representations-semisimplicity}, its
    Hermitian--Einstein reduction has zero central term.  Extension to \(H\)
    again solves the zero-central-term equation; hence the pushout is
    polystable of degree zero
    \cite{AnchoucheBiswas01EinsteinHermitian,
    BiswasStemmler11HermitianEinsteinParabolic}.  With central term \(\tau\),
    the induced equation has term \(d\varphi(\tau)\), giving the stated
    qualification.  On the Betti side, the algebraic image of a reductive
    Zariski closure is reductive; the de Rham assertion follows through global
    monodromy.

    Finally, for every representation \(\sigma:H\to\GL(W)\), there is a
    canonical identification
    \[
    (\varphi_*P)_\sigma\simeq P_{\sigma\circ\varphi}.
    \]
    The vector NAH and RH functors are natural under this identification, and
    the resulting comparisons are tensor natural in \(\sigma\).  Tannakian
    reconstruction gives the asserted principal compatibilities.
    \end{proof}

    \paragraph{Canonical realization datum.}
    \label{par:canonical-realization-domain}
    We finally specialize to the object of
    Theorem~\ref{thm:canonical-maximal-psl2-higgs}.  Let \(\mathcal C\) be a
    hyperbolic log--orbi curve and let
    \[
    \theta_{\mathcal C}
    =
    (\theta_{\mathcal C,x})_{x\in D_{\mathrm{orb}}\cup D_{\log}},
    \qquad
    \theta_{\mathcal C,x}
    =
    \kappa_x\varpi^\vee,
    \]
    be its cocharacter system.  The full Dolbeault parameter also includes
    the local Higgs residue computed in that theorem.  In every associated
    representation, write
    \(
    \varphi_{\alpha,x}=s_{\alpha,x}+Y_{\alpha,x}
    \)
    for the Jordan decomposition of the Levi factor of that residue.  At an
    orbifold point the nonzero coarse residue strictly raises the filtration,
    so its Levi image is zero; at a logarithmic point its Levi residue is
    nilpotent.  Hence \(s_{\alpha,x}=0\) in both cases.  The local parameter
    transformation
    \[
    \gamma_x=-(s_{\alpha,x}+\overline{s}_{\alpha,x})
    \]
    therefore gives \(\gamma_x=0\) at every marked point
    \cite[Main Table, p.~3; Proposition~4.21 and the clarification on
    pp.~27--28]{HuangGeorgiosSunZhao22TameparahoricnonabelianHodge}.  The same
    argument holds after every algebraic representation.  We may consequently
    let
    \[
    \boldsymbol\eta_{\mathcal C}
    =
    (\boldsymbol\theta_{\mathcal C},
     \boldsymbol\delta_{\mathcal C},
     \boldsymbol n_{\mathcal C})
    \]
    be the corresponding zero-Betti-weight realization datum: the cocharacter
    component of \(\boldsymbol\theta_{\mathcal C}\) is
    \(\theta_{\mathcal C}\), and its residue component is the one just
    computed.  The Betti component is the full monodromy datum produced by
    realization; it has finite semisimple inertia at orbifold points and
    carries no filtration.  Its logarithmic conjugacy classes are not
    identified here; their parabolic character is proved in
    Theorem~\ref{thm:uniformizing-representation}.

    On a
    square-root chart the maximal Higgs field has the unique invariant line
    \(L=\Theta^{-1}\).  Since \(L^2\simeq\omega_{\mathcal C}^{-1}\), its
    normalized parahoric degree is
    \[
    \operatorname{pdeg}(L)
    =-\frac12\deg\omega_{\mathcal C}<0.
    \]
    Equivalently, the primitive strictly antidominant character of the induced
    Borel reduction has character line
    \(L^{-2}\simeq\Theta^2\simeq\omega_{\mathcal C}\), of strictly positive
    parahoric degree, in agreement with our
    \(\operatorname{pdeg}\ge0\) convention.  The central \(\mu_2\)-twists
    preserve the projective invariant line and act trivially on \(L^{-2}\), so
    this reduction and its degree calculation descend to the canonical
    \(\PSL_2\)-object.  It is therefore stable; since
    \(X^*(\PSL_2)=0\), the principal degree-zero condition is automatic.  Thus
    \[
    (\mathcal U_{\mathcal C},\vartheta_{\mathcal C})
    \in
    \Higgs_{\PSL_2}(\mathcal C,
    \boldsymbol\theta_{\mathcal C})_{\mathrm{poly},0}.
    \]

\section{Categorical Uniformization}
\label{sec:categorical-uniformization}

    The previous sections constructed the realization package for principal
    log--orbi objects, with parahoric structures used only for their
    coarse-curve local data:
    \[
    \text{Dolbeault}
    \;\simeq\;
    \text{de Rham}
    \;\simeq\;
    \text{Betti}.
    \]
    We now apply this package to the canonical maximal
    \(\PSL_2\)-Higgs object
    \(
    (\mathcal U_{\mathcal C},\vartheta_{\mathcal C})
    \)
    constructed in Subsection~\ref{subsec:canonical-maximal-psl2}.  We first
    reconstruct its projective flat connection and Betti local system; their
    uniformizing interpretation is proved in
    Theorem~\ref{thm:uniformizing-representation}.

    The analytic input is the maximal-Higgs period-map argument in its compact,
    tame, and orbifold local forms
    \cite{Hitchin87TheselfdualequonRiemansurface,
    Simpson88constrctingVHSusingYMandapptouniformization,
    Sim90harmobdonnoncomcurves,BenBrian95orbiRSandYMHeuq}.

    The reconstruction below uses this intrinsic projective object together
    with its functorial Tannakian realizations.

\subsection{Betti realization of the maximal object}
\label{subsec:betti-realization-maximal-object}

    Let \(\mathcal C\) be a hyperbolic log--orbi curve, and retain the
    corresponding canonical realization datum
    \(\boldsymbol\eta_{\mathcal C}\) from
    Paragraph~\ref{par:canonical-realization-domain}, which proves
    \[
    (\mathcal U_{\mathcal C},\vartheta_{\mathcal C})
    \in
    \Higgs_{\PSL_2}(\mathcal C,
    \boldsymbol\theta_{\mathcal C})_{\mathrm{poly},0}.
    \]
    This subsection only constructs the de Rham and Betti realizations of this
    object and records their finite \'etale functoriality.  The reality,
    discreteness, faithfulness, and finite-covolume properties of the Betti
    representation are proved in the next subsection.

    \begin{proposition}
    \label{prop:Betti-realization-maximal-object}
    The canonical maximal \(\PSL_2\)-Higgs object
    \(
    (\mathcal U_{\mathcal C},\vartheta_{\mathcal C})
    \)
    has a de Rham realization
    \[
    (\mathcal P_{\mathcal C},\nabla_{\mathcal C})
    \in
    \MIC_{\PSL_2}(\mathcal C,
    \boldsymbol\delta_{\mathcal C})_{\mathrm{red}},
    \]
    and, after choosing \(c\in C^\circ\), a Betti realization
    \[
    \rho_{\mathcal C}:
    \pi_1^{\mathrm{orb}}(\mathcal C,c)
    \longrightarrow
    \PSL_2(\mathbb C),
    \]
    with local monodromy datum \(\boldsymbol n_{\mathcal C}\), well-defined up to
    \(\PSL_2(\mathbb C)\)-conjugacy.
    \end{proposition}

    \begin{proof}
    This is the specialization of
    Theorem~\ref{thm:tannakian-parahoric-realization} to \(G=\PSL_2\), the
    datum \(\boldsymbol\eta_{\mathcal C}\), and the stable degree-zero
    membership proved in
    Paragraph~\ref{par:canonical-realization-domain}.  The two realization
    functors give the stated de Rham and ordinary Betti objects; tensor
    naturality makes the unbased local system, and hence its conjugacy class,
    intrinsic.
    \end{proof}

    \begin{proposition}
    \label{prop:Betti-realization-functoriality}
    Let \(f:\mathcal C_1\to\mathcal C_2\) be a finite \'etale morphism of
    hyperbolic log--orbi curves, and let \(\mathbb V_{\mathcal C_i}\) denote
    the unbased principal \(\PSL_2\)-local system reconstructed above.  There
    is a natural isomorphism
    \[
    f^*\mathbb V_{\mathcal C_2}\simeq\mathbb V_{\mathcal C_1}.
    \]
    After choosing base points \(c_i\in C_i^\circ\) and a connecting
    path from \(f(c_1)\) to \(c_2\), the induced homomorphism \(f_*\) satisfies
    \[
    [\rho_{\mathcal C_1}]
    =
    [\rho_{\mathcal C_2}\circ f_*].
    \]
    \end{proposition}

    \begin{proof}
    Proposition~\ref{prop:maximal-psl2-functoriality} gives
    \[
    f^*(\mathcal U_{\mathcal C_2},\vartheta_{\mathcal C_2})
    \simeq
    (\mathcal U_{\mathcal C_1},\vartheta_{\mathcal C_1}).
    \]
    Applying the pullback comparison of
    Proposition~\ref{prop:realization-finite-etale-pullback} gives the unbased
    local-system isomorphism and hence the displayed conjugacy-class equality.
    Changing a base point or connecting path changes \(f_*\) by an inner
    automorphism, which conjugacy absorbs.  The coherence of the isomorphisms
    \(\alpha_f\) and of realization pullback shows that these identifications
    respect identities and composition.
    \end{proof}

    \subsection{Reality, discreteness, and finite covolume}
    \label{subsec:reality-discreteness-finite-covolume}

    We now identify the complex Betti representation constructed in
    Proposition~\ref{prop:Betti-realization-maximal-object} with the classical
    uniformizing representation.  This identification uses only the
    Hitchin--Simpson maximality and period-map argument; it does not use the
    categorical uniformization theorem, which will be proved only afterwards.

    \begin{theorem}[Uniformizing representation]
    \label{thm:uniformizing-representation}
    Let \(\mathcal C\) be a hyperbolic log--orbi curve, let
    \(\mathcal C^\circ\) be its open analytic orbifold obtained by removing the
    logarithmic points, and choose \(c\in C^\circ\).  A representative
    of the Betti conjugacy class
    \[
    \rho_{\mathcal C}:
    \pi_1^{\mathrm{orb}}(\mathcal C,c)
    \longrightarrow
    \PSL_2(\mathbb C)
    \]
    associated with the canonical maximal \(\PSL_2\)-Higgs object is conjugate
    into
    \[
    \PU(1,1)\simeq \PSL_2(\mathbb R).
    \]
    The representation is faithful, and its image is discrete and cofinite.
    After choosing such a conjugation and an orientation-preserving Cayley
    identification \(\mathbb D\simeq\mathbb H\), write
    \[
    \Gamma_{\mathcal C}
    :=
    \rho_{\mathcal C}
    \bigl(\pi_1^{\mathrm{orb}}(\mathcal C,c)\bigr)
    \subset \PSL_2(\mathbb R)
    \]
    for the resulting chosen subgroup.  Its \(\PSL_2(\mathbb R)\)-conjugacy
    class is independent of these choices and is the uniformizing cofinite
    Fuchsian lattice of \(\mathcal C\).  The associated projective period map
    identifies the analytic orbifold universal cover with the disk by an
    orientation-preserving biholomorphic isometry:
    \[
    \widetilde{\mathcal C^\circ}\xrightarrow{\sim}\mathbb D.
    \]
    \end{theorem}

    \begin{proof}
    The stable degree-zero membership proved in
    Paragraph~\ref{par:canonical-realization-domain} and
    Proposition~\ref{prop:Betti-realization-maximal-object}
    supply the tame harmonic and Betti realizations used below.  We use
    Simpson's construction of the polarized tame weight-one variation
    of Hodge structure and its projective period map
    \cite[\S8, especially the discussion preceding
    Proposition~8.2]{Simpson88constrctingVHSusingYMandapptouniformization}
    and the corresponding tame boundary analysis
    \cite[Main Theorem, \S7, and
    Theorem~8]{Sim90harmobdonnoncomcurves}.  On a local chart on which a square
    root \(\Theta^{\otimes2}\simeq\omega_{\mathcal C}\) is chosen, the canonical
    object has the standard polarized length-two lift of Hodge type
    \((1,0)+(0,1)\).  Two such square roots differ by a \(\mu_2\)-torsor, so the
    corresponding lifts differ by unitary central signs.  Uniqueness and
    tensor compatibility of the tame harmonic metrics identify the twisted
    local metrics; after projectivization the signs act trivially and preserve
    the polarization.  Consequently the projective polarization, the
    \(\PU(1,1)\)-reduction, and the local period maps descend and glue without
    choosing a global square root.

    Fix \(c\in C^\circ\), a lift to
    \(\widetilde{\mathcal C^\circ}\), and the corresponding representative of
    \(\rho_{\mathcal C}\).  The descended projective period map is then a
    \(\rho_{\mathcal C}\)-equivariant holomorphic map
    \[
    u:\widetilde{\mathcal C^\circ}\longrightarrow
    \mathbb D,
    \qquad
    \mathbb D\simeq \PU(1,1)/U(1).
    \]

    On every square-root chart the maximal lower-left Higgs entry is an
    isomorphism.  This assertion is invariant under the central signs above,
    and under the period-map correspondence that entry is the differential of
    \(u\).  The smooth interior curvature calculation therefore makes \(u\) a
    local isometry for the Poincar\'e metrics, normalized to have curvature
    \(-1\); see
    \cite[Example~1.5 and
    Corollary~4.23]{Hitchin87TheselfdualequonRiemansurface}.  The lower-left
    convention fixes the holomorphic, hence orientation-preserving, branch.

    It remains to check the metric at the two kinds of marked points.  At an
    orbifold point of order \(m\), the local-isometry calculation is made on
    the cyclic chart, where the metric is smooth, and then descends to the
    standard orbifold cone model of angle \(2\pi/m\); see
    \cite[Corollary~3.4 and
    Theorem~6.17]{BenBrian95orbiRSandYMHeuq}.  At a logarithmic point, the
    parabolic and tame local models have Poincar\'e-cusp asymptotics
    \[
    ds^2\asymp
    \frac{|dz|^2}{|z|^2(\log|z|)^2};
    \]
    see
    \cite[Lemma~2.2 and proof of
    Theorem~2.3(2)]{BiswasAresGastesiGovindarajan97ParabolicHiggsTeichmuller}
    and \cite[\S7 and Theorem~8]{Sim90harmobdonnoncomcurves}.  These citations
    supply the respective local models, not a single theorem for the mixed
    curve.  The metric \(u^*g_{\mathbb D}\) on the universal cover is
    deck-invariant and therefore descends to a curvature \(-1\) orbifold metric
    \(g_{\mathcal C}\) on \(\mathcal C^\circ\).  Since the orbifold and
    logarithmic neighborhoods are disjoint, the cited local models together
    with the compact core show that \(g_{\mathcal C}\) is complete and has
    finite area.

    The lifted metric on \(\widetilde{\mathcal C^\circ}\) is therefore complete,
    and \(u\) is a local isometry to the connected disk.  A local isometry
    from a complete connected Riemannian manifold to a connected target is a
    covering, since geodesics lift for all time; hence \(u\) is a covering map.
    Since \(\mathbb D\) is simply connected, \(u\) is an orientation-preserving
    biholomorphic isometry
    \[
    \widetilde{\mathcal C^\circ}\xrightarrow{\sim}\mathbb D.
    \]

    For every \(\gamma\in\pi_1^{\mathrm{orb}}(\mathcal C,c)\), equivariance gives
    \[
    u(\gamma\cdot z)=\rho_{\mathcal C}(\gamma)\,u(z).
    \]
    Hence the chosen representative preserves \(\mathbb D\).  Since \(u\) is
    injective, \(\rho_{\mathcal C}(\gamma)=1\) implies \(\gamma=1\), proving
    faithfulness.  Through \(u\), the image action is conjugate to the properly
    discontinuous deck action, so the image is discrete, and
    \[
    \rho_{\mathcal C}\bigl(\pi_1^{\mathrm{orb}}(\mathcal C,c)\bigr)
    \backslash\mathbb D
    \simeq
    \mathcal C^\circ.
    \]
    The finite area established above gives finite covolume.  After the chosen
    orientation-preserving Cayley identification, the image is the subgroup
    \(\Gamma_{\mathcal C}\subset\PSL_2(\mathbb R)\) in the statement.  Changing
    the basepoint, its lift, the representative of the Betti conjugacy class,
    or the Cayley identification changes this subgroup only by
    \(\PSL_2(\mathbb R)\)-conjugacy.
    \end{proof}

    \begin{remark}[Local monodromy]
    \label{rem:local-monodromy-uniformizing}
    The local monodromy agrees with the log--orbi structure.  At an orbifold
    point of order \(m\), faithfulness shows that the local generator maps to
    an elliptic element of exact order \(m\).  At a logarithmic point, a
    local loop around the point acts as a parabolic, hence projectively
    unipotent, element
    because it generates the stabilizer of a complete finite-area cusp.  Under
    the regular-singular local Riemann--Hilbert dictionary, this is the
    log-type stratum.  This conclusion comes from the preceding complete
    finite-area cusp argument; it is not inferred from adjustedness or from the
    nonzero strictly filtration-raising Dolbeault residue.
    \end{remark}

    The analytic input in Theorem~\ref{thm:uniformizing-representation} is the
    standard Hitchin--Simpson period-map argument for maximal Higgs bundles.
    Independently of this analytic step, the preceding construction produces
    the canonical principal \(\PSL_2\)-object without an \(\SL_2\)-lift, and
    the Tannakian realization transports it functorially to the de Rham and
    Betti sides.  Thus \(\mathcal C\) intrinsically determines a
    \(\PSL_2(\mathbb R)\)-conjugacy class of uniformizing lattices.  After
    choosing a base point, a Betti representative, and a Cayley identification,
    we write \(\Gamma_{\mathcal C}\) for one representative.  The next
    subsection makes compatible choices to construct a quasi-inverse to the
    compactified quotient functor.

    \subsection{The categorical uniformization theorem}
    \label{subsec:categorical-uniformization-theorem}

    For each hyperbolic log--orbi curve \(\mathcal C\), choose the
    orientation-preserving universal-cover identification supplied by
    Theorem~\ref{thm:uniformizing-representation}, and let
    \(\Gamma_{\mathcal C}\subset\PSL_2(\mathbb R)\) be the resulting deck
    group.  If \(f:\mathcal C_1\to\mathcal C_2\) is finite \'etale, a lift of
    its restriction to the open orbifolds is an element
    \(g_f\in\PSL_2(\mathbb R)\) satisfying
    \[
    g_f\Gamma_{\mathcal C_1}g_f^{-1}
    \subset\Gamma_{\mathcal C_2}
    \]
    with finite index.  A different lift left-multiplies \(g_f\) by an element
    of \(\Gamma_{\mathcal C_2}\), so \([g_f]\) is a well-defined morphism in
    \(\mathbf{FL}\).  Lifts of composites multiply; hence these choices define
    a functor
    \[
    U:\mathbf{HypLO}\longrightarrow \mathbf{FL}.
    \]

    \begin{theorem}[Categorical uniformization]
    \label{thm:categorical-uniformization}
    The compactified quotient functor
    \[
    Q_c:\mathbf{FL}\longrightarrow \mathbf{HypLO},
    \qquad
    \Gamma\longmapsto (\Gamma\backslash\mathbb H)^c,
    \]
    is an equivalence of categories.  
    After the preceding choices, a quasi-inverse is
    \[
    U:\mathbf{HypLO}\longrightarrow \mathbf{FL},
    \qquad
    \mathcal C\longmapsto \Gamma_{\mathcal C}.
    \]
    If
    \(
    f:\mathcal C_1\longrightarrow \mathcal C_2
    \)
    is finite \'etale and \(U(f)=[g_f]\), then
    \[
    g_f\Gamma_{\mathcal C_1}g_f^{-1}
    \subset\Gamma_{\mathcal C_2},
    \qquad
    \deg(f)
    =
    [\Gamma_{\mathcal C_2}:
    g_f\Gamma_{\mathcal C_1}g_f^{-1}].
    \]
    In particular, after identifying the source with its conjugate, the
    subgroup direction and degree formula are
    \(
    \Gamma_{\mathcal C_1}\hookrightarrow\Gamma_{\mathcal C_2}
    \)
    and
    \(
    \deg(f)=[\Gamma_{\mathcal C_2}:\Gamma_{\mathcal C_1}]
    \), respectively.
    \end{theorem}

    \begin{proof}
    If \(\mathcal C\in\mathbf{HypLO}\), then, by
    Theorem~\ref{thm:uniformizing-representation}, the period map identifies the
    open orbifold universal cover with \(\mathbb H\) and its deck group with
    \(\Gamma_{\mathcal C}\).  Hence
    \[
    Q_cU(\mathcal C)
    =
    (\Gamma_{\mathcal C}\backslash\mathbb H)^c
    \simeq
    \mathcal C.
    \]

    Conversely, for \(\Gamma\in\mathbf{FL}\), put
    \[
    \mathcal C=(\Gamma\backslash\mathbb H)^c.
    \]
    Theorem~\ref{thm:uniformizing-representation} identifies the Betti
    realization of the canonical maximal \(\PSL_2\)-object on \(\mathcal C\) with
    the classical uniformizing representation of \(\mathcal C\).  Therefore the
    reconstructed lattice \(\Gamma_{\mathcal C}\) is conjugate to \(\Gamma\),
    and the conjugating element represents an isomorphism in \(\mathbf{FL}\).

    For full faithfulness, a class
    \([g]\in\operatorname{Hom}_{\mathbf{FL}}(\Gamma_1,\Gamma_2)\) gives the
    quotient map \([z]\mapsto[gz]\), as in
    Definition~\ref{def:FL}.  Conversely, let
    \[
    f:(\Gamma_1\backslash\mathbb H)^c
    \longrightarrow
    (\Gamma_2\backslash\mathbb H)^c
    \]
    be finite \'etale.  On the open orbifolds,
    \(f\circ q_1:\mathbb H\to\Gamma_2\backslash\mathbb H\) lifts through
    \(q_2\) to an orientation-preserving automorphism
    \(g\in\PSL_2(\mathbb R)\).  Equivariance gives
    \(g\Gamma_1g^{-1}\subset\Gamma_2\), with finite index because \(f\) is
    finite.  Changing the lift left-multiplies \(g\) by a target deck
    transformation, so \(f\) determines the unique coset \(\Gamma_2g\).
    These constructions are inverse.  At a cusp the quotient covering has the
    chart \(t=u^e\), so its compactification is precisely a finite \'etale
    morphism in the log--orbi sense.  This proves full faithfulness.

    The same covering description gives
    \(\deg(f)=[\Gamma_2:g\Gamma_1g^{-1}]\).  Finally, changing the chosen
    uniformizing identification of each \(\mathcal C\) by
    \(a_{\mathcal C}\in\PSL_2(\mathbb R)\) conjugates its lattice; the classes
    \([a_{\mathcal C}]\) form a natural isomorphism between the resulting
    quasi-inverses.  Thus the equivalence is independent of the choices up to
    natural isomorphism.
    \end{proof}

    Thus classical hyperbolic uniformization is promoted from an existence
    theorem to a reconstruction theorem.  The conjugacy class underlying the
    quasi-inverse is reconstructed by the chain
    \[
    \mathcal C
    \longmapsto
    (\mathcal U_{\mathcal C},\vartheta_{\mathcal C})
    \longmapsto
    \rho_{\mathcal C}
    \longmapsto
    \Gamma_{\mathcal C}.
    \]

    The analytic--\'etale comparison has the following uniformizing
    deck-group form.

    \begin{theorem}
    \label{thm:galois-enhancement}
    Let \(\mathcal C\) be a hyperbolic log--orbi curve.  Choose a uniformizing
    universal orbifold covering
    \[
    q_{\mathcal C}:\mathbb H\longrightarrow\mathcal C^\circ
    \]
    and set
    \[
    \Gamma_{\mathcal C}
    :=
    \operatorname{Deck}(q_{\mathcal C})
    \subset
    \operatorname{Aut}^+(\mathbb H)
    =
    \PSL_2(\mathbb R).
    \]
    Then there is an equivalence of
    Galois categories
    \[
    \mathbf{FEt}(\mathcal C)
    \simeq
    \mathbf{FSet}^{\mathrm{cont}}_{\widehat{\Gamma}_{\mathcal C}}.
    \]
    Choose also a geometric base point \(\bar x\to C^\circ\), the corresponding
    analytic base point \(x\), and a lift
    \(\widetilde x\in q_{\mathcal C}^{-1}(x)\).  Then the equivalence matches
    the two fiber functors and induces an isomorphism of profinite groups
    \[
    \pi_1^{\mathrm{\acute et}}(\mathcal C,\bar x)
    \simeq
    \widehat{\Gamma}_{\mathcal C}.
    \]
    \end{theorem}

    \begin{proof}
    The lift \(\widetilde x\) identifies
    \(\pi_1^{\mathrm{orb}}(\mathcal C,x)\) with the deck group
    \(\Gamma_{\mathcal C}\).  Theorem~\ref{thm:analytic-etale-comparison}
    therefore identifies finite \(\Gamma_{\mathcal C}\)-sets with
    \(\mathbf{FEt}(\mathcal C)\).  Since every finite action factors through
    a finite quotient, these are precisely the finite continuous
    \(\widehat{\Gamma}_{\mathcal C}\)-sets.

    With the compatible base point and lift, this equivalence sends the
    geometric fiber functor to the forgetful functor on finite
    \(\widehat{\Gamma}_{\mathcal C}\)-sets, giving the stated group
    isomorphism.  Changing the lift or a connecting path changes that
    isomorphism by an inner automorphism; changing the chosen universal cover
    or its identification with \(\mathbb H\) conjugates the literal subgroup
    \(\Gamma_{\mathcal C}\subset\PSL_2(\mathbb R)\).
    \end{proof}

    \begin{remark}[Subgroup dictionary]
    \label{rem:subgroup-dictionary}
    Under the equivalence of Theorem~\ref{thm:galois-enhancement}, connected finite \'etale covers of \(\mathcal C\) correspond to conjugacy classes of finite-index subgroups
    \(
    H\le \Gamma_{\mathcal C}.
    \)
    The corresponding cover is
    \[
    (H\backslash\mathbb H)^c
    \longrightarrow
    (\Gamma_{\mathcal C}\backslash\mathbb H)^c
    \simeq
    \mathcal C.
    \]
    It is Galois precisely when \(H\triangleleft \Gamma_{\mathcal C}\), in which case the deck group is
    \(
    \Gamma_{\mathcal C}/H.
    \)
    \end{remark}

    \section{Examples and applications}
    \label{sec:examples-moduli}

    We record two standard illustrations of categorical uniformization: the
    identification of marked Teichmüller space with the relative Fuchsian
    component, and the Gauss--Schwarz equation for a hyperbolic triangle
    orbifold.  Neither supplies foundational input to the main theorem; the
    reconstruction identifies the canonical Betti object in each picture.

    \subsection{Teichm\"uller space and the Fuchsian component}
    \label{subsec:teichmuller-fuchsian-component}

    Fix a connected oriented topological log--orbifold surface
    \[
    \Sigma=
    \bigl(
    |\Sigma|;(p_1,n_1),\ldots,(p_r,n_r);q_1,\ldots,q_s
    \bigr)
    \]
    with negative orbifold Euler characteristic.  Thus \(\Sigma\) is of
    hyperbolic type.  Choose a regular base point
    \(s_0\in |\Sigma|\setminus\{p_1,\ldots,p_r,q_1,\ldots,q_s\}\), together with
    positively oriented loops \(c_i\) and \(d_j\), based at \(s_0\)
    along fixed paths, around \(p_i\) and \(q_j\), respectively.  Let
    \(\operatorname{Teich}(\Sigma)\) denote the marked Teichmüller space of
    complete finite-area hyperbolic orbifold structures on \(\Sigma\), with cone
    angle \(2\pi/n_i\) at \(p_i\) and cusps at \(q_j\), marked by
    orientation-preserving orbifold homeomorphisms that respect the labeled cone
    points and cusps.  We emphasize that this is the marked Teichmüller space,
    not the unmarked moduli space.

    Let
    \[
    X_{\mathrm{Fuch}}(\Sigma)
    \subset
    \operatorname{Hom}\bigl(\pi_1^{\mathrm{orb}}(\Sigma,s_0),\PSL_2(\mathbb R)\bigr)
    /\PSL_2(\mathbb R)
    \]
    be the orientation-compatible relative Fuchsian locus: its points are the
    conjugacy classes of discrete, faithful, cofinite representations \(\rho\)
    such that \(\rho(c_i)\) lies in the positive elliptic conjugacy class of angle
    \(2\pi/n_i\), \(\rho(d_j)\) lies in the positive parabolic conjugacy class
    represented by \(z\mapsto z+1\), and the marking of the quotient orbifold
    induced by \(\rho\) is orientation preserving.

    \begin{theorem}
    \label{thm:teichmuller-fuchsian-component}
    The holonomy construction defines a real-analytic diffeomorphism
    \[
    \operatorname{hol}:
    \operatorname{Teich}(\Sigma)
    \xrightarrow{\sim}
    X_{\mathrm{Fuch}}(\Sigma).
    \]
    \end{theorem}

    \begin{proof}
    A marked complete hyperbolic structure has an orientation-preserving
    developing pair; its holonomy factors through
    \(\pi_1^{\mathrm{orb}}(\Sigma,s_0)\), and changing the pair conjugates the
    representation.  Completeness and finite area make the holonomy discrete,
    faithful, and cofinite, while the oriented local models give the prescribed
    positive elliptic and parabolic classes.  Conversely, a point of
    \(X_{\mathrm{Fuch}}(\Sigma)\) yields the marked quotient by its image, so
    the two constructions are inverse.

    For \(r+s=0\), this is the classical real-analytic identification with the
    oriented Teichm\"uller component for \(\PSL_2(\mathbb R)\)
    \cite[Theorem~7.5 and the following paragraph]
    {Hitchin92LiegpsandTeichSp}.  For \(r+s>0\), the relative holonomy map is
    a local real-analytic diffeomorphism with the local conjugacy classes fixed
    \cite[Theorems~3.4 and~4.3; Corollary~4.4(1)]
    {Mondello2010PoissonConical}; its hypotheses hold because the cone angles
    are \(2\pi/n_i\leq\pi\) and the cusp angles are zero.  The objectwise
    quasi-inverse in Theorem~\ref{thm:categorical-uniformization}, with the
    marking retained as above, shows that \(\operatorname{hol}\) is bijective
    and hence a global real-analytic diffeomorphism.
    \end{proof}

    \begin{remark}
    \label{rem:teich-betti-realization}
    For the marked hyperbolic log--orbi curve \(\mathcal C\) corresponding to
    a point of \(\operatorname{Teich}(\Sigma)\), the Betti realization of the
    canonical maximal \(\PSL_2\)-object determines the holonomy class
    \[
    [\rho_{\mathcal C}]\in X_{\mathrm{Fuch}}(\Sigma).
    \]
    The pair \((\rho_{\mathcal C},\Gamma_{\mathcal C})\) is determined up to
    simultaneous \(\PSL_2(\mathbb R)\)-conjugacy.
    \end{remark}

    \subsection{Triangle orbifolds and hypergeometric uniformization}
    \label{subsec:triangle-orbifolds-hypergeometric}

    The simplest compact hyperbolic orbifold curves are triangle orbifolds,
    whose canonical uniformizing \(\PSL_2\)-opers admit hypergeometric
    expressions \cite[\S15.17(ii)]{DLMF15.17HypergeometricFunction}.  Using
    Theorem~\ref{thm:uniformizing-representation} and the rigidity of a
    three-point projective Fuchsian equation, the theorem below gives a
    geometric derivation of this hypergeometric uniformization, classically
    obtained by Schwarz via the Schwarz map.

    \begin{theorem}[Triangle orbifolds and Gauss--Schwarz uniformization]
    \label{thm:triangle-hypergeometric}
    Let
    \(
    \mathcal X=\mathbb P^1(p,q,r)
    \)
    be a hyperbolic triangle orbifold.  
    Then the uniformizing \(\PSL_2\)-oper on \(\mathcal X\) is represented on
    \[
    \mathbb P^1\setminus\{0,1,\infty\}
    \]
    by a Gauss hypergeometric equation whose local exponent differences are
    \(
    \left(\frac1p,\frac1q,\frac1r\right).
    \)
    Its projective monodromy is the hyperbolic triangle group
    \[
    \Gamma_{p,q,r}
    \cong
    \langle A,B,C\mid A^p=B^q=C^r=ABC=1\rangle
    \subset \PSL_2(\mathbb R).
    \]
    More precisely, set
    \[
    a=\frac{1-\frac1p-\frac1q+\frac1r}{2},\qquad
    b=\frac{1-\frac1p-\frac1q-\frac1r}{2},\qquad
    c=1-\frac1p.
    \]
    The equation is
    \[
    z(1-z)y''+\bigl(c-(a+b+1)z\bigr)y'-aby=0.
    \]
    \end{theorem}

    \begin{proof}
    By Theorem~\ref{thm:uniformizing-representation}, the Betti realization of the canonical maximal \(\PSL_2\)-object on \(\mathcal X\) is the uniformizing Fuchsian representation.  
    The positively oriented generator around an orbifold point of order
    \(n\) maps to the positive elliptic class of angle \(2\pi/n\).  Thus its
    projective exponent difference is \(1/n\), rather than merely an
    unspecified unit modulo \(n\).  At \(0,1,\infty\) the differences are
    \[
    \left(\frac1p,\frac1q,\frac1r\right).
    \]

    A second-order projective Fuchsian equation with precisely three regular
    singularities is rigid once these differences are fixed and is, up to a
    rank-one gauge and projective equivalence, the Gauss equation displayed
    in the theorem.  Its exponent differences at \(0,1,\infty\) are,
    respectively,
    \[
    1-c=\frac1p,\qquad
    c-a-b=\frac1q,\qquad
    a-b=\frac1r,
    \]
    \cite[Sections~15.10(i) and~15.11(i)]{DLMF15.17HypergeometricFunction}.
    Hence these exponent differences give the uniformizing projective
    connection identified by
    Theorem~\ref{thm:uniformizing-representation}, and its monodromy is the
    uniformizing triangle lattice with the presentation stated in the theorem;
    this is the genus-zero three-orbifold-point case of
    \cite[\S2, p.~228]{Faltings83RealprojstronRiemannsur}.
    \end{proof}

    \begin{remark}[Accessory parameters beyond the rigid case]
    \label{rem:accessory-parameters}
    The triangle case is rigid.  For four or more marked points, or for
    positive-genus hyperbolic log--orbi curves, a scalar Fuchsian equation is
    not determined solely by its local exponent data; one then encounters the
    classical accessory-parameter problem.  In the present framework, these
    parameters are encoded by the Betti realization of the canonical maximal
    \(\PSL_2\)-object as the log--orbi curve varies in moduli.
    \end{remark}

    \section{Function fields and orbifold approximation}
    \label{sec:function-field-galois}

    We conclude with an application to the Galois theory of function fields of
    transcendence degree one over \(\mathbb C\).  
    Let
    \(
    F=\mathbb C(C)
    \)
    be the function field of a smooth projective connected complex curve \(C\),
    fix a separable closure \(F^{\mathrm{sep}}/F\), and let
    \[
    \bar\eta:\operatorname{Spec}F^{\mathrm{sep}}
    \longrightarrow \operatorname{Spec}F
    \]
    be the resulting geometric generic point.  Thus
    \[
    G_F:=\operatorname{Gal}(F^{\mathrm{sep}}/F)
    =\pi_1^{\mathrm{\acute et}}(\operatorname{Spec}F,\bar\eta).
    \]
    Finite extensions of \(F\) correspond to finite covers of \(C\), but these
    covers are generally ramified and therefore are not detected by the ordinary
    finite \'etale theory of \(C\).  The role of orbifold curves is to absorb
    ramification into inertia: after assigning suitable stabilizers at the branch
    points, a ramified cover becomes finite \'etale over an orbifold model of
    \(C\).

    The first part proves the based inverse-limit approximation of \(G_F\) in
    Theorem~\ref{thm:orbifold-approximation-Galois}.  The second then applies
    categorical uniformization separately to each hyperbolic stage, giving an
    objectwise Fuchsian lattice and profinite comparison.

    \subsection{Orbifold approximation of the absolute Galois group}
    \label{subsec:orbifold-approximation-galois}

    We use the standard effective root-stack models over \(C\).  For a
    finite-support order function
    \[
    \mathbf m=(m_x)_{x\in C},
    \qquad
    m_x\geq1,
    \qquad
    m_x=1\ \text{for all but finitely many }x,
    \]
    let \(\mathcal X_{\mathbf m}\) be the iterated root stack obtained by taking
    the \(m_x\)-th root along each divisor \(x\) with \(m_x>1\).  It has coarse
    space \(C\), trivial generic stabilizer, and stabilizer \(\mu_{m_x}\) at
    \(x\); see
    \cite[Definition~2.2.4 and Theorem~4.1]{Cadman2007UsingStacksTangency}.
    We define \(\operatorname{Orb}(C)\) to be the chosen skeletal poset of these
    models.  If \(\mathbf m\mid\mathbf m'\) pointwise, its unique refinement
    arrow is
    \[
    r_{\mathbf m',\mathbf m}:\mathcal X_{\mathbf m'}
    \longrightarrow \mathcal X_{\mathbf m}.
    \]
    Locally at \(x\), if \(m'_x=q m_x\), the induced inertia homomorphism is
    \[
    \mu_{m'_x}\longrightarrow\mu_{m_x},
    \qquad
    \zeta\longmapsto\zeta^q.
    \]
    It generally has nontrivial kernel, so a refinement arrow need not be
    representable and is not, in general, finite \'etale in the sense
    of Definition~\ref{def:etale-morphism-log-orbi}.  Nevertheless, finite
    \'etale covers may be pulled back along it, since representable finite
    \'etale morphisms are stable under arbitrary base change.

    \begin{proposition} \label{prop:OrbC-cofiltered} 
    The category \(\operatorname{Orb}(C)\) is cofiltered. 
    \end{proposition}

    \begin{proof}
    The constant function \(m_x=1\) gives a nonempty object.  Given
    \(\mathcal X_{\mathbf m}\) and \(\mathcal X_{\mathbf n}\), the pointwise
    least common multiple
    \[
    \ell_x=\operatorname{lcm}(m_x,n_x)
    \]
    has finite support and defines a common refinement
    \(\mathcal X_{\boldsymbol\ell}\).  Since the chosen category is a poset,
    there are no distinct parallel arrows to equalize.  Hence
    \(\operatorname{Orb}(C)\) is cofiltered.
    \end{proof}

    The basic mechanism is the following elementary resolution of ramification: ramification in a finite cover of \(C\) can be converted into orbifold inertia on the target and source.

    \begin{proposition}[Orbifold \'etale resolution of ramification]
    \label{prop:orbifold-etale-resolution} 
    Let \(f\colon Y\longrightarrow C\) be a finite morphism of smooth projective connected curves over \(\mathbb C\). Then there exist orbifold curves \(\mathcal Y\) and \(\mathcal X\), with coarse spaces \(Y\) and \(C\), and a finite \'etale morphism
    \[
    \bar f\colon\mathcal Y\longrightarrow \mathcal X
    \]
    whose underlying coarse map is \(f\).

    More explicitly, for a branch point \(x\in C\), one may choose
    \[
    m_x=\operatorname{lcm}\{e_y:y\in f^{-1}(x)\},
    \]
    where \(e_y\) is the ramification index of \(f\) at \(y\). 
    Then the point \(y\in f^{-1}(x)\) receives stabilizer order
    \(h_y=m_x/e_y\).
    Any common multiple of the ramification indices over \(x\) would also give a valid orbifold \'etale model; the least common multiple is the minimal convenient choice. 
    \end{proposition}

    \begin{proof}
    Put \(h_y=m_x/e_y\).  Near \(y\), choose coarse coordinates in which
    \(t=u^{e_y}\).  The corresponding map of root-stack charts is
    \[
    [\Delta_w/\mu_{h_y}]
    \longrightarrow
    [\Delta_z/\mu_{m_x}],
    \qquad
    z=w,
    \]
    with the injective inertia homomorphism
    \(\mu_{h_y}\hookrightarrow\mu_{m_x}\).  In coarse coordinates
    \[
    u=w^{h_y},
    \qquad
    t=z^{m_x}=u^{e_y},
    \]
    so this chart map lifts the given ramified map.  Its inertia maps are
    injective, hence it is representable.  After pullback to the target atlas
    \(\Delta_z\to[\Delta_z/\mu_{m_x}]\), it is the disjoint union of \(e_y\)
    isomorphisms.  It is therefore finite \'etale.

    Away from the branch locus no stabilizers are needed.  The universal
    property of the root construction gives the global representable lift and
    glues these local maps; see
    \cite[Proposition~3.3.3 and Theorem~3.3.6]{Cadman2007UsingStacksTangency}.
    The explicit chart calculation above shows that the resulting morphism is
    finite \'etale:
    \[
    \bar f\colon\mathcal Y\longrightarrow \mathcal X
    \]
    with coarse map \(f\).
    \end{proof}

    For each \(\mathcal X_{\mathbf m}\in\operatorname{Orb}(C)\), let
    \(j_{\mathbf m}:\operatorname{Spec}F\to\mathcal X_{\mathbf m}\) be the
    generic map.  Restriction gives
    \[
    j_{\mathbf m}^*:\mathbf{FEt}(\mathcal X_{\mathbf m})
    \longrightarrow \mathbf{FEt}(\operatorname{Spec}F).
    \]
    If \(\mathbf m\mid\mathbf m'\), pullback along the refinement gives
    \[
    r_{\mathbf m',\mathbf m}^*:
    \mathbf{FEt}(\mathcal X_{\mathbf m})
    \longrightarrow
    \mathbf{FEt}(\mathcal X_{\mathbf m'}),
    \qquad
    j_{\mathbf m'}^*r_{\mathbf m',\mathbf m}^*
    \simeq j_{\mathbf m}^*.
    \]
    The proof below shows that generic restriction at every fixed stage is
    fully faithful; consequently these pullback functors are fully faithful.

    \begin{theorem}[Orbifold presentation of finite extensions]
    \label{thm:orbifold-presentation-finite-extensions} 
    The compatible generic-restriction functors induce an equivalence
    \[
    \mathbf{FEt}(\operatorname{Spec}F)
    \simeq
    \operatorname*{2\text{-}colim}_{\mathcal X\in\operatorname{Orb}(C)^{\mathrm{op}}}
    \mathbf{FEt}(\mathcal X).
    \]
    The right-hand side is the filtered \(2\)-colimit along the fully faithful
    pullback functors \(r_{\mathbf m',\mathbf m}^*\).
    \end{theorem}

    \begin{proof}
    We first record fixed-stage full faithfulness.  Choose a normal \'etale
    atlas \(U\to\mathcal X_{\mathbf m}\).  For finite \'etale
    covers \(\mathcal P,\mathcal Q\) of \(\mathcal X_{\mathbf m}\), a morphism
    between their generic fibers extends uniquely over every normal component
    of \(U\) by
    \cite[Exp.~I, Proposition~10.1 and Corollaries~10.2--10.3]{SGA11971}.
    The two pullbacks of this extension to
    \(U\times_{\mathcal X_{\mathbf m}}U\) agree generically and hence agree by
    uniqueness.  Thus the extension descends uniquely to
    \(\mathcal X_{\mathbf m}\).  Therefore every \(j_{\mathbf m}^*\) is fully
    faithful.  It follows in turn that every refinement pullback is fully
    faithful: extend a morphism between two pulled-back covers first at the
    generic point and then uniquely at the original stage.

    Let \(A=\prod_{i=1}^a K_i\) be an arbitrary finite \'etale
    \(F\)-algebra, and let \(Y_i\to C\) be the normalization of \(C\) in
    \(K_i\).  At each branch point choose \(m_x\) to be a common multiple of all
    ramification indices occurring in all the covers \(Y_i\to C\), and put
    \(m_x=1\) elsewhere.  Proposition~\ref{prop:orbifold-etale-resolution}
    gives finite \'etale covers
    \[
    \mathcal Y_i\longrightarrow\mathcal X_{\mathbf m}.
    \]
    Their disjoint union has generic fiber \(\operatorname{Spec}A\), proving
    essential surjectivity.

    Finally, take objects represented over stages \(\mathcal X_{\mathbf m}\)
    and \(\mathcal X_{\mathbf n}\), together with an arbitrary morphism of their
    generic fibers.  Pull both objects back to the pointwise-lcm refinement
    \(\mathcal X_{\boldsymbol\ell}\).  Fixed-stage full faithfulness extends the
    generic morphism uniquely there.  The same uniqueness shows that two
    representatives define the same morphism in the filtered \(2\)-colimit
    exactly when their generic morphisms agree.  This proves full faithfulness
    and the asserted equivalence.
    \end{proof}

    \begin{theorem}[Orbifold approximation of the absolute Galois group]
    \label{thm:orbifold-approximation-Galois} 
    Relative to the fixed separable closure \(F^{\mathrm{sep}}/F\) and geometric
    generic point \(\bar\eta\), there is a canonical isomorphism of profinite
    groups
    \[
    G_F
    =\pi_1^{\mathrm{\acute et}}(\operatorname{Spec}F,\bar\eta)
    \xrightarrow{\sim}
    \varprojlim_{\mathcal X\in\operatorname{Orb}(C)}
    \pi_1^{\mathrm{\acute et}}(\mathcal X,\bar\eta).
    \]
    \end{theorem}

    \begin{proof}
    Let
    \[
    \omega_{\bar\eta}:\mathbf{FEt}(\operatorname{Spec}F)
    \longrightarrow\mathbf{FSet}
    \]
    be the geometric fiber functor.  At the stage \(\mathcal X_{\mathbf m}\), set
    \[
    \omega_{\mathbf m}:=\omega_{\bar\eta}\circ j_{\mathbf m}^*.
    \]
    The refinement compatibility of the generic maps gives
    \(\omega_{\mathbf m'}\circ r_{\mathbf m',\mathbf m}^*\simeq
    \omega_{\mathbf m}\), and hence a transition homomorphism
    \[
    r_{\mathbf m',\mathbf m,*}:
    \pi_1^{\mathrm{\acute et}}(\mathcal X_{\mathbf m'},\bar\eta)
    \longrightarrow
    \pi_1^{\mathrm{\acute et}}(\mathcal X_{\mathbf m},\bar\eta).
    \]

    By definition,
    \[
    G_F=\operatorname{Aut}(\omega_{\bar\eta}),
    \qquad
    \pi_1^{\mathrm{\acute et}}(\mathcal X_{\mathbf m},\bar\eta)
    =\operatorname{Aut}(\omega_{\mathbf m}).
    \]
    Restriction therefore gives a homomorphism
    \[
    \operatorname{Aut}(\omega_{\bar\eta})
    \longrightarrow
    \varprojlim_{\mathbf m}\operatorname{Aut}(\omega_{\mathbf m}).
    \]
    It is injective because every finite \'etale \(F\)-algebra is
    represented at some stage of the filtered \(2\)-colimit in
    Theorem~\ref{thm:orbifold-presentation-finite-extensions}.  Conversely, a
    compatible family of stage automorphisms acts on the geometric fiber of a
    finite generic object by choosing any stage at which that object is
    represented.  A pointwise-lcm
    refinement and full faithfulness show that this action is independent of
    the presentation and is natural for every generic morphism.  It therefore
    defines an element of \(\operatorname{Aut}(\omega_{\bar\eta})\), proving the
    displayed homomorphism is an isomorphism.

    The profinite topology on each automorphism group has a basis given by
    stabilizers of elements in finite objects.  Since every finite generic
    object occurs at a stage, the isomorphism identifies this basis with the
    inverse-limit topology.  The construction of fundamental groups from fiber
    functors and their profinite topology is standard Galois-category
    background; see \cite[Exp.~V, \S\S5--7]{SGA11971}.  The filtered
    \(2\)-colimit/inverse-limit passage itself is the direct argument above.
    Without the chosen geometric generic point, the resulting identification
    is canonical only up to inner automorphism.
    \end{proof}

    \begin{corollary}[Finite Galois quotients] \label{cor:finite-Galois-quotients}
    Let \(K/F\) be a finite Galois extension with group \(G=\operatorname{Gal}(K/F)\). Then there exists an orbifold model \(\mathcal X\in\operatorname{Orb}(C)\) and a connected finite \'etale Galois cover \(\mathcal Y\longrightarrow \mathcal X\) whose generic fiber is \(\operatorname{Spec}K\to\operatorname{Spec}F\).
    Choose an \(F\)-embedding \(K\hookrightarrow F^{\mathrm{sep}}\), and let
    \(\bar\eta_Y\) be the induced lift of \(\bar\eta\) to
    \(\mathcal Y\).

    Therefore there is an exact sequence
    \[
    1
    \longrightarrow
    \pi_1^{\mathrm{\acute et}}(\mathcal Y,\bar\eta_Y)
    \longrightarrow
    \pi_1^{\mathrm{\acute et}}(\mathcal X,\bar\eta)
    \longrightarrow
    G
    \longrightarrow
    1.
    \]
    \end{corollary}

    \begin{proof} 
    Let \(Y\to C\) be the normalization of \(C\) in \(K\). Choose an orbifold \'etale resolution \(\mathcal Y\longrightarrow \mathcal X\) as in Proposition~\ref{prop:orbifold-etale-resolution}. Since \(K/F\) is Galois, the group \(G\) acts on \(Y\) over \(C\). The assigned stabilizer order at a point of \(Y\) depends only on its ramification index over \(C\), and this index is constant along \(G\)-orbits. Hence the \(G\)-action preserves the orbifold structure on \(\mathcal Y\), and \(\mathcal Y\longrightarrow \mathcal X\) is a connected finite \'etale Galois cover with deck group \(G\). The exact sequence is the standard sequence of \'etale fundamental groups associated with a connected finite \'etale Galois cover.
    \end{proof}
    \subsection{Hyperbolic stages and objectwise uniformization}
    \label{subsec:hyperbolic-stages-galois}
    We now apply categorical uniformization objectwise to the hyperbolic
    models.  Let
    \[
    \operatorname{Orb}_{\mathrm{hyp}}(C)
    \subset
    \operatorname{Orb}(C)
    \]
    be the full subcategory consisting of hyperbolic orbifold models.  We use
    only its objects in the following comparison.

    \begin{proposition}[Objectwise hyperbolic-stage comparison]
    \label{prop:hyperbolic-stage-comparison}
    Let \(\mathcal X\in\operatorname{Orb}_{\mathrm{hyp}}(C)\).  Choose a
    uniformizing identification and its lattice
    \(\Gamma_{\mathcal X}\subset\PSL_2(\mathbb R)\), a regular analytic base
    point with a lift to \(\mathbb H\), and an \'etale path from the fixed
    geometric generic point \(\bar\eta\) to the corresponding geometric base
    point.  Then
    \[
    \pi_1^{\mathrm{\acute et}}(\mathcal X,\bar\eta)
    \simeq
    \widehat{\Gamma}_{\mathcal X}.
    \]
    Without these choices the isomorphism is defined only up to inner
    automorphism.
    \end{proposition}

    \begin{proof}
    Apply Theorem~\ref{thm:galois-enhancement} to the individual hyperbolic
    orbifold \(\mathcal X\), and transport its geometric fiber functor along
    the chosen \'etale path.
    \end{proof}

    This comparison is deliberately objectwise.  The root-stack refinement
    arrows in \(\operatorname{Orb}(C)\) are generally nonrepresentable and are
    not finite \'etale morphisms in \(\mathbf{HypLO}\); consequently the
    argument supplies neither finite-index transition maps between the literal
    lattices nor a Fuchsian-lattice-valued pro-system.  In particular, it does
    not define a representation \(G_F\to\PSL_2(\mathbb R)\).

    It remains to identify which orbifold models populate this hyperbolic sector. This membership is governed by the sign of the orbifold canonical degree and depends on the genus.
    We partition the full orbifold approximation into geometrically distinct sectors. For an orbifold model 
    \[
    \mathcal X=\bigl(C;(x_1,m_1),\ldots,(x_r,m_r)\bigr)
    \]
    over the fixed coarse curve \(C\), its orbifold canonical degree is
    \[
    \deg\omega_{\mathcal X}
    =
    2g(C)-2+
    \sum_{i=1}^r
    \left(1-\frac1{m_i}\right)
    =
    -\chi(\mathcal X).
    \]
    The sign of \(\deg\omega_{\mathcal X}\) yields the standard geometric trichotomy, stratifying \(\operatorname{Orb}(C)\) into three full subcategories: the \emph{spherical} sector \(\operatorname{Orb}_{\mathrm{sph}}(C)\) (\(<0\)), the \emph{Euclidean} sector \(\operatorname{Orb}_{\mathrm{euc}}(C)\) (\(=0\)), and the \emph{hyperbolic} sector \(\operatorname{Orb}_{\mathrm{hyp}}(C)\) (\(>0\)).  This is a partition of objects, not a coproduct decomposition of \(\operatorname{Orb}(C)\): a refinement may pass from one sector to another.

    \begin{proposition}[Sector decomposition]
    \label{prop:sector-decomposition}
    The presence of these sectors is strictly governed by the genus of the underlying coarse curve \(C\):
    \begin{enumerate}
    \item If \(g(C)\ge2\), then \(\operatorname{Orb}(C) = \operatorname{Orb}_{\mathrm{hyp}}(C)\). Every orbifold model over \(C\) is hyperbolic.
    \item If \(g(C)=1\), the trivial unramified model is Euclidean, while every nontrivial orbifold model is hyperbolic.
    \item If \(C=\mathbb P^1\), the spherical, Euclidean, and hyperbolic sectors are all nonempty.
    \end{enumerate}
    \end{proposition}

    \begin{proof}
    The trichotomy follows immediately from the degree formula. For \(g(C)\ge2\), the base degree \(2g(C)-2\) is strictly positive, so adding nonnegative orbifold contributions preserves hyperbolicity. For \(g(C)=1\), the base degree is zero, meaning the total degree is positive if and only if there is at least one stacky point. For \(C=\mathbb P^1\), the base degree is \(-2\); the trivial model is spherical, classical signatures such as \((2,3,6)\) or \((2,4,4)\) yield Euclidean models, and configurations with sufficiently many points or large enough orders produce hyperbolic models.
    \end{proof}

    For \(F=\mathbb C(t)\) all three sectors occur.  After a uniformizing
    choice, each
    \(\mathcal X=\mathbb P^1(m_1,\ldots,m_r)
    \in\operatorname{Orb}_{\mathrm{hyp}}(\mathbb P^1)\) has, up to conjugacy,
    the genus-zero cofinite Fuchsian lattice
    \[
    \Gamma_{\mathcal X}
    \cong
    \left\langle
    \gamma_1,\ldots,\gamma_r
    \ \middle|\
    \gamma_1^{m_1}=\cdots=\gamma_r^{m_r}=1,\ 
    \gamma_1\cdots\gamma_r=1
    \right\rangle.
    \]
    For each such \(\mathcal X\), Proposition~\ref{prop:hyperbolic-stage-comparison}
    identifies its based \'etale fundamental group with
    \(\widehat{\Gamma}_{\mathcal X}\), subject to the choices stated there.

\section*{Acknowledgments}
    The first author thanks Nianzi Li for patient explanations of local
    harmonic models and their relation to the non-abelian Hodge
    correspondence, and Jianping Wang for helpful discussions on parahoric
    structures.

\bibliographystyle{plain}
\bibliography{reference}

\end{document}